\documentclass[a4paper, 12pt]{amsart}
 \usepackage{amsmath, amsthm, amssymb, amsfonts}
\usepackage[utf8]{inputenc}
\usepackage[normalem]{ulem}
\usepackage[colorlinks = true,
            linkcolor = black,
            urlcolor  = black,
            citecolor = red,
            anchorcolor = red]{hyperref}
\usepackage{url}
\usepackage{longtable} 
\usepackage{enumitem}
\usepackage{relsize}
\usepackage{bm}
\usepackage{xcolor}
\usepackage{graphicx}
\usepackage{mathdots}
\usepackage{color}
\usepackage{tikz-cd,tikz}
\usepackage{amssymb}

\usetikzlibrary{arrows, decorations.markings}
\usetikzlibrary{positioning}
\tikzstyle{vecArrow} = [thick, decoration={markings,mark=at position
   1 with {\arrow[semithick]{open triangle 60}}},
   double distance=1.4pt, shorten >= 5.5pt,
   preaction = {decorate},
   postaction = {draw,line width=1.4pt, white,shorten >= 4.5pt}]
\tikzstyle{innerWhite} = [semithick, white,line width=1.4pt, shorten >= 4.5pt]
\tikzstyle{vecEq} = [thick,
   double distance=1.4pt,
   preaction = {decorate},
   postaction = {draw,line width=1.4pt, white,shorten >= 4.5pt}]
\usetikzlibrary{matrix}
\usetikzlibrary{shapes}
\usetikzlibrary{arrows,decorations.markings, tikzmark}
\usepackage{stmaryrd}

\pgfarrowsdeclare{bad to}{bad to}
{
  \pgfarrowsleftextend{-2\pgflinewidth}
  \pgfarrowsrightextend{\pgflinewidth}
}
{
  \pgfsetlinewidth{0.8\pgflinewidth}
  \pgfsetdash{}{0pt}
  \pgfsetroundcap
  \pgfsetroundjoin
  \pgfpathmoveto{\pgfpoint{-3\pgflinewidth}{4\pgflinewidth}}
  \pgfpathcurveto
  {\pgfpoint{-2.75\pgflinewidth}{2.5\pgflinewidth}}
  {\pgfpoint{0pt}{0.25\pgflinewidth}}
  {\pgfpoint{0.75\pgflinewidth}{0pt}}
  \pgfpathcurveto
  {\pgfpoint{0pt}{-0.25\pgflinewidth}}
  {\pgfpoint{-2.75\pgflinewidth}{-2.5\pgflinewidth}}
  {\pgfpoint{-3\pgflinewidth}{-4\pgflinewidth}}
  \pgfusepathqstroke
}

\def\M{\mathcal{M}}

\theoremstyle{definition}
\newtheorem{definition}{Definition}[section]
\newtheorem{theorem}[definition]{Theorem}
\newtheorem{example}[definition]{Example}
\newtheorem{proposition}[definition]{Proposition}

\newtheorem{corollary}[definition]{Corollary}

\newtheorem{lemma}[definition]{Lemma}
\newtheorem{conjecture}[definition]{Conjecture}

\newtheorem*{question}{Question}
\newtheorem{remark}[definition]{Remark}

\newtheorem{notation}[definition]{Notation}

\newcommand{\QQ}{\mathbb{Q}}

\newcommand{\BA}{{\mathbb{A}}}

\newcommand{\BC}{{\mathbb{C}}}
\newcommand{\BD}{{\mathbb{D}}}

\newcommand{\BF}{{\mathbb{F}}}

\newcommand{\BH}{{\mathbb{H}}}

\newcommand{\BL}{{\mathbb{L}}}

\newcommand{\BP}{{\mathbb{P}}}
\newcommand{\BQ}{{\mathbb{Q}}}
\newcommand{\BR}{{\mathbb{R}}}

\newcommand{\BZ}{{\mathbb{Z}}}

\newcommand{\CA}{{\mathcal A}}
\newcommand{\CB}{{\mathcal B}}
\newcommand{\CC}{{\mathcal C}}

\newcommand{\CE}{{\mathcal E}}
\newcommand{\CF}{{\mathcal F}}

\newcommand{\CH}{{\mathcal H}}

\newcommand{\CJ}{{\mathcal J}}

\newcommand{\CM}{{\mathcal M}}

\newcommand{\CO}{{\mathcal O}}

\newcommand{\CR}{{\mathcal R}}
\newcommand{\CS}{{\mathcal S}}
\newcommand{\CT}{{\mathcal T}}

\newcommand{\CW}{{\mathcal W}}
\newcommand{\CX}{{\mathcal X}}
\newcommand{\CY}{{\mathcal Y}}
\newcommand{\CZ}{{\mathcal Z}}
\newcommand{\CHom}{{\mathcal Hom}}

\DeclareMathOperator{\id}{id}
\DeclareMathOperator{\coker}{coker}

\newcommand{\Pic}{\mathop{\rm Pic}\nolimits}

\newcommand{\Ext}{{\rm Ext}}
\newcommand{\Li}{{\rm Li}}

\newcommand{\Image}{{\rm Im}}
\newcommand{\Kernel}{{\rm Ker}}

\newcommand{\pt}{{\mathsf{p}}}

\newcommand{\sfb}{{\mathsf{b}}}
\newcommand{\sfw}{{\mathsf{w}}}
\newcommand{\sfc}{{\mathsf{c}}}
\newcommand{\sff}{{\mathsf{f}}}
\newcommand{\congpf}{\xymatrix@1@=15pt{\ar[r]^-\sim&}}

\newcommand{\GW}{\mathrm{GW}}

\newcommand{\ch}{\mathrm{ch}}

\newcommand{\Coh}{\mathrm{Coh}}
\newcommand{\Per}{\mathrm{Per}}
\newcommand{\Hom}{\mathrm{Hom}}
\newcommand{\supp}{\mathrm{supp}}
\newcommand{\PT}{\mathrm{PT}}
\newcommand{\DT}{\mathrm{DT}}
\newcommand{\bs}{\mathrm{BS}}
\newcommand{\PPT}{{}^p\PT}
\newcommand{\PDT}{{}^p\DT}
\newcommand{\Pairs}{\mathrm{Pairs}}
\newcommand{\PPairs}{{}^p\mathrm{Pairs}}

\newcommand{\fbs}{\CF_{\textup{BS}}}
\newcommand{\tbs}{\CT_{\textup{BS}}}

\newcommand{\HH}{\mathrm{H}}
\newcommand{\rk}{\mathrm{rk}}
\newcommand{\len}{\mathrm{len}}
\newcommand{\twist}{\mathrm{t}_{\Phi}}
\newcommand{\cotwist}{\mathrm{cot}_{\Phi}}
\newcommand{\eff}{\mathrm{eff}}
\newcommand{\Nef}{\mathrm{Nef}}
\newcommand{\Amp}{\mathrm{Amp}}
\newcommand{\Obj}{\mathrm{Obj}}
\newcommand{\ex}{\mathrm{ex}}

\usepackage{tikz}
\usepackage{tikz-cd}
\usetikzlibrary{decorations.pathmorphing}
\AtBeginDocument{%
	\def\MR#1{}
}

\newcommand{\ssetminus}{\mathbin{\vcenter{\hbox{$\scriptscriptstyle\setminus$}}}}

\DeclareSymbolFont{extraup}{U}{zavm}{m}{n}
\DeclareMathSymbol{\varheart}{\mathalpha}{extraup}{86}
\DeclareMathSymbol{\vardiamond}{\mathalpha}{extraup}{87}

\setlist[enumerate]{label={(\roman*)}}

\begin{document}

\baselineskip=16pt
\parskip=5pt

\title[Weyl symmetry via spherical twists]{Weyl symmetry for curve counting invariants via spherical twists}
\author{Tim-Henrik Buelles}
\address {ETH Z\"urich, Department of Mathematics}
\email{buelles@math.ethz.ch}
\author{Miguel Moreira}
\address {ETH Z\"urich, Department of Mathematics}
\email{miguel.moreira@math.ethz.ch}

\date{\today}
\maketitle
\begin{abstract} 
We study the curve counting invariants of Calabi--Yau 3-folds via the Weyl reflection along a ruled divisor. We obtain a new rationality result and functional equation for the generating functions of Pandharipande--Thomas invariants. When the divisor arises as resolution of a curve of $A_1$-singularities, our results match the rationality of the associated Calabi--Yau orbifold.

The symmetry on generating functions descends from the action of an infinite dihedral group of derived auto-equivalences, which is generated by the derived dual and a spherical twist. Our techniques involve wall-crossing formulas and generalized DT invariants for surface-like objects.

\end{abstract}
\setcounter{tocdepth}{1}
\tableofcontents

\section{Introduction}

\subsection{Overview}\label{subsec: overview}

Let $X$ be a Calabi--Yau 3-fold containing a smooth geometrically ruled divisor~$W$. Physical considerations for BPS state counts~\cite{KMP96,KM96} suggest that the curve counting invariants of $X$ are constrained by this constellation. More precisely, let $\sfw\in H_4(X,\BZ)$ be the class of the divisor and $\sfb\in H_2(X,\BZ)$ be the class of the rational curve of the ruling. Consider the \emph{Weyl symmetry} on $H_2(X,\BZ)$ defined by
\begin{equation*}\beta \longmapsto \beta'=\beta + (\sfw\cdot\beta)\, \sfb\,.\end{equation*}
Since $\sfw\cdot \,\sfb=-2$, this defines a reflection. The guiding example is that of an elliptic Calabi--Yau 3-fold
\[ X \to \BP^1\times\BP^1\]
which is fibered in elliptic K3 surfaces over $\BP^1$, see Section~\ref{subsec: STU}. For K3 curve classes $\beta$, the Weyl symmetry $\beta\leftrightarrow\beta'$ is exactly realized on the level of Gopakumar--Vafa invariants
\[ n_{g,\beta}^{K3} = n_{g,\beta'}^{K3}\,.\]
The equality is reminiscent of the monodromy for quasi-polarized K3 surfaces. For arbitrary $\beta\in H_2(X,\BZ)$ such an equality cannot hold, for example when $\beta\in H_2(W,\BZ)$ in which case the invariants are given by the local surface~$K_W$. Instead, we find that the Weyl symmetry \emph{is realized as a functional equation}. This symmetry is analogous to the rationality and the $q\leftrightarrow q^{-1}$ invariance for generating series of Pandharipande--Thomas stable pairs invariants. 

The Pandharipande--Thomas (PT)~\cite{PT09} invariants~$\PT_{\beta, n}\in\BZ$ are curve counting invariants enumerating stable pairs in the derived category $D^b(X)$ with curve class $\beta$ and Euler characteristic $n\in\BZ$. Our results concern the $2$-variable generating series\footnote{We use the non-standard sign $-q$ which simplifies some formulas.}
\[ \PT_{\beta}(q,Q) = \sum_{n,j\in \BZ} \PT_{\beta+j\sfb, n} \, (-q)^n\, Q^j\,.\]
Let $\CE$ be a rank 2 bundle over a smooth projective curve~$C$ of genus~$g$ and $p\colon W=\BP_C(\CE)\to C$ be the corresponding $\BP^1$-bundle. We will assume that $X$ admits a nef class $A\in\Nef(X)$ which vanishes only on the extremal ray spanned by~$\sfb$, i.e.\footnote{We do not require the line bundle to be basepoint-free and we do not assume a contraction morphism. Such a nef class exists in many cases, e.g.\ for elliptic Calabi--Yau 3-folds.}
\begin{equation}\label{eq: nef}\tag{$\diamondsuit$}
\Kernel\left(N^\eff_1(X)\xrightarrow{A\cdot\,}\BZ\right)=\BZ_{\geq 0}\cdot \sfb\,.
\end{equation}
The generating series $\PT_0$ of multiples of $\sfb$ is easily computed as
\[\PT_0(q,Q)=\prod_{j\geq 1}(1-q^j Q )^{(2g-2)j}.\]
Our main result is:
\begin{theorem}\label{thm: main}
Let $X$ be a Calabi--Yau 3-fold containing a smooth divisor~$W$ satisfying condition~\eqref{eq: nef}. Then
\[ \frac{\PT_\beta(q,Q)}{\PT_0(q,Q)} \in \BQ(q,Q)\]
is the expansion of a rational function $f_\beta(q,Q)$ such that
\begin{align*}
    f_\beta(q^{-1},Q) &=f_\beta(q,Q)\,,\\
    f_\beta(q,Q^{-1})&=Q^{-\sfw\cdot\beta}\,f_\beta(q,Q)\,.
\end{align*}
\end{theorem}

The rationality in~$q$ and the invariance under $q\leftrightarrow q^{-1}$ are well-known~\cite{Br11,PT09,To10,To20}. The symmetry is induced by the action of the derived dual~$\BD^X$ on $D^b(X)$. Analogously, we introduce a derived anti-equivalence~$\rho$ of order two, which promotes the Weyl reflection to the derived category and induces the second functional equation on generating functions. It is defined as 
\[ \rho = t_\Phi\circ\BD^X\,,\]
where $t_\Phi$ is a derived equivalence of infinite order induced by a spherical functor~$\Phi$.

The image of a stable pair under~$\rho$ leads to complicated objects in the derived category and a symmetry on invariants is not easily deduced. Instead, we consider an abelian category 
\[ \CA\subset D^{[-1,0]}(X)\,,\]
defined as a tilt of $\Coh(X)$ along a torsion pair. The action of $\rho$ on $\CA$ is analogous to the action of $\BD^X$ on $\Coh(X)$. In particular, we consider a notion of dimension which is preserved by $\rho$ (up to shift). Define the extension closure
\[ {}^p\CB=\Big\langle \CO_X[1]\,,\CA_{\leq 1}\Big\rangle_\ex\,.\]
The action of $\rho$ induces a symmetry for \emph{perverse PT invariants}~$\PPT_{\gamma, n}$ enumerating torsion-free objects in~${}^p\CB$. These objects are allowed to have non-trivial first Chern class a multiple of the class $\sfw$. For $r\in\BZ$ and $\gamma=(r\sfw,\beta)$ define the generating series
\[ {}^p\PT_\gamma(q,Q) = \sum_{n,j\in\BZ} \PPT_{\gamma+j\sfb, n}\, (-q)^n\, Q^j \in \BQ[[q^{\pm 1}, Q^{\pm 1}]]\,. \]
The rationality and functional equation for ${}^p\PT_\gamma$ is proved via Joyce's wall-crossing formula~\cite{Jo07}. The formula involves generalized DT invariants for surface-like objects supported on $W$ with non-trivial Euler pairings.

\begin{theorem}\label{thm: perverse} 
\[ {}^p\PT_\gamma(q,Q) \in \BQ(q,Q)\]
is the expansion of a rational function $f_\gamma\in\BQ(q,Q)$ with functional equation
\[f_\gamma(q^{-1},Q^{-1}) = Q^{-\sfw\cdot\beta+2r}\, f_\gamma(q,Q)\,.\]

\end{theorem}

Theorem~\ref{thm: main} is a consequence of Theorem~\ref{thm: perverse} in the special case $r=0$, together with the $q\leftrightarrow q^{-1}$ symmetry. The comparison between stable pairs and perverse stable pairs is given by a second wall-crossing. The following formula holds as an equality of rational functions but not necessarily as generating series.

\begin{theorem}\label{thm: wall-crossing}
\[{}^p\PT_{(0,\beta)}(q,Q) = \frac{\PT_\beta(q,Q)}{\PT_0(q,Q)}\,.\]
\end{theorem}

\subsection{Crepant resolution}\label{subsec: crepant}
The results and techniques of this paper are strongly influenced by the recent proof of the crepant resolution conjecture by Beentjes, Calabrese and Rennemo~\cite{BCR18} for Donaldson--Thomas~(DT) invariants~\cite{Th00}. Consider a type III contraction $X\to Y$ with exceptional divisor $W$, contracting the rational curves of the ruling. Assume that $X\to Y$ is the distinguished crepant resolution of the (singular) coarse moduli space of a Calabi--Yau orbifold~$\CY$
\[
\begin{tikzcd}[row sep=small, column sep=small]
\CY \ar[dr]&&X\ar[dl]\\&Y&
\end{tikzcd}
\]
 The derived McKay correspondence proposed by Bridgeland, King, and Reid~\cite{BKR01}  induces a derived equivalence
\[ \Phi\colon D^b(X) \xrightarrow{\sim} D^b(\CY)\,,\]
which restricts to an equivalence~\cite[Theorem~1.4]{Ca16}
\[ \CA \xrightarrow{\sim}\Coh(\CY) \,.\]
The notion of perverse stable pairs on $X$ coincides with the image of stable pairs on $\CY$. The results of Theorems~\ref{thm: main},~\ref{thm: perverse}, and \ref{thm: wall-crossing} are the rationality and functional equation of $\PT(\CY)$ and the wall-crossing between $\Phi^{-1}\big(\PT(\CY)\big)$ and Bryan--Steinberg pairs of $X\to Y$~\cite{BCR18,BS16,Ca16}. The nef class is given by the pullback of an ample class on~$Y$ and the derived anti-equivalence $\rho$ corresponds to the derived dual of~$\CY$
\[ \rho = \Phi^{-1}\circ \BD^\CY\circ \Phi\,.\]

\subsection{Spherical twist}\label{subsec: spherical twists}
Define the functor $\Phi\colon D^b(C)\to D^b(X)$ as
\[\Phi(V)=\iota_\ast\big(\CO_p(-1)\otimes p^\ast V\big)\,.\]
This defines a spherical functor~\cite{AL17, Ho05, ST01}. Let $\Phi_R$ be the right adjoint. The cone of the counit morphism defines the spherical twist~$\twist$, an autoequivalence of~$D^b(X)$, via
\[ \Phi\circ\Phi_R \to \id \to \twist\,.\]
The derived dual $\BD^X$ and the spherical twist~$\twist$ generate an infinite dihedral group (containing~$\rho$) which underlies the functional equations of Theorem~\ref{thm: main}.

\subsection{Gromov--Witten/ BPS invariants}

The second functional equation of Theorem~\ref{thm: main} implies strong constraints for the enumerative invariants in curve classes $\beta+j\sfb$ for varying $j\in\BZ$ and fixed genus. In particular, finitely many~$j$ determine the full set of these invariants. Let $\GW_{g,\beta}$ be the Gromov--Witten invariants of $X$ and assume that the GW/PT correspondence~\cite{MNOPI,MNOPII,MOOP11,PP17} holds for~$X$.
\begin{corollary}\label{cor: rationality} For all $(g,\beta)\neq (0,m\sfb)\,,(1,m\sfb)$ the series
\[ \sum_{j\in \BZ} \GW_{g,\beta+j\sfb}\, Q^j \]
is the expansion of a rational function $f_\beta(Q)$ with functional equation
\[ f_\beta(Q^{-1}) = Q^{-\sfw\cdot\beta} f_\beta(Q)\,.\]
\end{corollary}
The rational function is expected to have the particular form
\[ f_\beta(Q) = \frac{p_\beta(Q)}{(1-Q)^{d}}\]
which leads to polynomiality of $\GW_{g,\beta+j\sfb}$ and the limit behavior of BPS counts (as $j\to\infty$) discussed in the physics literature~\cite[Section 5]{KKV97}. For the local Hirzebruch surface $K_W$ we give full proofs in Appendix~\ref{appendix}.

\subsection{Elliptic Calabi--Yau 3-folds}
Let $p\colon S\to C$ be a $\BP^1$-bundle over a smooth projective curve~$C$ and
\[f\colon X\to S\]
an elliptic fibration\footnote{Since $X$ is Calabi--Yau, $C$ is necessarily rational and $S$ is a Hirzebruch surface.} with a section~$W$. Let $D$ be a sufficiently ample divisor on~$C$ such that $-K_S+p^\ast D$ is ample. A nef class satisfying the~condition~\eqref{eq: nef} is given by
\[\sfw+f^*(-K_S+p^\ast D)\in H^2(X,\BZ)\,.\]

For any $\beta\in H_2(W,\BZ)$ define 
\[ P_\beta(q,t) = \sum_{d\geq 0}\sum_{n\in \BZ} \PT_{\beta+d\sff, n}\, (-q)^n\, t^d\,,\]
where $\sff\in H_2(X,\BZ)$ is the class of a smooth elliptic fiber. Recent considerations in topological string theory~\cite{HKK15} predict that 
\[ Z_\beta(q, t) = \frac{P_\beta(q,t)}{P_0(q,t)}\]
is the expansion of a meromorphic Jacobi form. Theorem~\ref{thm: main} implies non-trivial constraints among the Jacobi forms $\{Z_{\beta+j\sfb}\}_{j\in\BZ}$.

\subsection{STU}\label{subsec: STU}
Theorem~\ref{thm: main} and Corollary~\ref{cor: rationality} provide mathematical proofs of a heterotic mirror symmetry on BPS invariants as observed in~\cite{KKRS05}. The symmetry is discussed for type IIA duals of the STU model, i.e.\ the elliptic Calabi--Yau 3-fold 
\[X\to \BP^1\times \BP^1\]
such that both projections to $\BP^1$ define K3-fibrations with~$528$ singular fibers with exactly one double point as singularity. This geometry can be constructed as a hypersurface in a toric variety~\cite{KMPS10}.

The symmetry on BPS invariants~\cite[Section 6.10.3]{KKRS05} is realized by the second functional equation of Theorem~\ref{thm: main} and we can identify the infinite order symmetry~\cite[Equation 6.65]{KKRS05} with the action of~$\twist$. The rationality and functional equation of Corollary~\ref{cor: rationality} verifies~\cite[Equation 6.67]{KKRS05}. We obtain the precise form of the rational function for the local case $K_{\BP^1\times\BP^1}$ in Appendix~\ref{appendix}.\medskip

As a special case of the rationality and functional equation, consider $\beta=h\sff$ a multiple of the elliptic fiber class. Then, the generating function is in fact a \emph{Laurent polynomial} in $Q$ and the functional equation 
\[ f_\beta(q,Q^{-1})=Q^{-\sfw\cdot\beta}\,f_\beta(q,Q)\]
holds at the level of coefficients and recovers the symmetry 
\[n^{K3}_{g,m\sfb+h\sff}=n^{K3}_{g,(h-m)\sfb+h\sff}\]
of BPS invariants for K3 surfaces. This symmetry is usually seen as a consequence of the monodromy for quasi-polarized K3 surfaces.\medskip

A related geometry, also called an STU model in the physics literature, may be useful towards a crepant resolution conjecture in the \emph{non hard Lefschetz} case. We consider 
\[ X \to \BF_1\]
an elliptic Calabi--Yau 3-fold over the Hirzebruch surface~$\BF_1$. The fibration has a section~$W$ and we obtain $X \dashrightarrow X'$ as the Atiyah flop along the rational curve in~$W$ of self-intersection~$-1$. After this transformation we have a type~II contraction $X'\to Y'$ with exceptional divisor~$\BP^2$, which is the crepant resolution of an isolated canonical singularity. After the flop formula for DT invariants~\cite{Cal16,To13}, the symmetry of Theorem~\ref{thm: main} must induce a symmetry on~$X'$.

\subsection{Outline}
We briefly sketch the strategy of the paper. Section~\ref{sec: perverse t-structure} contains a discussion of perverse sheaves and the perverse $t$-structure associated to the geometry. We introduce the anti-equivalence~$\rho$ and show several important properties that will be needed in the later parts. In Section~\ref{sec: Hall} we recall some facts about Hall algebras, pairs, and wall-crossing, and we set notation for the rest of the paper. Stability conditions play an important role for this paper and we comment on them in Section~\ref{sec: stability}. In Section~\ref{sec: BS} we introduce invariants which resemble Bryan--Steinberg invariants \cite{BS16} and we prove a wall-crossing formula between those and usual $\PT$ invariants. The wall-crossing formula shows a relation of the form
\[\bs_\beta=\frac{\PT_\beta}{\PT_0}\]
and thus gives a natural interpretation to the quotient on the right hand side. The rationality and symmetry for $\PPT$ invariants are proven in Section~\ref{sec: perversept}. Essentially, the result is obtained by comparing $\PPT$ invariants with $\rho(\PPT)$ invariants in two ways: first using the anti-equivalence~$\rho$, and then using wall-crossing. In Section~\ref{sec: zetawallcrossing} we describe a wall-crossing between the $\bs$ invariants and the perverse $\PPT$ invariants (which in the crepant case $X\to Y$ are the orbifold invariants). An important aspect is that while $\PT$ and $\bs$ invariants are defined using the integration map on the Hall algebra obtained from the heart $\Coh(X)\subset D^b(X)$, the perverse $\PPT$ invariants are defined using the heart $\CA\subset D^b(X)$. The $\zeta$-wall-crossing of Section~\ref{sec: zetawallcrossing} takes place in $\CA$. In Section~\ref{subsec: limitI} we identify $\bs$-pairs as the pairs in the end of the $\zeta$–wall-crossing.
The following diagram represents the different invariants we use in the paper and their relations. The squiggly arrows represent wall-crossing formulas. 
\begin{center}
    \begin{tikzcd}[row sep=1.2cm]
     \PT\arrow[r,leftrightsquigarrow, "\ref{sec: BS}"]&
   \bs\arrow[r, Rightarrow, no head, "\ref{subsec: limitI}"] &
    (\zeta, 0)\arrow[r,leftrightsquigarrow, "\ref{sec: zetawallcrossing}"]&
    \PPT\arrow[d,leftrightsquigarrow, bend left, "\ref{sec: perversept}"]\arrow[d,leftrightarrow, bend right, "\ref{subsec: rho}" swap]
    \\ 
    &
    &
    &
    \rho(\PPT)
    \end{tikzcd}
\end{center}

\subsection{Related work}
The following question was posed by Toda~\cite{To16}:
\begin{question}\label{question: Toda}
How are stable pair invariants on a Calabi–Yau 3–fold constrained, due to the presence of non-trivial autoequivalences of the derived category?
\end{question}
The most famous instance is the rationality and functional equation induced by the derived dual. Similarly, the elliptic transformation law for $Z_\beta(q,t)$ is deduced from a derived involution~\cite{OS20}. Significant progress for abelian 3-folds was made using Bridgeland stability conditions~\cite{OPT18}. The Seidel--Thomas spherical twist for an embedded $\BP^2$ was considered in~\cite{To16} and certain polynomial relations for stable pairs invariants were obtained. Our results provide an answer to Question~\ref{question: Toda} for the involution~$\rho$. The flop construction $X\dashrightarrow X'$ of the previous section must connect our results with the ones obtained in~\cite[Theorem 1.2]{To16}.

\subsection{Conventions}
We work over the complex numbers. The canonical bundle of~$W$ is denoted $K_W$. Intersection products are denoted by a dot, e.g.\ $\sfw\cdot\beta$. Stable pairs are considered in cohomological degree $-1$ and $0$. This convention follows~\cite{BCR18} and differs from~\cite{PT09}.

\subsection{Acknowledgements} We are grateful to Y.\ Bae, T.\ Beckmann, G.\ Oberdieck, R.\ Pandharipande, D.\ Nesterov, J.\ Rennemo, E.\ Scheidegger, R.\ Thomas for discussions on stable pairs in the derived category and curve counting on Calabi--Yau 3-folds. Conversations with E.\ Scheidegger on the STU model were very helpful. We thank G.\ Oberdieck for pointing out the connection to the DT crepant resolution conjecture, and R.\ Thomas for discussions on spherical twists and wall-crossing. The first author thanks the IHES for hospitality during the final stage of this work. The authors were supported by ERC-2017-AdG-786580-MACI. The project received funding from the
European Research Council (ERC) under the European Union Horizon 2020 research and innovation programme (grant agreement 786580).


\section{Perverse t-structure}\label{sec: perverse t-structure}

In this section we give a self-contained discussion of perverse sheaves and duality associated to the following geometry. 

\subsection{Geometry}\label{subsec: geometry} Let $C$ be a smooth projective curve, $\CE$ a locally free sheaf of rank~$2$, and $W=\BP_C(\CE)$ a geometrically ruled surface with projection $p\colon W \to C$. We assume that $\CE^\vee$ is globally generated\footnote{Twisting~$\CE^\vee$ with an ample line bundle does not change the geometry~$\BP_C(\CE)$.} and we fix line bundles $L_1,L_2\in\Pic(C)$ such that
\[ 0 \to L_1 \to \CE^\vee \to L_2 \to 0\,.\] 

Let $X$ be a smooth projective Calabi--Yau 3-fold containing~$W$ as a divisor: 
\[
\begin{tikzcd}[row sep=small, column sep=small]
W \ar[r, hook, "\iota"] \ar[d, "p", swap]& X \\ C
\end{tikzcd}
\]
The curve class of a fiber of~$p$ (and its pushforward to~$X$) is denoted~$\sfb$. The nef class~$A$ of condition~\eqref{eq: nef} restricts to a multiple of the fiber class, i.e.\ $\iota^\ast A$ is numerically equivalent to $a_0\sfb$ for some $a_0\in\BZ_{>0}$. Recall that we have the Euler sequence on~$W$ which we will use repeatedly:
\[ 0\to \omega_p\to \CO_p(-1)\otimes p^\ast \CE^\vee \to \CO_W\to 0\,.\]

\subsection{Torsion pair}
Define the category
\[ \CT = \big\{T\in \Coh(X) \mid R^1p_\ast(\iota^\ast T)=0\big\}\,.\]
\begin{lemma}
The subcategory $\CT\subset \Coh(X)$ is closed under extensions and quotients in $\Coh(X)$.
\end{lemma}
\begin{proof}
Use the long exact sequence of higher pushforward sheaves and the fact that $R^2p_\ast=0$ since the fibers of $p$ are $1$-dimensional.
\end{proof}

By~\cite[Lemma 2.15]{To13} we obtain a torsion pair $(\CT,\CF)$ in $\Coh(X)$ where 
\[\CF=\CT^\perp=\{F\in\Coh(X)\mid \Hom(T,F)=0 \text{ for all }T\in\CT\}\,.\]
We consider the \emph{perverse $t$-structure} on $D^b(X)$ whose heart is the tilt~\cite{HRS96}
\[\CA=\big\langle \CF[1],\CT\big\rangle\,.\]

Every $E\in D^b(X)$ has associated perverse cohomology ${}^p\CH^i(E)\in\CA$ and exact triangles lead to long exact sequences of perverse cohomology. 
Define the \emph{perverse dimension}
\[^p\dim(E) = \max\big\{\dim \supp(E)\cap(X\ssetminus W)\,,\dim p\big(\supp(E)\cap W\big)\big\}\,.\] 
We write $\CA_{\leq k}$ for elements of $\CA$ with perverse dimension at most~$k$ and $\CA_{k}$ for elements with pure perverse dimension~$k$, i.e.\ \[ \Hom (\CA_{\leq k-1},\CA_k)=0\,.\]
We also denote $\CF_k[1]=\CF[1]\cap \CA_k$ and $\CT_k=\CT\cap \CA_k$.

\subsection{Duality}\label{subsec: rho}
The derived dualizing functor~$(-)^\vee=R\CHom(-,\CO_X)$ is a duality for the standard $t$-structure on~$D^b(X)$. We introduce a duality $\rho$ on $D^b(X)$ which is the analog for the perverse $t$-structure. 

Define the functor $\Phi\colon D^b(C)\to D^b(X)$ as
\[\Phi(V)=\iota_\ast\big(\CO_p(-1)\otimes p^\ast V\big)\,.\]
The right adjoint is
\[ \Phi_R(V) = Rp_\ast\big( \CO_p(1)\otimes\omega_W[-1]\otimes L\iota^\ast V\big)\,.\]
The cotwist $\cotwist$ is defined as the cone of the unit morphism
\[ \id \to \Phi_R\circ\Phi \to \cotwist\,.\]
A direct calculation shows that $\Phi_R\circ\Phi$ splits as \[\Phi_R\circ\Phi \cong \id\,\oplus \,\omega_C[-2]\]
and $\cotwist$ is isomorphic to $\omega_C[-2]$, which is an auto-equivalence. Thus, $\Phi$ is a spherical functor~\cite{AL17, Ho05, ST01} and we obtain an auto-equivalence of $D^b(X)$, the twist $\twist$, defined as the cone of the counit morphism~\cite[Theorem 1.1]{AL17}
\[ \Phi\circ\Phi_R \to \id \to \twist\,.\]
We consider an anti-equivalence of order two defined as\footnote{The derived dual of Section~\ref{subsec: overview} is $\BD^X=[2]\circ(-)^\vee$.}
\[ \rho = \twist\circ[2]\circ(-)^\vee\,.\]
For any $E\in D^b(X)$ we have the important exact triangle
\begin{equation}\label{eq: triangle} \tag{$\Delta$} E^\vee[2] \to \rho(E) \to \Phi\circ\Phi_R[1]\big(E^\vee[2]\big)\,.\end{equation}

We can now state the main, and most difficult, result of this section.
\begin{theorem}\label{thm: rho}\hfill
\begin{enumerate}
    \item $\rho\big(\CA_0\big)\subset\CA_0[-1]$,
    \item $\rho\big(\CA_1\big) \subset \CA_1$.
\end{enumerate}
\end{theorem}
\begin{proof}[Outline]
The proof will be given in Sections~\ref{subsec: basic} to~\ref{subsec: 1-dim perverse}. We start with properties and basic results in Sections~\ref{subsec: properties},~\ref{subsec: basic}. In Section~\ref{subsec: supp} we prove that objects in $\CA$ with support contained in~$W$ are successive extensions of objects which are \emph{scheme-theoretically} supported on~$W$. This will also be applied in Section~\ref{sec: stability} to prove that a function~$\nu$ defines a stability function on~$\CA_{\leq 1}$. Theorem~\ref{thm: rho} (i) is proved in Section~\ref{subsec: 0-dim perverse}. In Section~\ref{subsec: 1-dim perverse} we prove that for any $E\in\CA_{\leq 1}$ the perverse cohomology sheaves satisfy
\begin{equation}\label{eq: star}\tag{$\ast$}
    {}^p\CH^i\big(\rho(E)\big)=0\,, \quad i\neq 0,1\,,\quad {}^p\CH^1\big(\rho(E)\big)\in\CA_0\,.
\end{equation}
This suffices to deduce Theorem~\ref{thm: rho} (ii).
\end{proof}

Theorem~\ref{thm: rho} should remind the reader of an analogous property of the derived dual $\BD^X$ acting on $\Coh(X)$:
\[\BD^X(\Coh_0(X))=\Coh_0(X)[-1] \,,\quad \BD^X(\Coh_1(X))=\Coh_1(X)\,.\]
Indeed, the next section clarifies the origin of this analogy.

\subsection{Crepant case}

We explain now our main motivation for the tilt $\CA$ and for the derived anti-equivalence $\rho$ by considering the case of a type III contraction $X\to Y$, as described in Section~\ref{subsec: crepant}. 

In this setting, $Y$ is the coarse moduli space of a Calabi--Yau orbifold $\CY$ that has $B\BZ_2$-singularities along a copy of the curve~$C$. The derived categories of $X$ and $\CY$ are isomorphic via the derived McKay correspondence~\cite{BKR01}
\[\Phi\colon D^b(X)\xrightarrow{\sim} D^b(\CY)\,.\]

The heart $\CA\subset D^b(X)$ coincides with Bridgeland's category of perverse sheaves~\cite{Br02,VdB04}
\[\CA={}^0\Per(X/Y)\,,\]
so under the McKay correspondence it should be regarded as $\Coh(\CY)$.
Indeed, let $j_0\colon C_0\hookrightarrow Y$ be the contraction of~$W$, i.e.\ $C_0=\pi(W)$. Then, for any $T\in \Coh(X)$ the higher pushforward $R^1 \pi_\ast T$ is supported on $C_0$, so $R^1 \pi_\ast T=0$ if and only if
\[0=j^\ast R^1 \pi_\ast T=R^1 p_\ast \iota^\ast T.\]
The equality used holds by the proper base change theorem.

Under the McKay correspondence, the notion of perverse dimension that we defined coincides with the usual dimension on the orbifold. The anti-equivalence $\rho$ coincides with the derived dual $\BD^\CY=R\CHom(-,\CO_{\CY})[2]$ on the orbifold, i.e.:
\begin{proposition}\label{prop: rhocrepant}In the setting above, we have
\[ \rho = \Phi^{-1}\circ\BD^{\CY}\circ\Phi\,.\]
\end{proposition}
\begin{proof}
We let $\Psi=\Phi^{-1}\circ \BD^{\CY}\circ\Phi \circ \rho$. Since $\Phi$ is a derived equivalence, whereas $\rho$ and $\BD^\CY$ are derived anti-equivalences, the composition $\Psi$ is a derived equivalence. We prove that $\Psi$ is isomorphic to the identity by analysing $\Psi(k(x))$ and using again~\cite[Corollary~5.23]{Hu06}. 

If $x\in X\ssetminus W$ then Lemma~\ref{lem: 0-dim generators} shows that 
\[\Psi(k(x))=\big(\Phi^{-1}\circ \BD^{\CY}\circ\Phi\big)\big(k(x)[-1]\big)=\big(\Phi^{-1}\circ \BD^{\CY}\big)\big(k(\pi(x)\big)[1]=k(x)\,.\]

For $x\in W$, one has the exact triangle of Lemma~\ref{lem: 0-dim generators} and applying $\Phi^{-1}\circ \BD^{\CY}\circ\Phi$ to it produces the exact triangle 
\begin{equation}\CO_B(-1)\to \Psi(k(x))\to \CO_B(-2)[1].\label{eq: trianglelocal}\end{equation}
We used that $\Phi^{-1}\circ \BD^{\CY}\circ\Phi$ is an anti-equivalence and we determine the images of $\CO_B(-2), \CO_B(-1)[-1]$ using~\cite[Section~4.3]{BCY12},~\cite[Appendix A]{BCR18}
\[
    \Phi\big(\CO_B(-2)[1]\big)=\CO_p^+\,,\quad \Phi\big(\CO_B(-1)\big)=\CO_p^-\,,\quad \BD^\CY\big(\CO_p^\pm\big)=\CO_p^\pm[-1]\,.
\]

Extensions determined by \eqref{eq: trianglelocal} are classified by \[\Hom(\CO_B(-2), \CO_B(-1))\cong \BC^2\]
and we get that $\Psi(k(x))\cong k(f(x))$ for some $f(x)\in B=\pi^{-1}\big(\pi(x)\big)$. 

By \cite[Corollary~5.23]{Hu06} it follows that $f\colon X\to X$ is an isomorphism and $\Psi= (M\otimes - )\circ f_*$ for some line bundle $M$. Since $f_{|X\ssetminus W}=\textup{id}_{X\ssetminus W}$, we conclude that $f=\textup{id}$. By Proposition~\ref{prop: rho} and the fact that $\Phi$ preserves structure sheaves, one easily sees that $\Psi(\CO_X)=\CO_X$ and thus $M$ is the trivial line bundle, so $\Psi\cong \textup{id}$.\qedhere
\end{proof}

As we mentioned in Section~\ref{subsec: crepant}, when $X$ is obtained as a crepant resolution our results follow from \cite{BCR18}. The previous proposition explains how the heart $\CA$ and the duality $\rho$ play the role of $\Coh(\CY)$ and $\BD^\CY$, respectively, in the proof of the rationality and functional equation for the orbifold $\PT$ invariants \cite{BCR18}. 

\subsection{Properties of \texorpdfstring{$\rho$}{rho}}\label{subsec: properties}

We gather here some of the key properties of the duality operator $\rho$. We begin with a direct computation of the image of some objects (of perverse dimension $0$) under $\rho$.

\begin{lemma}\label{lem: 0-dim generators}
For all points $x\in X$ and fibers $B\subset W$
\begin{enumerate}
    \item If $x\not \in W$, then $\rho(k(x))=k(x)[-1]$,
    \item $\rho\big(\CO_B(-2)[1]\big) = \CO_B(-2)$,
    \item $\rho\big(\CO_B(-1)\big) = \CO_B(-1)[-1]$,
    \item if $x\in B$ there is an exact triangle
    \[ \CO_B(-2) \to \rho\big(k(x)\big)\to \CO_B(-1)[-1]\,,\]
    \item for all $k\leq -2$, $\rho\big(\CO_B(k)[1]\big)\in\CA_0[-1]$,
    \item for all $k\geq -1$, $\rho\big(\CO_B(k)\big)\in\CA_0[-1]$.
\end{enumerate}
\end{lemma}
\begin{proof}
Part~(i) follows from $k(x)^\vee[2] = k(x)[-1]$ and $\Phi_R\big(k(x)\big)=0$. Part~(ii) and (iii) are computed directly. Then, any $x\in B$ corresponds to an exact triangle
\[ \CO_B(-1) \to k(x) \to \CO_B(-2)[1]\,,\]
and application of $\rho$ yields~(iv). For~(v) and~(vi) we can use induction on~$k$ to reduce to~(ii) and (iii) respectively.
\end{proof}

\begin{proposition}\label{prop: rho} We have
\[ \rho\big(\CO_X) = \CO_X[2]\,,\qquad \rho\circ \rho = \id\,.\]
\end{proposition}
\begin{proof}
The first claim follows from $\Phi_R(\CO_X)=0$, thus 
\[ \rho(\CO_X) = \CO_X^\vee[2]=\CO_X[2]\,.\]
For the second claim we use the computations in Lemma~\ref{lem: 0-dim generators} which imply that for all $x\in X$ there is $y\in X$ such that
\[ \rho\circ\rho\big(k(x)\big)\cong k(y)\,.\]
Moreover, $x=y$ for $x\in X\ssetminus W$. Now we apply the general fact~\cite[Corollary~5.23]{Hu06} that any auto-equivalence $\Psi$ with $\Psi(k(x))\cong k(f(x))$ is of the form
\[ \Psi = (M\otimes - )\circ f_*\,,\]
where $f\colon X\to X$ is an isomorphism and $M$ is a line bundle. Then $f_{|X\ssetminus W}=\id$, thus $f=\id$, and by the first claim $M$ must be the trivial line bundle. 
\end{proof}

The action of~$\rho$ on cohomology can be directly computed using the exact triangle~\eqref{eq: triangle}. For our purposes, it suffices to consider objects $E\in\CA_{\leq 1}$, in particular $\ch_0(E)=0$, and $\ch_1(E)$ is a multiple of~$\sfw$. It is convenient to compute the action using $(\ch_1,\ch_2,\chi)$.

\begin{proposition}\label{prop: actionCoh}
The anti-equivalence~$\rho$ acts on $(\ch_1,\ch_2,\chi)$ as
\[\big(r\sfw,\beta,n\big)\stackrel{\rho}{\longmapsto} \big(r\sfw,\beta+(\sfw\cdot\beta-2r)\,\sfb,-n\big)\,.\]
\end{proposition}
\subsection{Basic results (proof of Theorem~\ref{thm: rho})}\label{subsec: basic}

We start by setting up some notation that will later be useful in the induction process we'll use.
\begin{notation}\label{not: ch}
Let $\omega\in\Amp(X)$ be an ample class and $E\in\Coh(X)$ with at most $1$-dimensional support outside of~$W$. Denote by $\ch_i^\omega(E)=\omega^{3-i}\cdot\ch_i(E)$. We write $\ch^\omega(E')<\ch^\omega(E)$, if 
\begin{enumerate}
    \item $0\leq \ch_1^\omega(E')<\ch_1^\omega(E)$, or
    \item $0=\ch_1^\omega(E')=\ch_1^\omega(E)$, and $0\leq\ch_2^\omega(E')<\ch_2^\omega(E)$, or
    \item $0=\ch_1^\omega(E')=\ch_1^\omega(E)$, and $0=\ch_2^\omega(E')=\ch_2^\omega(E)$, and\linebreak[5] $0~\leq~\ch_3^\omega(E')<\ch_3^\omega(E)$.
\end{enumerate}
Then, $\ch^\omega(E)\geq 0$ with equality if and only if $E=0$. Note that $\ch^\omega(E)>0$ is minimal if and only if $E\cong k(x)$ for some $x\in X$.
\end{notation}
\begin{notation}\label{not: R1p}
For $G',G\in\Coh(W)$ we write $R^1p_\ast G'<R^1p_\ast G$ if
\begin{enumerate}
    \item $0\leq \rk\big(R^1p_\ast G'\big)<\rk\big(R^1p_\ast G\big)$, or
    \item $0=\rk\big(R^1p_\ast G'\big)=\rk\big(R^1p_\ast G\big)$, and $\len\big(R^1p_\ast G'\big)<\len\big(R^1p_\ast G\big)$, where $\len(-)$ is the length of a $0$-dimensional sheaf.
\end{enumerate}
\end{notation}
\begin{lemma}\label{lem: basic}
\begin{enumerate}
    \item For all $T\in\Coh(X)$ 
    \[ R^1p_\ast L\iota^\ast T = R^1p_\ast\iota^\ast T \,.\]
    \item There is a short exact sequence
    \[0 \to R^1p_\ast L^{-1} \iota^\ast(T) \to p_\ast L\iota^\ast (T) \to  p_\ast \iota^\ast (T)\to 0\,.\]
    \item For all $G\in\Coh(W)$, $L^k\iota^\ast\iota_\ast G=0$ for $k\neq 0,-1$ and
    \[ L^{-1}\iota^\ast \iota_\ast G = \omega_W^{\vee}\otimes G\,,\quad \iota^\ast \iota_\ast G = G\,.\]
\end{enumerate}
\end{lemma}
\begin{proof}
There is a spectral sequence
\[ E_2^{k,l}=R^kp_\ast \CH^l(L\iota^\ast T) \implies R^{k+l}p_\ast L\iota^\ast T\,.\]
Since $\dim(p)=1$, the only non-vanishing term contributing to $R^1p_\ast L\iota^\ast T$ is $R^1p_\ast\iota^\ast T$ and the differentials vanish. The second statement follows analogously. The third assertion follows from $\iota^\ast \CO_X(-W)=\omega_W^{\vee}$.\qedhere
\end{proof}

\begin{lemma}\label{lem: Rp}\hfill
\begin{enumerate}
    \item If $\iota_\ast G\in\CF$, then $p_\ast G=0$,
    \item if $\iota_\ast G\in\CA_{\leq 1}$, then $Rp_\ast(G)\in\Coh(C)$.
\end{enumerate}
\end{lemma}

\begin{proof}
If $p_\ast G\neq 0$ we may choose a sufficiently ample $L\in\Pic(C)$ and a non-zero section $L^\vee\to p_\ast G$. By adjunction we have a non-zero $p^\ast L^\vee\to G$. This contradicts $\iota_\ast G\in\CF$ because $R^1p_\ast p^\ast L^\vee=0$, i.e.\ $\iota_\ast p^\ast L^\vee\in\CT$. The statement~(ii) follows from (i) and the definition of~$\CT$.
\end{proof}

\begin{lemma}\label{lem: ample}
\begin{enumerate}
    \item For all $G\in\Coh(W)$
    \[\rk\big(R^1p_\ast(\CO_p(1)\otimes G)\big) \leq \rk(R^1p_\ast G)\,,\]
    with strict inequality if $\rk(R^1p_\ast G) >0$. In that case
    \[\rk\big(R^1p_\ast(\omega_W^\vee\otimes G)\big) < \rk(R^1p_\ast G)\,.\]
    \item If $\rk(R^1p_\ast G)=0$, then
    \[\dim\big(R^1p_\ast(\CO_p(1)\otimes G)\big) \leq \dim(R^1p_\ast G)\,,\]
    with strict inequality if $\dim(R^1p_\ast G)>0$. In that case
    \[\dim\big(R^1p_\ast(\omega_W^\vee\otimes G)\big) < \dim(R^1p_\ast G)\,.\]

\end{enumerate}
\end{lemma}

\begin{proof}The second assertion follows from the first one since 
\[\omega_W = \CO_p(-2)\otimes p^\ast(\omega_C\otimes\det\CE^\vee)\,.\]
(i) Denote by $r_k=\rk\big(R^1p_\ast(\CO_p(k)\otimes G)\big)$. Let $C_0\subset W$ be the zero locus of a section of $\CO_p(1)$, thus $C_0$ is a section of the projection~$p$. For all $k\in\BZ$ there is a sequence
\[ \CO_p(k-1)\otimes G \to \CO_p(k)\otimes G \to \CO_{C_0}(k)\otimes G\to 0\,.\]
Thus, $r_{k}\leq r_{k-1}$. The Euler sequence on $W$ implies
\begin{align*} &\det(\CE^\vee)\otimes R^1p_\ast(\CO_p(k-2)\otimes G) \to \CE^\vee\otimes R^1p_\ast(\CO_p(k-1)\otimes G)\\
&\to R^1p_\ast(\CO_p(k)\otimes G)\to 0 \,,
\end{align*}
thus $r_{k-2}-2\,r_{k-1}+r_k\geq 0$. If $r_{k-1}=r_{k-2}$, then $r_k=r_{k-1}$, thus $r_k=r_0>0$ for all $k\geq 0$. This is a contradiction since $\CO_p(1)$ is $p$-ample and so $r_k=0$ for $k\gg 0$. For (ii) The proof is the same, with rank replaced by the length of $0$-dimensional sheaves.
\end{proof}
Recall the sequence from Section~\ref{subsec: geometry}
\[ 0 \to L_1 \to \CE^\vee \to L_2 \to 0\,.\]
Let $g$ be the genus of~$C$ and define
\begin{align*}
k_-&=-g+\min\{0,\deg(L_1),\deg(L_2)\}-1\,,\\
k_+&=-g+\max\{0,\deg(L_1),\deg(L_2)\}+1\,.
\end{align*}
We have the following technical lemma which we will apply multiple times.
\begin{lemma}\label{lem: morphism}
Let $0\neq \iota_\ast G \in \CA_{\leq 1}$. There is a line bundle $L\in\Pic(C)$ and a non-zero morphism $K \to G$ with 
    \[K=\CO_p(-1)\otimes p^\ast L\,, \textup{ or } K=\omega_p\otimes p^\ast L[1]\,.\]
If $Rp_\ast G \neq 0$, we may choose $L$ such that
\[k_- +\frac{\chi(G)}{\max\{\rk(Rp_\ast G), 1\}}\leq \chi(L)\leq k_+ + \frac{\chi(G)}{\max\{\rk(Rp_\ast G), 1\}}\,.\]
If $Rp_\ast G = 0$, we may choose $L$ such that
\[\chi(L) = \chi\big(G\otimes\CO_p(1)\big) - 1\,.\]
\end{lemma}
\begin{remark}
Note that if $\iota_\ast G\in \CT_{\leq 1}$, then $K=\CO_p(-1)\otimes p^\ast L$ because
\[ \Hom(\CF[1], \CT)=0\,.\] 
\end{remark}

\begin{proof}
Recall that $Rp_\ast G \in \Coh(C)$ is a sheaf by Lemma~\ref{lem: Rp}, in particular $\rk(Rp_\ast G)\geq 0$. First, assume that $Rp_\ast G \neq 0$. Let $M \in \Pic(C)$ with
\[ \rk(Rp_\ast G)\deg(M) < \chi(G)\,,\]
then by Riemann--Roch
\[H^0(C, M^\vee\otimes Rp_\ast G)\neq 0\,.\]
We may choose $M$ so that $\deg(M)$ is the nearest integer to
\[\frac{\chi(G)}{\max\{\rk(Rp_\ast G), 1\}}-1\,.\]
Note that when $\rk(Rp_\ast G)=0$ we must have $\chi(G)=\chi(Rp_\ast G)\geq 0$ since $Rp_\ast G\in \Coh_0(C)$.

Now we can use the Euler sequence on~$W$ which yields an exact triangle in~$\CA_{\leq 1}$:
\[\CO_p(-1)\otimes p^\ast(\CE^\vee\otimes M) \to p^\ast M \to \omega_p\otimes p^\ast M[1]\,.\]
Since $\Hom(p^\ast M, G)=H^0(C, M^\vee\otimes Rp_\ast G)\neq 0$, we find
\begin{align*} &\Hom(\CO_p(-1)\otimes p^\ast(\CE^\vee\otimes M), G)\neq 0\,,\textup{ or } \\
&\Hom(\omega_p\otimes p^\ast L[1], G)\neq 0\,.
\end{align*}
In the latter case, set $K=\omega_p\otimes p^\ast M[1]$ and $L=M$. In the former case we can use the sequence
\[ 0 \to L_1 \to \CE^\vee \to L_2 \to 0\]
and argue as above, i.e.\ we can set $K=\CO_p(-1)\otimes L_i\otimes M$ and $L=L_i\otimes M$ for $i=1$ or $i=2$. Since $\chi(M)=\deg(M)+1-g$, we find in all three cases
the bound stated for $\chi(L)$.

Now assume that $Rp_\ast(G)=0$, thus $G\in\CT$ is a sheaf. If $G\neq 0$, we may choose a section $j\colon C_0\hookrightarrow W$ in the linear system $|\CO_p(1)|$, such that $j^\ast G\neq 0$.There is an exact triangle
\[ G \to G\otimes \CO_p(1)\to j_\ast Lj^\ast\big(G\otimes \CO_p(1)\big)\,\]
and, since $Rp_\ast(G)=0$,
\[p_\ast\big(G\otimes \CO_p(1)\big)\cong p_\ast j_\ast Lj^\ast \big(G\otimes \CO_p(1)\big)\,,\]
By Lemma~\ref{lem: basic}, the latter surjects onto $p_\ast j_\ast j^\ast \big(G\otimes \CO_p(1)\big)$ which is non-zero since $C_0$ is a section of~$p$. Now apply the first part to $G\otimes \CO_p(1)$ to obtain a non-zero $p^\ast L \to G\otimes\CO_p(1)$ and twist by $\CO_p(-1)$.
\end{proof}

\subsection{Support (proof of Theorem~\ref{thm: rho})}\label{subsec: supp}
\begin{lemma}\label{lem: supp}
For all $T\in\CT$ there are $T',T''\in\CT$ and an exact sequence
    \[ 0 \to T' \to T \to T'' \to 0\,,\]
    such that 
    \begin{enumerate}
        \item $T'\in\Coh_{\leq 1}(X)$ and $\iota^\ast T'\in\Coh_0(W)$,
        \item $\supp(T'')_{red}\subset W$.
    \end{enumerate} 
\end{lemma}
\begin{proof}
Let $\supp(T)_{red}=Z\cup W'$ with $W'\subset W$, $\dim(Z)\leq 1$ and $Z\cap W$ empty or $0$-dimensional. By a standard argument, we can find such a sequence with $\supp_{red}(T')\subset Z$ and $\supp_{red}(T'')\subset W$, see e.g.\ \cite[Tag 01YD]{stacks-project}. Then, $T'\in\CT$ is immediate from the definition and $T''\in\CT$ since $\CT$ is closed under quotients.
\end{proof}

The rest of this section concerns sheaves with support contained in~$W$. Let $B\subset W$ be a fiber of the projection~$p$.
\begin{proposition}\label{prop: supp}
Let $T\in\CT$ and $F\in\CF$, then
\begin{enumerate}
    \item If $\supp(T)_{red}\subset W$, then $T\in\big\langle \CT\cap\iota_\ast\Coh(W)\big\rangle_\ex$,
    \item if $\supp(T)_{red}\subset B$, then $T\in\big\langle \CT\cap\iota_\ast\Coh(B)\big\rangle_\ex$,
    \item $F\in\big\langle \CF\cap\iota_\ast\Coh(W)\big\rangle_\ex$,
    \item if $\supp(F)_{red}\subset B$, then $F\in\big\langle \CF\cap\iota_\ast\Coh(B)\big\rangle_\ex$
\end{enumerate}
\end{proposition}

\begin{proof}[Proof of Proposition~\ref{prop: supp} (i), (ii)]
Let $T\in\CT$ with $\supp(T)_{red}\subset W$, then there is an exact sequence
\[ 0 \to T' \to T \to \iota_\ast\iota^\ast T\to 0\,,\]
with $T'$ a quotient of $T\otimes\CO_X(-W)$. Note that $\iota_\ast\iota^\ast T\in\CT$ as it's a quotient of $T$. It follows from Lemma~\ref{lem: ample} that $T\otimes\CO_X(-W)\in\CT$, thus $T'\in\CT$. The sequence implies $\ch^\omega(T')<\ch^\omega(T)$, see Notation~\ref{not: ch}. Since $\ch^\omega(T)=0$ if and only if $T=0$, we conclude by induction.\medskip

The analogous argument proves~(ii). By~(i) we may consider sheaves $\iota_\ast G\in\CT$ with $\supp(G)_{red}\subset B$. Let $j\colon B\hookrightarrow W$, then we have an exact sequence
\[ 0 \to G' \to G \to j_\ast j^\ast G \to 0\,,\]
with $G'$ a quotient of $G\otimes\CO_W(-B)$. Since $\iota_\ast\big(G\otimes\CO_W(-B)\big)\in\CT$ we can conclude as above.
\end{proof}

For the proofs of (iii) and (iv) we require the following results. Recall the Notation~\ref{not: R1p}.

\begin{lemma}\label{lem: F supp W}
For all $F\in\CF$ there are $F',F''\in\CF$ and an exact sequence
\[ 0\to F' \to F \to F'' \to 0\]
such that 
\begin{enumerate}
    \item $F''\cong\iota_\ast\iota^\ast F''$,
    \item $R^1p_\ast\iota^\ast F\xrightarrow{\sim}R^1p_\ast\iota^\ast F''$,
    \item $R^1p_\ast(\iota^\ast F')<R^1p_\ast(\iota^\ast F)$.
\end{enumerate}
\end{lemma}

\begin{proof}
Consider the restriction $F\twoheadrightarrow \iota_\ast\iota^\ast F$ and the decomposition
\[ 0 \to T \to \iota_\ast\iota^\ast F \to F'' \to 0\]
obtained from the torsion pair $(\CT\,,\CF)$. Since $\CF$ is closed under subobjects, we obtain the desired sequence of sheaves in~$\CF$. Property~(i) follows since $F''$ is a quotient of $\iota_\ast\iota^\ast F$. For (ii) note that $\iota^\ast F = \iota^\ast\iota_\ast\iota^\ast F$ and, as consequence of the definition of the torsion pair $(\CT\,,\CF)$, the map $\iota_\ast\iota^\ast F\twoheadrightarrow F''$ induces an isomorphism on $R^1p_\ast\iota^\ast$.\medskip

For~(iii) we consider the pullback $L\iota^\ast$ of the sequence and apply $Rp_\ast$ to obtain
\[ p_\ast L\iota^\ast F'' \to R^1p_\ast\iota^\ast F' \to R^1p_\ast\iota^\ast F \to R^1p_\ast\iota^\ast F'' \to 0\,.\]
The last map is an isomorphism, thus the first one must be surjective. By Lemma~\ref{prop: ell}, $p_\ast\iota^\ast F'' = 0$ and by Lemma~\ref{lem: basic}: 
\[p_\ast L\iota^\ast F''=R^1p_\ast(L\iota^{-1} F'') = R^1p_\ast(\omega_W^\vee\otimes \iota^\ast F'')\,. \]
Lemma~\ref{lem: ample} together with~(ii) implies~(iii).
\end{proof}

\begin{lemma}\label{lem: F supp B}
For all $\iota_\ast G\in\CF$ supported on finitely many fibers of~$p$, there exists a fiber $j\colon B_y\hookrightarrow W$ and $\iota_\ast G',\iota_\ast G''\in\CF$ and an exact sequence
\[ 0\to G' \to G \to G'' \to 0\]
such that 
\begin{enumerate}
    \item $G''\cong j_\ast j^\ast G''$,
    \item $R^1p_\ast G \otimes k(y)\xrightarrow{\sim}R^1p_\ast G''\otimes k(y)$,
    \item $R^1p_\ast G' < R^1p_\ast G$.
\end{enumerate}
\end{lemma}

\begin{proof}
The proof is parallel to the proof of Lemma~\ref{lem: F supp W}.
\end{proof}

\begin{proof}[Proof of Proposition~\ref{prop: supp} (iii), (iv)]
To prove (iii) we use Lemma~\ref{lem: F supp W} and induction to reduce to $F\in\CF$ with $R^1p_\ast(\iota^\ast F)=0$. But then $F\in\CF\cap\CT$, thus $F=0$. The analogous argument proves~(iv).
\end{proof}

\subsection{Zero-dimensional perverse sheaves (proof of Theorem~\ref{thm: rho})}\label{subsec: 0-dim perverse}

We use a generating set of objects with extension closure $\CA_0$ to prove Theorem~\ref{thm: rho} (i).
\begin{lemma}\label{lem: A0}
Denote the fibers of the projection by $B_y=p^{-1}(y)$, then
\begin{enumerate}
    \item $\CA_0\cap \CF[1] = \big\langle \{\CO_{B_y}(k)[1]: y\in C\,,  k\leq -2\}\big\rangle_\ex$,
    \item $\CA_0\cap \CT = \big\langle \Coh_0(X)\,, \{\CO_{B_y}(k): y\in C\,, k\geq -1\}\big\rangle_\ex$.
\end{enumerate}
\end{lemma}

\begin{proof}
By Proposition~\ref{prop: supp}~(iv), $\CA_0\cap \CF[1]$ is the extension closure of shifted sheaves~$G[1]$ supported on a single fiber~$j\colon B_y\hookrightarrow W$. Then $p_\ast j_\ast G=0$ by Lemma~\ref{prop: ell}, thus decomposing $G$ into a $0$-dimensional sheaf and a sum of line bundles we find that $G$ is torsion-free and only line bundles $\CO_{B_y}(k)$ with $k\leq -2$ appear.\medskip

For~(ii) use Lemma~\ref{lem: supp} and Proposition~\ref{prop: supp}~(ii) to reduce to $\Coh_0(X)$ and sheaves supported on some $j\colon B_y\hookrightarrow W$. Decomposing the latter into a sum of a $0$-dimensional sheaf and line bundles $\CO_{B_y}(k)$, we must have $k\geq -1$.
\end{proof}

Theorem~\ref{thm: rho} (i) now follows from Lemmas \ref{lem: 0-dim generators} and \ref{lem: A0}.

\subsection{One-dimensional perverse sheaves (proof of Theorem~\ref{thm: rho})}\label{subsec: 1-dim perverse}
Let $F\in\CF$. By Lemma~\ref{lem: F supp W} we may assume that $F\cong\iota_\ast G$ is supported on~$W$. The proof of Lemma~\ref{prop: ell} showed that we have an injection
\[ G \hookrightarrow \omega_p\otimes p^\ast V\,,\]
where $V=R^1p_\ast G$. Let $T$ be the cokernel. The inclusion induces an isomorphism on $R^1p_\ast(-)$, thus $\iota_\ast T\in\CT_{\leq 1}$. We have an exact triangle
\[ T \to G[1] \to \omega_p\otimes p^\ast V[1]\,.\]
To prove Theorem~\ref{thm: rho} (ii) it suffices to consider sheaves in~$\CT_{\leq 1}$ and objects of the form $\iota_\ast\big(\omega_p\otimes p^\ast V\big)[1]$. 

Recall that for any $E\in D^b(X)$ we have an exact triangle~\eqref{eq: triangle}
\[ E^\vee[2] \to \rho(E) \to \Phi\circ\Phi_R[1]\big(E^\vee[2]\big)\,.\]

We consider the long exact sequence of cohomology sheaves for the \emph{standard $t$-structure} associated to this triangle. Let $\CH^i=\CH^i\big(\rho(E)\big)$. Property~\eqref{eq: star} is equivalent to
\[\CH^{-1}\in\CF\,,\quad \CH^1\in \CT\cap\CA_0\,,\quad \CH^i= 0\,, \quad i\neq -1,0,1\,.\]
\begin{lemma}\label{lem: 1-dim generators}Let $V\in\Coh(C)$, then
\begin{enumerate}
    \item $\rho\Big(\iota_\ast\big(\omega_p\otimes p^\ast V\big)[1]\Big)$ satisfies Property~\eqref{eq: star},
    \item $\rho\Big(\iota_\ast\big(\CO_p(-1)\otimes p^\ast V\big)\Big)$ satisfies Property~\eqref{eq: star}.
\end{enumerate}
\end{lemma}

\begin{proof} Denote by $E=\iota_\ast\big(\omega_p\otimes p^\ast V\big)[1]$, then
\[E^\vee[2] = \iota_\ast\big(p^\ast(\omega_C\otimes V^\vee)\big)\,.\]
Note that $V^\vee=R\CHom(V,\CO_C)$ has cohomology sheaves
\[ \CH^0\big(V^\vee\big)\in\Coh_1(C)\,,\quad \CH^1\big(V^\vee\big)\in\Coh_0(C)\,,\quad \CH^i\big(V^\vee\big) = 0\,,\quad i\neq 0,1\,.\]
Then, we find that
\[ \CH^0\big(E^\vee[2]\big)\in\CT_{\leq 1}\,,\quad \CH^1\big(E^\vee[2]\big)\in\CT_{\leq 1}\cap\CA_0\,,\quad \CH^i\big(E^\vee[2]\big)=0\,, \quad i\neq 0,1\,.\]
Direct computation yields
\[\Phi_R[1]\big(E^\vee[2]\big) = p_\ast\big(\CO_p(1)\big)\otimes\omega_C\otimes V^\vee\,,\]
with cohomology sheaves
\begin{align*} &\CH^0\big(\Phi\circ\Phi_R[1](E^\vee[2])\big)\in\CT_{\leq 1}\,,\quad \CH^1\big(\Phi\circ\Phi_R[1](E^\vee[2])\big)\in\CT_{\leq 1}\cap\CA_0\,,\\
&\CH^i\big(\Phi\circ\Phi_R[1](E^\vee[2])\big)=0\,, \quad i\neq 0,1\,.
\end{align*}
Together this proves~(i). Recall the functor $\Phi\colon D^b(C)\to D^b(X)$ defined as
\[ \Phi(V) = \iota_\ast\big(\CO_p(-1)\otimes p^\ast V\big)\,.\]
For~(ii), let $E=\Phi(V)$, then
\[ E^\vee[2] = \iota_\ast\big(\CO_p(-1)\otimes p^\ast(\omega_C\otimes\det(\CE^\vee)\otimes V^\vee)\big)[1] = \Phi\big(\widetilde V[1]\big)\,,\]
where $\widetilde V = \omega_C\otimes\det(\CE^\vee)\otimes V^\vee$. Using $\Phi_R\circ\Phi \cong \id \,\oplus\, \omega_C[-2]$ we obtain a split exact triangle
\[ \rho(E) \to \Phi\big(\widetilde V[2] + \omega_C\otimes\widetilde V\big) \to \Phi(\widetilde V)[2]\,.\]
Thus, $\rho\big(E\big) \cong \Phi\big(\omega_C\otimes \widetilde V\big)$, which satisfies 
\begin{align*} &\CH^0\big(\Phi(\omega_C\otimes \widetilde V)\big)\in\CT_{\leq 1}\,,\quad \CH^1\big(\Phi(\omega_C\otimes \widetilde V)\big)\in\CT_{\leq 1}\cap\CA_0\,,\\
&\CH^i\big(\Phi(\omega_C\otimes \widetilde V)\big)=0\,, \quad i\neq 0,1\,.\qedhere
\end{align*}
\end{proof}

\begin{proposition}\label{prop: rho(1-dim)} For all $E\in\Coh_{\leq 1}(X)\cap\CT$ the image $\rho(E)$ satisfies Property~\eqref{eq: star}.
\end{proposition}

\begin{proof}
We decompose $E$ with respect to the torsion pair $(\CA_0\,,\CA_1)$ of $\CA_{\leq 1}$. The $\CA_0$ part is covered by Theorem~\ref{thm: rho} (i). Thus, assume 
\[E\in\Coh_{\leq 1}(X)\cap\CT\cap\CA_1\,,\]
in particular $E\in\Coh_1(X)$. We apply Lemma~\ref{lem: supp} to $E$. First assume that $\iota^\ast E\in \Coh_0(W)$. It follows from purity of~$E$ that $L\iota^\ast E = \iota^\ast E$. Dualizing, we have $E^\vee[2]\in\Coh_1(X)\cap\CA_1$ and
\[L\iota^\ast \big(E^\vee[2]\big) = \iota^\ast \big(E^\vee[2]\big)\in \Coh_0(W)\,.\]
We have the exact triangle~\eqref{eq: triangle}
\[ E^\vee[2] \to \rho(E) \to \Phi\circ\Phi_R[1]\big(E^\vee[2]\big)\,.\]
The left and right objects are sheaves in~$\CT_{\leq 1}$, thus $\rho(E)\in\CT_{\leq 1}$ as well.

It remains to prove Property~\eqref{eq: star} for sheaves $E=\iota_\ast G\in \Coh_{1}(X)$. Let $\CH^i=\CH^i\big(\rho(E)\big)$, we must prove that 
\[\CH^{-1}\in\CF\,,\quad \CH^1\in \CT\cap\CA_0\,,\quad \CH^i= 0\,, \quad i\neq -1,0,1\,.\]
Note that $Rp_\ast(L\iota^\ast E^\vee[2])=p_\ast(L\iota^\ast E^\vee[2])$ lies in $D^{[-1,0]}(C)$, thus
\[ \CH^i\big(\Phi\circ\Phi_R[1](E^\vee[2])\big)=0\,,\qquad i\neq -1,0\,.\]
Thus, in fact $\CH^i =0$ for $i\neq -1,0$ from the long exact sequence. We must argue that $\CH^{-1}\in\CF$. Note that $E^\vee[2] = \iota_\ast(G^\vee\otimes\omega_W[1])$. We have an exact sequence of sheaves
\[ 0 \to \CH^{-1} \to \CO_p(-1)\otimes p^\ast p_\ast \big(G^\vee\otimes\omega_W[1]\otimes\CO_p(1)\big)\to \iota_\ast(G^\vee\otimes \omega_W[1])\,.\]
Let $L\in\Pic(C)$ be a line bundle. Since $Rp_\ast \CO_W = \CO_C$ we have 
\begin{align*} &\Hom\Big(\CO_p(-1)\otimes p^\ast L, \CO_p(-1)\otimes p^\ast p_\ast \big(G^\vee\otimes\omega_W[1]\otimes\CO_p(1)\big)\Big)\\
&\cong \Hom\Big(L, p_\ast \big(G^\vee\otimes\omega_W[1]\otimes\CO_p(1)\big)\Big)\\
&\cong \Hom\Big(\CO_p(-1)\otimes p^\ast L, \iota_\ast(G^\vee\otimes \omega_W[1])\Big)\,.
\end{align*}
Thus,
\[ \Hom\Big(\CO_p(-1)\otimes p^\ast L, \CH^{-1}\Big)=0\,,\]
and $\CH^{-1}\in\CF$ by Lemma~\ref{lem: morphism}. \end{proof}
\begin{lemma}\label{lem: generatorsT}\hfill
    \[ \CT_{\leq 1} = \Big\langle\Coh_{\leq 1}(X)\cap\CT\,, \Phi\big(\Coh_1(C)\big) \Big\rangle_\ex\,.\]
\end{lemma}
\begin{proof}
The inclusion ``$\supset$" is immediate. By Lemma~\ref{lem: supp} and Proposition~\ref{prop: supp} we know that $\CT_{\leq 1}$ is the extension closure of sheaves $T\in\Coh_{\leq 1}(X)$ with $\iota^\ast T\in\Coh_0(W)$ and pushforwards~$\iota_\ast(G)\in\CT_{\leq 1}$. Thus, it suffices to consider sheaves~$T=\iota_\ast G$. Let
\[ T_0 \to T \to T_1\]
be the decomposition in~$\CA$ with respect to the torsion pair~$(\CA_0\,,\CA_1)$. Since $T\in\Coh(X)$ we have $\Hom(\CF[1],T)=0$, thus 
\[T_0\in\Coh(X)\cap\CA_0=\CT_0\,.\]
This category is closed under quotients. Thus, replacing~$T_0$ by its image, we may assume that $T_0\to T$ is an injection of sheaves. It follows that $T_1\in\Coh(X)\cap\CA_1=\CT_1$. We have $\CT_0\subset\Coh_{\leq 1}\cap\CT$, thus we may assume $\iota_\ast G\in\CT_1$.\medskip

By Lemma~\ref{lem: morphism} there is a line bundle~$L$ and a non-zero morphism
\[ \CO_p(-1)\otimes p^\ast L \to G\,.\]
Taking the image and cokernel of this map, we obtain an exact sequence of sheaves in $\CT_{\leq 1}$
\[ 0 \to \iota_\ast G' \to \iota_\ast G \to \iota_\ast G''\to 0\,,\]
such that $0\neq\iota_\ast(G')\in\CA_1$. If
\[ \CO_p(-1)\otimes p^\ast L \twoheadrightarrow G'\]
is an isomorphism, then $\iota_\ast G'\in \Phi\big(\Coh_1(C)\big)$. Otherwise, $G'$  has dimension at most one. By Proposition~\ref{prop: ell}~(i) we have\footnote{Here we use $\ell(-)$ as defined in Section~\ref{subsec: Nironi stability}. The properties proved in Proposition~\ref{prop: ell} do not depend on Lemma~\ref{lem: generatorsT}.} $\ell(\iota_\ast G') > 0$, thus 
\[\ell(\iota_\ast G)>\ell(\iota_\ast G'')\geq 0\,.\]
By Proposition~\ref{prop: ell}~(iii) we have $\ell(\iota_\ast G'')=0$ if and only if $\iota_\ast G''\in\CA_0$, so we can conclude by induction.
\end{proof}
\begin{proposition}
For all $T\in\CT_{\leq 1}$ the image~$\rho(T)$ satisfies Property~\eqref{eq: star}.
\end{proposition}

\begin{proof}
Follows from Lemma~\ref{lem: 1-dim generators}, Proposition~\ref{prop: rho(1-dim)}, and Lemma~\ref{lem: generatorsT}.
\end{proof}

\begin{proof}[Proof of Theorem~\ref{thm: rho} (ii)]
The results of this section imply that for all $E\in\CA$ the image $\rho(E)$ satisfies Property~\eqref{eq: star}. Let $E\in\CA_1$ and $Q\in\CA_0$. Then, $\rho(Q)\in\CA_0[-1]$ by Theorem~\ref{thm: rho} (i) and, by purity of $E$,
\[ \Hom(\rho(E), Q[-1]) = \Hom(\rho(Q)[1], E)=0\,.\]
Thus, $\rho(E)\in\CA_{\leq 1}$. But then $\rho(E)\in\CA_1$ because
\[ \Hom(Q,\rho(E)) = \Hom(E,\rho(Q))=0\,,\]
since $\Hom^k(E,F)=0$ for all $E,F\in\CA$ and $k<0$. 
\end{proof}

\section{Hall algebras, pairs, and wall-crossing}\label{sec: Hall}

\subsection{Numerical Grothendieck groups}
The numerical Grothendieck group $N(X)$ is the Grothendieck group of $D^b(X)$ modulo the Euler paring. We will tacitly use the injection into the even cohomology via the Chern character. The class $[E]\in N(X)$ is equivalently characterised by
\[\big(\ch_0(E), \ch_1(E), \ch_2(E), \chi(E)\big)\,.\]
The numerical Grothendieck group admits a dimension filtration~$N_{\leq k}(X)$. For our purposes, we define $N_{\leq k}$ as the numerical Grothendieck group of~$\CA_{\leq k}$. We will only consider objects of perverse dimension~$\leq 1$. Explicitly,
\[N_0=\BZ\cdot\sfb\oplus \BZ\cdot\pt\,,\quad N_{\leq 1}=\BZ\cdot\sfw\oplus N_{\leq 1}(X)\,,\]
where $\sfb$ and $\sfw$ are the classes of a fiber resp.\ the divisor as introduced in Section~\ref{subsec: overview} and $N_0(X)\cong \BZ$ is spanned by the point class~$\pt$. We also define $N_1=N_{\leq 1}/N_0$ and we choose a splitting 
\[N_{\leq 1}=N_0\oplus N_1\,.\]
An element $\alpha\in N_{\leq 1}$ can be written as 
\[\alpha=(\gamma, c)=(r\sfw, \beta+j\sfb, n)\]
where $\gamma=(r, \beta)\in N_1$ and $c=(j, n)\in N_0$.

We will consider various generating series of DT invariants using the Novikov parameter $z$ of $\BQ[[N_{\leq 1}]]$ and we use the notation 
\[Q=z^\sfb\,,\quad -q=z^{\pt}\,,\quad t=z^{[\CO_X]}\,.\]
In particular, for $\alpha$ as above $z^\alpha=z^\gamma\, Q^j\, (-q)^n$.

\subsection{Hall algebra}

We briefly recall the notion of Hall algebras following~\cite{To20}. Let $\CC\subset D^b(X)$ be the heart of a bounded $t$-structure. In our applications we use two different hearts: 
\[\CC=\langle \Coh_{\geq 2}[1], \Coh_{\leq 1}\rangle\textup{ and }\CC=\langle \CA_{\geq 2}[1],\CA_{\leq 1}\rangle\,.\]
The first is used to define $\PT$ and $\bs$ invariants, the second is used to define $\PPT$ invariants. Both of these hearts are open by \cite[Lemma~4.1]{BCR18} so they satisfy the technical hypothesis in \cite[Appendix B]{BCR18}, \cite[Section 3]{BR19}.

The objects of $\CC$ form an algebraic stack which we still denote by $\CC$ and we assume that it is an open substack of the stack $\CM$ of objects 
\[\big\{E\in D^b(X): \Ext^{<0}(E,E)=0\big\}\,.\] The Hall algebra $\HH(\CC)$ is the $\BQ$-vector space generated by maps of algebraic stacks $[\CZ\to \CC]$, where $\CZ$ is an algebraic stack of finite type with affine stabilizers, modulo some motivic relations described in~\cite{To20}.

The Hall algebra $\HH(\CC)$ admits a product induced by extensions and, via cartesian products, is a module over $K(\textup{St}/\BC)$, the Grothendieck ring of stacks with affine stabilizers. Equivalently,
\[K(\textup{St}/\BC)=K(\textup{Var}/\BC)[\BL^{-1}, (\BL^n-1)^{-1}]\]
where $\BL=[\BA^1\to \mathrm{Spec}\,\BC]$.
The decomposition 
\[\CC=\coprod_{\alpha\in N(X)}\CC_\alpha\]
into numerical classes induces a decomposition of the Hall algebra
\[\HH(\CC)=\bigoplus_{\alpha}\HH_\alpha(\CC).\]

The feature of most interest in the Hall algebra is the existence of the integration map. To state this we introduce two more definitions. We let
$\HH^{\textup{reg}}(\CC)\subset \HH(\CC)$ be the $K(\textup{Var}/\BC)[\BL^{-1}]$-submodule spanned by $[Z\to \CC]$ so that $Z$ is a variety and 
\[\HH^{\textup{sc}}(\CC)=H^{\textup{reg}}(\CC)/(\BL-1)\HH^{\textup{reg}}(\CC).\]
This has the structure of a Poisson algebra. The integration map maps $\HH^{\textup{sc}}(\CC)$ to the Poisson torus 
\[\QQ[N(X)]=\bigoplus_{\alpha\in N(X)}\BQ z^\alpha.\]
The Poisson torus has the structure of a Poisson algebra as well. Its bracket is defined by
\[\{z^{\alpha}, z^{\alpha'}\}=(-1)^{\chi(\alpha, \alpha')}\chi(\alpha, \alpha')z^{\alpha+\alpha'}.\]

\begin{theorem}[{\cite[Theorem 2.8]{To20}}]
\label{thm: integrationmap}
There is a Poisson algebra homomorphism 
\[I\colon \HH^{\textup{sc}}(\CC)\to \QQ[N(X)]\]
such that if $Z$ is a variety and $f\colon Z\to \CC_\alpha\hookrightarrow \CC$ then
\[I([Z\overset{f}{\to} \CC])=\left(\int_{Z}f^\ast \nu_{\CC}\right)z^\alpha\]
where $\nu_{\CC}$ is the Behrend function on the stack $\CC$.
\end{theorem}

The Hall algebra can be enlarged to the graded pre-algebra $H^{\textup{gr}}(\CC)$ by defining its generators to be $[\CZ\to \CX]$ with $\CZ$ being an algebraic stack with affine stabilizers such that $\CZ_\alpha$ is of finite type for every $\alpha\in N(X)$ (instead of asking that $\CZ$ is already of finite type). One can define analogous versions $H^{\textup{gr,reg}}(\CC),H^{\textup{gr,sc}}(\CC)$. The integration map extends to
\[I: H^{\textup{gr,sc}}(\CC)\to \BQ\{N(X)\}.\]

\subsection{Pairs}
\label{subsec: pairs}
We consider various notions of stable objects in $D^b(X)$ and their associated generating series. All of them are defined via a pair of subcategories $(\CT,\CF)$ of either $\Coh_{\leq 1}$ or $\CA_{\leq 1}$. We consider the categories
\[ \CB=\Big\langle \CO_X[1]\,,\Coh_{\leq 1}\Big\rangle_\ex\,,\quad {}^p\CB=\Big\langle \CO_X[1]\,,\CA_{\leq 1}\Big\rangle_\ex\,.\]

\begin{definition}[{\cite[Definition 3.9]{BCR18}}]\label{def: pair}
An object $P\in\CB$ or $P\in{}^p\CB$ is called a $(\CT,\CF)$-pair if
\begin{enumerate}
    \item $\rk(P)=-1$,
    \item $\Hom(T,P)=0$ for all $T\in\CT$,
    \item $\Hom(P,F)=0$ for all $F\in\CF$.
\end{enumerate}
\end{definition}

In Section~\ref{sec: BS} we consider BS and PT pairs which are defined in~$\CB$. Sections~\ref{sec: perversept} and~\ref{sec: zetawallcrossing} concern pairs defined in~${}^p\CB$. The categories $(\CT,\CF)$ arise in two ways:
\begin{enumerate}
    \item As torsion pairs associated to a stability function, or
    \item in the passage of one torsion pair to another, i.e.\ when crossing a wall.
\end{enumerate}

In the former case, the stability function is $\nu$ in Section~\ref{sec: perversept} and $\zeta$ in Section~\ref{sec: zetawallcrossing}. In the latter case, given two torsion pairs $(\CT_{\pm},\CF_{\pm})$ on different sides of a wall (and sufficiently close to the wall), we consider $(\CT_{+},\CF_{-})$. Joyce's wall-crossing formula yields the comparison between pairs on either side of the wall via semistable objects on $\CW=\CT_-\cap\CF_+$.

The notion of $(\CT,\CF)$-pairs with fixed numerical class $\alpha\in N(X)$ defines a stack $\Pairs(\CT,\CF)_\alpha$ which is of finite type in all of our applications and defines an element in the Hall algebra (Lemmas \ref{prop: finitetype}, \ref{prop: finitetypezeta} and \ref{prop: finitetypebs}).

\subsection{Joyce's wall-crossing formula}
\label{subsec: wallcrossing}
Let $(\CT_\pm,\CF_\pm)$ be two torsion pairs and $\CW=\CT_-\cap\CF_+$ be as above. When all the terms are defined, we have an identity in the Hall algebra
\[[\CW]\ast [\Pairs(\CT_-, \CF_-)]=[\Pairs(\CT_+, \CF_+)]\ast[\CW]\,.\]

The ``no-poles'' theorem by Joyce \cite[Theorem 8.7]{Jo07} and Behrend-Ronagh \cite[Theorems 4, 5]{BR19} tells us that in adequate conditions 
\[(\BL-1)\log(\CW)\in H^{\textup{gr,sc}}(\CC)\]
and, therefore, 
\[w=I\big((\BL-1)\log(\CW)\big)\in \BQ\{N(X)\}\]
is well-defined. The conditions that guarantee this are the following:
\begin{enumerate}
    \item $\CW_\alpha$ is an algebraic stack of finite type,
    \item $\CW$ is closed under extensions and direct summands,
    \item for every $\alpha\in N(X)$ there are finitely many ways to decompose $\alpha=\alpha_1+\ldots+\alpha_n$ such that $\CW_{\alpha_i}\neq \emptyset$.
\end{enumerate}
When all these conditions are satisfied (including the moduli of pairs defining elements in the Hall algebra), we say that the pairs $(\CT_\pm, \CF_\pm)$ are wall-crossing material. When this happens, we have Joyce's wall-crossing formula which we will repeatedly use:
\[I\big((\BL-1)\Pairs(\CT_+,\CF_+)\big) = \exp\big(\{w,-\}\big)\circ I\big((\BL-1)\Pairs(\CT_-,\CF_-)\big)\,.\]

\subsection{Rational functions}
In this paper we repeatedly encounter series expansions of rational functions 
\[f\in \BQ(N_0)=\BQ(q,Q)\,.\]
The ``direction'' of the expansion will play an important role, especially in the $\zeta$-wall-crossing in Section \ref{sec: zetawallcrossing}. We make here precise what ``direction'' means.

Given a non-zero linear function $L\colon N_0\to \BR$, we say that a set $S\subset N_0$ is $L$-bounded if for every $M\in \BR$, the set
\[\#\{c\in S\colon L(c)<M\}\]
is finite. Given $L$, we can define a completion $\BQ[N_0]_L$ of $\BQ[N_0]$ to be the set of formal power series
\[\sum_{c\in N_0}a_c z^c\]
such that $\{c\colon a_c\neq 0\}$ is $L$-bounded. The product of power series is well-defined in this completion. Given a rational function $f=g/h$ with $g,h\in \BQ[N_0]=\BQ[q,Q]$, we say that $F\in \BQ[N_0]_L$ is the expansion of $f$ with respect to $L$ if $hF=g$ in the ring $\BQ[N_0]_L$.

We briefly go over the different choices of $L$ used throughout the paper and clarify the statements of our results. The series $\PT_\beta$ for usual stable pairs invariants or $\bs_\beta$ for Bryan--Steinberg invariants (see Section \ref{sec: BS}) can be defined in the completion $\BQ[N_0]_L$ where 
\[L(j,n)=L(j\sfb, n)=j+\varepsilon\, n\] for $0<\varepsilon\ll 1.$

The generating series of perverse stable pairs $\PPT_\gamma$ is defined in the completion $\BQ[N_0]_d$ where
\[d(j,n)=2n+j.\]
In particular, the precise formulation of Theorem \ref{thm: perverse} is that $\PPT_\gamma$ is the expansion of the rational function $f_\gamma$ with respect to $d$. Theorem \ref{thm: wall-crossing} is to be understood in $\BQ(q,Q)$: the left and right hand side are the expansions of the same rational function in different directions.

This re-expansion in different directions is fundamental in Section \ref{sec: zetawallcrossing}. There, we will introduce series $\PDT^{\zeta, (\mu, \infty)}_\gamma$ that interpolate between each side of Theorem \ref{thm: wall-crossing}: they will be the expansion of the same rational function $f_\gamma$ with respect to 
\[L_\mu(j, n)=2n+j+\frac{j}{\mu a_0}.\]
Note that $L_\mu$ for $\mu\gg 1$ is equivalent to $d$ and for $\mu\ll 1$ it is equivalent to the linear function used for $\PT$ or $\bs$.


\section{Stability}\label{sec: stability}
We use three different stability functions to define stable pairs and study their wall-crossing:
\begin{enumerate}
    \item For Bryan--Steinberg type stable pairs in Section~\ref{sec: BS} we use
    \[\mu^A\colon \Coh_{\leq 1}(X)\ssetminus\{0\} \to (-\infty,+\infty]\times(-\infty,+\infty]\,.\]
    \item For perverse stable pairs in Section~\ref{sec: perversept} we use
    \[\nu\colon \CA_{\leq 1}\ssetminus\{0\}\to (-\infty, +\infty]\,.\]
    \item For the $\bs/\PPT$ wall-crossing in Section~\ref{sec: zetawallcrossing} we use
    \[\zeta\colon \CA_{\leq 1}\ssetminus\{0\}\to (-\infty, +\infty]\times(-\infty, +\infty]\,.\]
    \end{enumerate}
We comment on~(i) in Section~\ref{subsec: BS stability}. The necessary  results about $\mu^A$-stability were proved by Bryan--Steinberg~\cite{BS16} and require only minor modification for our setting. For~(ii) we give full proofs in Sections~\ref{subsec: Nironi stability},~\ref{subsec: curveclasses}, and~\ref{subsec: boundedness}. We also observe in Section~\ref{subsec: weakstability} that $\CA_{\leq 1}$ and $\nu$-stability can be obtained from a weak stability condition through a tilting process. Finally, for (iii) we can employ the techniques used for (ii) in a similar way to study $\zeta$-stability. We briefly comment on this in Section~\ref{subsec: refined stability}.

\subsection{Bryan--Steinberg stability}\label{subsec: BS stability}

Let $Y$ be the coarse moduli space of an orbifold CY-3 satisfying the \emph{hard Lefschetz condition} and let
\[ \pi\colon X \to Y\]
be the distinguished crepant resolution~\cite{BKR01, CT08}. Denote by $\widetilde H\in\Nef(X)$ the pullback of an ample class on~$Y$, and let $\omega\in\Amp(X)$ be ample such that $\omega -\widetilde H$ is ample as well. Bryan--Steinberg~\cite{BS16} introduce a function on $\Coh_{\leq 1}(X)$ defined as 
\[\mu^\pi(E) = \Bigg(\frac{\chi(E)}{\widetilde H\cdot \ch_2(E)}\,,\frac{\chi(E)}{\omega \cdot\ch_2(E)}\Bigg)\,.\]
They are able to prove the necessary technical results~\cite[Theorem 38, Lemma 47, Lemma 51]{BS16} which allow to employ Joyce's Hall algebra machinery. We can use the exact same pathway. Critically, we do not require a projection $X\to Y$, the existence of a nef class $A\in\Nef(X)$ as described in the~condition~\eqref{eq: nef} suffices. We then define $\mu^A$ by the same formula as $\mu^\pi$, replacing $\widetilde H$ by $A$. The proofs in~\cite{BS16} carry over to our setting where a projection $\pi$ does not necessarily exist:

\begin{proposition}[\cite{BS16}]
\label{prop: BSstabilityfinitetype}
The slope $\mu^A$ defines a stability condition on $\Coh_{\leq 1}(X)$. Moreover, the moduli stack of $\mu^A$-semistable sheaves $\CM^{\mu^A}_{(\beta, n)}$ is a finite type open substack of the moduli stack $\CM$ parametrizing perfect complexes $E\in D^b(X)$ with $\Ext^{<0}(E,E)=0$.
\end{proposition}
\begin{proof}
As we pointed out already, the proofs of Theorem 38 and Lemma 47 carry over verbatim to show that $\mu^A$ is a stability condition and that the family of sheaves in $\CM^{\mu^A}_{(\beta, n)}$ is bounded. The fact that $\CM^{\mu^A}_{(\beta, n)}$ is a finite type open substack of $\CM$ then follows from \cite[Theorem 3.20]{To08}.\qedhere
\end{proof}

\subsection{Nironi stability}\label{subsec: Nironi stability}

Recall the nef class $A\in\rm{Nef}(X)$ and $a_0\in\BZ_{>0}$ such that $\iota^\ast A$ is numerically equivalent to~$a_0\sfb$. Let $g$ be the genus of the curve~$C$. For $E\in\CA_{\leq 1}$ with 
\[\big(\ch_1(E), \ch_2(E), \chi(E)\big) =\big(r\sfw,\beta, n\big)\]
define the slope $\nu\colon \CA_{\leq 1}\ssetminus\{0\}\to \BQ\cup\{+\infty\}$ as $\displaystyle \nu(E)=\frac{d(E)}{\ell(E)}$, where
\begin{align*} 
    d(E)&=r(1-g)+2\,n-\frac{1}{2}\,\sfw\cdot\beta\,,\\
    \ell(E)&=2\,A\cdot\beta + r\, a_0\,.
\end{align*}
Note that for $G\in\Coh(W)$, by Grothendieck--Riemann--Roch
\[A\cdot\ch_2(\iota_\ast G)= a_0\,\rk(Rp_\ast G)\,.\]
In the crepant case, the class $A$ can be taken as the pullback of an ample class from the coarse moduli space~$Y$ and the stability matches the notion of Nironi's slope stability~\cite{Ni08} on $\Coh_{\leq 1}(\CY)$.

Recall that Nironi's slope stability is defined in the analogous way, using a self-dual generating bundle~$V$ and the modified Hilbert polynomial
\[ p_E(k) = \chi\big(V,E\otimes \CO_X(A)^k\big) = \ell(E)\,k+d(E)\,.\]
Our definition resembles this notion replacing $V$ by the $\rho$-invariant $K$-theory class of $\CO_X\oplus\CO_X(W/2)$ and replacing the Euler pairing by the Mukai pairing.
\begin{example}
To illustrate $\nu$ for zero-dimensional perverse sheaves, consider a skyscraper sheaf $k(x)$ and the perverse sheaves $\CO_B(-2)[1]$ and $\CO_B(-1)$ supported on a fiber $B=p^{-1}(y)$. In the crepant case, these objects correspond to a non-stacky point, and the stacky points $\CO^+_y$ and $\CO^-_y$ respectively~\cite[Section 4.3]{BCY12}. In all three cases $\ell(-)=0$ and the computation for $d(-)$ is
\begin{align*}
    d(k(x))&=0+2-0=2\,,\\
    d(\CO_B(-2)[1])&=-d(\CO_B(-2))=-(0-2+1)=1\,,\\
    d((\CO_B(-1))&=0+0+1=1\,.
\end{align*}
\end{example}
\begin{proposition}\label{prop: ell}\hfill
\begin{enumerate}
    \item For all $T\in\CT_{\leq 1}$ set-theoretically supported on $W$ we have $\ell(T)\geq 0$, with equality if and only if $T\in\CT_0$.
    \item For all $F\in\CF_{\leq 1}$ we have $\ell(F)\leq 0$, with equality if and only if $F\in \CF_0$.
    \item For all $E\in\CA_{\leq 1}$ we have $\ell(E)\geq 0$, with equality if and only if $E\in\CA_0$. In that case, $d(E)\geq 0$, with equality if and only if $E=0$.
\end{enumerate}
\end{proposition}
\begin{proof}
For (i) and (ii) we may apply Proposition~\ref{prop: supp} and assume that $T$ and $F$ are scheme-theoretically supported on~$W$, i.e.\ we consider pushforwards~$\iota_\ast G$ with $G\in\Coh(W)$.

(i) Let $\iota_\ast G\in \CT_{\leq 1}$. Since $R^1p_\ast G=0$ we have $\rk(Rp_\ast G)\geq 0$, thus both summands of $\ell(\iota_\ast G)$ are non-negative, and $\ell(\iota_\ast G)=0$ if and only if $r=0$ and $A\cdot\ch_2(\iota_\ast G)=0$, thus $\iota_\ast G\in\CT_0$.\medskip

(ii) Let $\iota_\ast G\in \CF$. We claim that $r\leq \rk(R^1p_\ast G)$. Let $V=R^1p_\ast G$ and consider the map $Rp_\ast G \to V[-1]$ which lifts to $G\to \omega_p\otimes p^\ast V$. Let $\Kernel$ and $\Image$ be the kernel and image, i.e.\ 
\[ \Kernel \to G \to \Image\,.\]
Since $\CF$ is closed under subobjects, $\Kernel\in\CF$. Since $\Image\subset \omega_p\otimes p^\ast V$ we have $p_\ast\Image =0$. The isomorphism
\[ R^1p_\ast G \cong V\]
factors through $R^1p_\ast\Image$. We find that $R^1p_\ast\Kernel=0$, thus 
\[\iota_\ast \Kernel\in\CF\cap\CT=0\]
and $G\hookrightarrow \omega_p\otimes p^\ast V$ which implies $r\leq \rk(R^1p_\ast G)$ by comparing ranks. Since $p_\ast G=0$ by Lemma~\ref{prop: ell}, this implies
\[ \ell(\iota_\ast G) = -2a_0\,\rk(R^1p_\ast G) + r\,a_0 \leq - a_0\,\rk(R^1p_\ast F)\leq 0\,.\]
From this we get $\ell(\iota_\ast G)=0$ if and only if $\rk(R^1p_\ast G)=0$ and then $r=0$, thus $\iota_\ast G\in \CF_0$.

(iii) By Lemma~\ref{lem: supp}, Proposition~\ref{prop: supp}, and (i)-(ii) it remains to consider $E\in \Coh_{\leq 1}(X)$ such that $\iota^\ast E\in \Coh_0(W)$. Then, by condition~\eqref{eq: nef}
\[\ell(E)=2A\cdot\ch_2(E)\geq 0\,,\]
with equality if and only if~$\ch_2(E)\in \BZ_{\geq 0}\cdot\sfb$. Since $\sfw\cdot\ch_2(E)\geq 0$, whereas $\sfw\cdot\sfb=-2$, we must in fact have $\ch_2(E)=0$, i.e.\ $E$ is a $0$-dimensional sheaf.

For the positivity of $d(-)$ on $\CA_0$ we may use Lemma \ref{lem: A0}. If $E\in \Coh_0(X)$ then $d(E)=2\chi(E)\geq 0$. Moreover, we can compute directly 
\begin{align*}d(\CO_B(k))&=2k+3> 0\textup{ for }k\geq -1\\
d(\CO_B(k)[1])&=-(2k+3)> 0\textup{ for }k\leq -2.\qedhere\\
\end{align*}
\end{proof}

\begin{proposition}\label{prop: nustability}
The slope $\nu$ defines a stability condition on $\CA_{\leq 1}$:
\begin{enumerate}
    \item $\nu$ satisfies the see-saw property,
    \item Harder--Narasimhan filtrations exist.
\end{enumerate}
\end{proposition}
\begin{proof}
The proof of the see-saw property is standard, so it's enough to prove that $\CA_{\leq 1}$ is $\nu$-Artinian.

Suppose that $E_1\supseteq E_2\supseteq \ldots$ in $\CA_{\leq 1}$. Then $\ell(E_{i})$ is a decreasing sequence bounded below by $0$, so it must stabilize. Thus, for large enough $i$ the cone $C(E_{i+1}\to E_i)\in \CA$ must be in $\CA_0$ so $\nu(E_{i+1})\leq \nu(E_i)$.
\end{proof}

\begin{proposition}\label{prop: nusymmetry}
The slope~$\nu$ satisfies \[\nu\big(\rho(E)\big)=-\nu(E)\,,\quad \nu\big(E\otimes \CO_X(A)\big)=\nu(E)+1\,.\]
\end{proposition}
\begin{proof}
The equality $\ell\big(\rho(E)\big)=\ell(E)$ is clear since $A\cdot \sfb=0$. Using Proposition~\ref{prop: actionCoh} we have
\begin{align*}d\big(\rho(E)\big)&=r(1-g)-2n-\frac{1}{2}\,\sfw\cdot \big(\beta+(\sfw\cdot\beta -2r)\, \sfb\big) \\
&=-r(1-g)-2n+\frac{1}{2}\,\sfw\cdot \beta =-d(E)\,.
\end{align*}

For the second equality, a computation using $A^2\cdot \sfw=0$ shows that
\[\ell\big(E\otimes \CO_X(A)\big)=\ell(E)\,,\quad d\big(E\otimes \CO_X(A)\big)=d(E)+\ell(E)\,.\qedhere\]
\end{proof}
\begin{definition}
An object $E\in\CA_{\leq 1}$ is called $\nu$-stable (resp.\ semistable), if for all non-trivial subobjects $F\to E$ in $\CA_{\leq 1}$ we have $\nu(F)<\nu(E)$ (resp.\ $\nu(F)\leq\nu(E)$).
\end{definition}
The following lemma will be useful in Section~\ref{subsec: moduli}.
\begin{lemma}\label{lem: stable}
Let $L\in\Pic(C)$, then
\begin{enumerate}
    \item $\iota_\ast\big(\CO_p(-1)\otimes p^\ast L\big)$ is $\nu$-stable of slope $\chi(L)+\frac{1}{2}\deg(\CE)+1-g$,
    \item $\iota_\ast\big(\omega_p\otimes p^\ast L[1]\big)$ is $\nu$-stable of slope $\chi(L)$.
\end{enumerate}
\end{lemma}
\begin{proof}
Let $E$ be the object in (i) or (ii). Note that $\ell(E)=1$ in both cases. From the description of $\CA_0$ in Lemma~\ref{lem: 0-dim generators} we see that $E$ is torsion-free in $\CA_{\leq 1}$, i.e.\ $\Hom(\CA_0,E)=0$. Let
\[ E' \to E \to E''\]
be an exact triangle in $\CA_{\leq 1}$. Then $\ell(E')=1$ and $\ell(E'')=0$, therefore $d(E'')\geq 0$ with equality if and only if $E''=0$. Thus, either $\nu(E')<\nu(E)$ or $E'=E$. The slope $\nu(E)$ is easily computed.
\end{proof}

\subsection{Curve classes}\label{subsec: curveclasses}
We denote by $N_1^{\textup{eff}}$ the image of $\CA_{\leq 1}$ in $N_1$. By Lemma~\ref{lem: supp} and Proposition~\ref{prop: supp}, $N_1^\eff$ is the cone generated by classes $[E]$ where $E$ is from one of the three sets
\begin{align*}    
S_1&=\big\{E\in \Coh_{\leq 1}(X)\colon \iota^\ast E\in \Coh_0(W)\big\}\,,\\
S_2&=\CT\cap \iota_\ast \Coh(W)\,, \\
S_3&=\big(\CF\cap \iota_\ast \Coh(W)\big)[1]\,. 
\end{align*}

Let $\Delta\subset N_1^\eff$ be the cone generated by classes form $S_2$ and $S_3$.
\begin{lemma}\label{lem: effectivefinite}
 For any $l>0$, the set
 \[\{\gamma\in\Delta: \ell(\gamma)\leq l\}\]
is finite.
\end{lemma}
\begin{proof}
It suffices to prove the claim for classes~$[E]$ with $E$ from either set $S_2$ or $S_3$. Consider $\iota_\ast G\in S_2$. Recall that $\ell(\iota_\ast G)=a_0\left(2\rk(Rp_\ast G)+r\right)$ and, because $\iota_\ast G\in \CT$, we have $\rk(Rp_\ast G)=\rk(p_\ast G)\geq 0$. So there are only finitely many possibilities for $r$ and for $A\cdot \ch_2(\iota_\ast G)$. Since $N_1(W)$ has rank~$2$, the map
\[N_1(W)_\BQ\,/\,\BQ\cdot \sfb \stackrel{A\cdot }{\longrightarrow}\BQ\]
is an isomorphism, showing that there are finitely many possibilities for $\ch_2(\iota_\ast G)$ in $N_1(X)/\BZ\cdot\sfb$.

The argument for $S_3$ is similar to $S_2$. Indeed, for $\iota_\ast G[1]\in S_3$ Lemma~\ref{prop: ell}~(ii) implies that
\[-A\cdot \ch_2(\iota_\ast(G))=-a_0 \rk(Rp_\ast G)=a_0 \rk(R^1p_\ast G)\]
is bounded (recall that $p_\ast G=0$ by Lemma~\ref{lem: Rp}~(i)), so again there are finitely many possibilities for both $A\cdot \ch_2(\iota_\ast G)$ and $r$.
\end{proof}

We say that a decomposition $\gamma=\sum\gamma_i$ is effective if all $\gamma_i\in N_1^\eff$.
\begin{corollary}\label{cor: finite}
There are only finitely many effective decompositions of~$\gamma\in N_1^\eff$. 
\end{corollary}
\begin{proof}
Every effective decomposition of $\gamma$ is a sum 
\[\gamma = \gamma'+\gamma''\]
with $\gamma'\in N_1^\eff$ a sum of classes from~$S_1$, and $\gamma''\in\Delta$. In particular, $\gamma'$ is an effective curve class and $\ell(\gamma'')\leq \ell(\gamma)$. By Lemma~\ref{lem: effectivefinite} there are finitely many such classes~$\gamma''$. By standard arguments~\cite[Corollary 1.19]{KM98}, there are finitely many decompositions of $\gamma'$ into effective curve classes.
\end{proof}

\subsection{Weak stability}\label{subsec: weakstability}
In this section we connect $\nu$-stability to the notion of weak stability in the sense of Toda~\cite{To10a}. We obtain an alternative description of the category~$\CA_{\leq 1}$. This section does not contain any results which are strictly necessary for the remainder of the paper and it rather serves as a comparison. In~\cite{PT19,To08} the authors study the moduli problem for (weak) stability conditions on tilted hearts. They are able to prove that the two key properties, \emph{generic flatness}, and \emph{boundedness of semistable objects} are preserved, in some sense, under a tilting process
\[ (Z,\CC) \rightsquigarrow (Z^\dagger, \CC^\dagger)\,.\]
It seems likely that this technique can be employed to deduce the results in Section~\ref{subsec: moduli}, although we will not pursue it in this paper.

Let $\CC_{\leq 1}=\Coh_{\leq 1}(X/W)$ be the category of coherent sheaves which are at most $1$-dimensional outside of~$W$. This category was studied in~\cite{To16} for Calabi--Yau 3-folds containing an embedded~$\BP^2$. The numerical $K$-group of~$\CC_{\leq 1}$ is the same as that of~$\CA_{\leq 1}$
\[N_0=\BZ\cdot\sfb\oplus \BZ\cdot\pt\,,\quad N_{\leq 1}=\BZ\cdot\sfw\oplus N_{\leq 1}(X)\,.\]
We can define a weak stability function~$Z=(Z_0,Z_1)$ associated to the filtration 
\[0\subset N_0\subset N_{\leq 1}\,.\]
Let $\omega\in\Amp(X)$ be an ample class. For $E\in\CC_{\leq 1}$ define
\begin{align*}
    Z_1(E) &= -\ell(E) + i\, \omega^2\cdot\ch_1(E)\,,\\
    Z_0(E) &= -d(E) + i\, \omega\cdot\ch_2(E)\,.
\end{align*}
Here, $d(E)$ and $\ell(E)$ are as defined in Section~\ref{subsec: Nironi stability}. If $[E]\in N_0$, set $Z(E)=Z_0(E)$, otherwise $Z(E)=Z_1(E)$. Then, for all $0\neq E\in\CC_{\leq 1}$:
\begin{enumerate}
    \item $Z(E)\in \BH\cup\BR_{<0}$, 
    \item $E$ admits a Harder--Narasimhan filtration.
\end{enumerate}
Property~(i) follows from condition~\eqref{eq: nef}. Property~(ii) holds because~$\CC_{\leq 1}$ is Noetherian and the image of~$Z$ is discrete.\footnote{We have not checked the support property for~$Z$. It might be possible to give a proof following the arguments in the surface case~\cite[Section 4]{BM11}.} 

Now we consider a tilting process
\[(Z,\CC_{\leq 1}) \rightsquigarrow (Z^\dagger, \CC_{\leq 1}^\dagger)\,.\]
Define the generalized slope of $0\neq E\in\CC_{\leq 1}$ as
\[ \lambda(E) = -\frac{\mathrm{Re}\,Z(E)}{\mathrm{Im}\,Z(E)}\in (-\infty,\infty]\,.\]
This leads to the standard construction of a torsion pair
\begin{align*} \CT_{\lambda} &= \big\langle \lambda\textup{-semistable } E\in\CC_{\leq 1}\textup{ with } \lambda(E)\geq 0\big\rangle_\ex\,, \\
\CF_{\lambda} &= \big\langle \lambda\textup{-semistable } E\in\CC_{\leq 1}\textup{ with } \lambda(E)< 0\big\rangle_\ex\,.
\end{align*}
Define the tilt as
\[ \CC_{\leq 1}^\dagger = \big\langle \CF_{\lambda}[1],\CT_{\lambda}\big\rangle_\ex\,,\]
and the function
\[Z^\dagger(E) = -d(E) + i\,\ell(E)\,.\]
Proposition~\ref{prop: ell} has two consequences. Firstly, the pair $(\CT_\lambda\,,\CF_\lambda)$ agrees with the perverse torsion pair:
\[ \CT_{\lambda} = \CT_{\leq 1}\,,\quad \CF_{\lambda}=\CF\,.\]
In particular, $\CA_{\leq 1} = \CC_{\leq 1}^\dagger$. Secondly, we have for all $0\neq E\in\CA_{\leq 1}$
\[ Z^\dagger(E) \in \BH \cup \BR_{<0}\,.\]
Harder--Narasimhan filtrations exist by Proposition~\ref{prop: nustability}. The associated slope function of~$Z^\dagger$ is precisely~$\nu$. In particular, $Z^\dagger$-semistability coincides with $\nu$-semistability. Note that this resembles the standard way to interpret slope stability on curves as Bridgeland stability~\cite[Example 5.4]{Br02},~\cite{MS17}. We have obtained $\CA_{\leq 1}$ and $\nu$-stability through a tilting process from~$(Z,\CC_{\leq 1})$. We do not know if this fits the general framework of tilting process established in~\cite{PT19}.


\subsection{Boundedness}\label{subsec: boundedness}
In this section we prove some boundedness and finiteness results that will be needed to ensure that the moduli stacks of $\nu$-semistable sheaves are finite type (see Proposition~\ref{prop: finitetype}. This condition is necessary for the application of the wall-crossing formula), for the analysis of the wall-crossing formula, and the proof of rationality in Section~\ref{sec: perversept}. For $E\in\CA_{\leq 1}$ we denote by $\nu_+(E)$, $\nu_-(E)$ the maximal and minimal slopes of the Harder--Narasimhan factors with respect to $\nu$-stability. For $I\subset \BR\cup\{+\infty\}$ denote by $\CM^\nu(I)$ the stack of all $E\in\CA_{\leq 1}$ such that all HN-factors have slope contained in $I$. If $I=[\delta_-,\delta_+]$, this is equivalent to $\nu_+(E)\leq \delta_+$ and $\nu_-(E)\geq \delta_-$. The substack $\CM_\gamma^\nu(I)$ parametrizes all such $E$ with fixed $[E]=\gamma\in N_1$. The special case $I=[\delta,\delta]$ parametrizes $\nu$-semistable $E$ of slope $\delta$ and is denoted $\CM_\gamma^{\nu}(\delta)$. The substack $\CM_{(\gamma,c)}^{\nu}\subset\CM_\gamma^{\nu}(\delta)$ corresponds to a fixed class $(\gamma,c)\in N_{\leq 1}$. 
We write $c_E$ to denote the class of $[E]$ in $N_0$. 
\begin{proposition}\label{prop: finite}
 Let $I\subset\BR$ be a bounded interval and $E\in \CM_\gamma^\nu(I)$. There exists a finite subset $S\subset N_0$ depending on $\gamma$ and $I$ such that $c_E\in S$, if one of the following holds:
\begin{enumerate}
    \item $E\in\Coh_{\leq 1}(X)$ with $\iota^\ast E\in \Coh_0(W)$,
    \item $E\cong\iota_\ast\iota^\ast E$.
\end{enumerate}
\end{proposition}
\begin{proof}
(i) In the first case, $\ch_2(E)\in N_1^\eff(X)$ is an effective curve class with residue $\gamma\in N_1^\eff$. The class $\gamma + j\sfb$ is effective only for finitely many negative values of~$j$.  On the other hand, note that for any $E\in\Coh_{\leq 1}(X)$ with $\iota^\ast E\in\Coh_0(W)$ we have $\ch_2(E)\cdot \sfw\geq 0$. If $j\gg 0$, then $\sfw\cdot(\gamma+j\sfb)<0$, since $\sfw\cdot\sfb=-2$. Thus, $j$ must lie in a bounded interval, so we have finitely many curve classes~$\ch_2(E)$. Recall that by definition of $\nu(E)$ we have
\[\chi(E) = \frac{1}{2}\big(\ell(E)\cdot\nu(E) + \frac{1}{2}\sfw\cdot\ch_2(E)\big)\,.\]
Since $\nu(E)\in I$, also $\chi(E)$ lies in a bounded interval.

(ii) Let $I\subset [\delta_-,\delta_+]$ and $G=\iota^\ast E$. We first prove that $\chi(E)$ is bounded below. For this we may assume $\chi(E)<0$. By Lemma~\ref{lem: basic} we have $Rp_\ast(G)\in\Coh(C)$ and also \[\chi(E)=\chi(G)=\chi(Rp_\ast G)\,.\]
Let $L\in\Pic(C)$ with
\[ \rk(Rp_\ast G)(\chi(L)+1-g) > \chi(G)\,.\]
We may choose $\chi(L) = \chi(G) + g$. Then by Riemann--Roch
\[ 0 \neq H^1(Rp_\ast G\otimes L^\vee \otimes \omega_C) = \Hom(Rp_\ast G, L)\,.\]
The latter is isomorphic to $\Hom(G, \omega_p\otimes p^\ast L[1])$ by adjunction. The object $\iota_\ast\big(\omega_p\otimes p^\ast L[1]\big)$ is stable by Lemma~\ref{lem: stable}, with slope $\chi(L)$. Since $\nu_-(E)\geq \delta_-$, we must have $\chi(G)\geq \delta_- -g$.

Now we prove that $\chi(E)$ is bounded above. For this we may assume $\chi(E)>0$, in particular $Rp_\ast G\neq 0$. By Lemma~\ref{lem: morphism} we obtain $L\in\Pic(C)$ with
\[k_- +\frac{\chi(G)}{\max\{\rk(Rp_\ast G), 1\}}\leq \chi(L)\leq k_+ +\frac{\chi(G)}{\max\{\rk(Rp_\ast G), 1\}}\]
and a non-zero morphism $K\to G$ with
\[K=\CO_p(-1)\otimes p^\ast L\,, \textup{ or } K=\omega_p\otimes p^\ast L[1]\,.\]
The object $\iota_\ast K\in\CA_{\leq 1}$ is stable by Lemma~\ref{lem: stable}, with slope 
\[ \nu(\iota_\ast K)=\chi(L) + \frac{1}{2}\deg(\CE) + 1-g\,, \textup{ or } \nu(\iota_\ast K)=\chi(L)\,.\]
Since $E\in \CM_\gamma^\nu(I)$ it follows that $\nu(K)\leq \delta_+$. But if $\chi(G)\gg 0$ we get $\chi(L)\gg 0$ (recall that $a_0\rk(Rp_\ast G)=A\cdot \ch_2(E)$ only depends on $\gamma$) and thus $\nu(K)\gg 0$, a contradiction.

We conclude that $\chi(G)$ is bounded. By the same argument as in (i), since $\nu(E)=d(E)/\ell(E)\in I$ is also bounded we can show that there are only finitely many possibilities for $j$ in $\ch_2(E)=\beta+j\sfb$, finishing the proof.
\end{proof}

We can now prove the boundedness of certain families of objects in~$\CA_{\leq 1}$. The underlying notion of sheaf of $t$-structures is established in~\cite{AP06} which we apply to the heart of perverse $t$-structure $\CA\subset D^b(X)$. For a discussion of bounded families see~\cite[Section 3]{To08}. We will repeatedly use the following useful result~\cite[Lemma 3.16]{To08} which relies on the finite dimensionality of $\Ext^1$-groups.
\begin{lemma}\label{lem: boundedExt}
Let $\CS_i$ be sets of objects in~$D^b(X)$ for $i=1,2,3$ such that $S_1$, $S_2$ are bounded. Assume that for any object $E_3\in\CS_3$ there are $E_i\in\CS_i$ for $i=1,2$ and an exact triangle
\[ E_1 \to E_3 \to E_2\,.\]
Then, $\CS_3$ is also bounded.
\end{lemma}
First, we consider the family of zero-dimensional perverse sheaves.

\begin{lemma}\label{lem: bounded0}
Let $D\geq 0$ and $\CS$ be the family of $E\in\CA_0$ with $d(E)=D$. Then, $\CS$ is a bounded family.
\end{lemma}
\begin{proof}
By Lemma~\ref{lem: A0}, every $E\in\CA_0$ admits a quotient $E\to Q$ in $\CA_0$ where $Q$ is one of the following objects:
\[ k(x)\,,\quad \CO_{B_y}(k-1)\,, \quad \CO_{B_y}(-k-2)[1]\,.\]
Here, $x\in X$ is a point, $B_y=p^{-1}(y)$ a fiber of $p$, and $k\geq 0$. By Lemma~\ref{prop: ell} we have $0<d(Q)\leq d(E)$, in particular $0\leq k\leq d(E)$. The family of such objects~$Q$ is bounded. We can conclude by induction and Lemma~\ref{lem: boundedExt}.
\end{proof}

We can now prove the following result.
\begin{proposition}
\label{prop: boundedfamilies}
Let $I \subset \BR$ be a bounded interval and $\gamma\in N_1$. Let $\CS$ be one of the following families of objects in~$\CM_\gamma^\nu(I)$:
\begin{enumerate}
    \item the set of $E\in\Coh_{\leq 1}(X)$ with $\iota^\ast E\in \Coh_0(W)$,
    \item the set of $E\cong\iota_\ast\iota^\ast E$.
\end{enumerate}
Then, $\CS$ is a bounded family.
\end{proposition}

\begin{proof}
(i) Let $I\subset[\delta_-,\delta_+]$ and let $\omega\in\Amp(X)$ be an ample class. We consider $\omega$-slope stability on $\Coh_{\leq 1}(X)$ defined by 
\[\mu_\omega(E) = \frac{\chi(E)}{\omega\cdot\ch_2(E)}\,.\]
By Proposition~\ref{prop: finite}~(i), the set of curve classes $\ch_2(E)$ for $E\in\CS$ is finite, so we can define 
\[m_-=\min\limits_{E\in\CS}\Big\{\frac{1}{2}\sfw\cdot\ch_2(E)\Big\}\,,\quad m_+=\max\limits_{E\in\CS}\Big\{\frac{1}{2}\sfw\cdot\ch_2(E)\Big\}\,.\]

Let $F\subset E$ be a subsheaf and $E\twoheadrightarrow Q$ a quotient, then\begin{align*}
    \chi(F)&\leq \frac{1}{2}\big(\ell(F)\cdot\delta_+ + m_+\big)\,,\\
    \chi(Q)&\geq \frac{1}{2}\big(\ell(Q)\cdot\delta_- + m_-\big)\,,
\end{align*} 
Recall that $\Coh_0(X)\subset \CT_0$, thus $E$ is torsion-free and $\omega\cdot\ch_2(F)>0$. By Lemma~\ref{prop: ell} we have $0\leq \ell(F),\ell(Q)\leq \ell(E)$ and so we obtain a bounded interval~$J$ (depending only on $\gamma$ and $I$) such that for all $E$ as above, the HN-factors of $E$ with respect to $\mu_\omega$-stabilty have slope contained in~$J$. Boundedness of the family of such~$E$ now follows from boundedness of $\mu_\omega$-stability~\cite[Theorem~3.3.7]{HL97}.\medskip

(ii) Assume that $E\cong \iota_\ast\iota^\ast E$ and denote by $G=\iota^\ast E$. By Proposition~\ref{prop: finite}~(ii) the set of classes $\alpha=[E]\in N_{\leq 1}$ for $E\in\CS$ is finite. Fix one such $\alpha$. We use Lemma~\ref{lem: morphism} to obtain $L\in\Pic(C)$ with $\chi(L)\geq n(\alpha)$ bounded below by some $n(\alpha)\in\BZ$ determined from the class $\alpha\in N_{\leq 1}$. We have a non-zero morphism
\[ K \to G\]
such that $K$ is either $\CO_p(-1)\otimes p^\ast L$ or $\omega_p\otimes p^\ast L[1]$. In both cases, $K$ is stable by Lemma~\ref{lem: stable}. Let $G'$ be the image of this morphism in~$\CA_{\leq 1}$, thus we obtain an exact triangle with pushforward in~$\CA_{\leq 1}$
\[ G' \to G \to G''\,.\]
Note that $\Hom(\CA_0,\iota_\ast G)=0$ since $\iota_\ast G\in \CM^\nu_\alpha(I)$ and $I\subset\BR$ is finite. Thus, $\ell(\iota_\ast G')>0$ by Lemma~\ref{prop: ell}. We can now bound the slopes of the HN-factors of $G'$ and $G''$ as follows. There are obvious inequalities
\[ \nu_+(\iota_\ast G')\leq \nu_+(\iota_\ast G)\,,\quad \nu_-(\iota_\ast G'')\geq \nu_-(\iota_\ast G)\,.\]
Since $G'$ is a quotient of $K$, we get $\nu_-(\iota_\ast G')\geq\nu(K)$, which is bounded below via $\chi(L)\geq n(\alpha)$ and Lemma~\ref{lem: stable}. Thus,  $d(\iota_\ast G') = \ell(\iota_\ast G')\nu(\iota_\ast G')$ lies in a bounded interval determined by~$\alpha$ and then the same is true for~$\iota_\ast G''$. We can conclude by induction on $\ell(\iota_\ast G)$ and Lemma~\ref{lem: boundedExt}. The case $\ell(\iota_\ast G)=0$ is covered by Lemma~\ref{lem: bounded0}.
\end{proof}

\subsection{Moduli stacks}\label{subsec: moduli}
The goal of this section is to explain the existence of finite type moduli spaces of $\nu$-semistable objects and stable pairs. The setup is as follows. 

Let $\CA=\langle \CF[1],\CT\rangle$ be the category of perverse sheaves defined in Section~\ref{sec: perverse t-structure} as the tilt along the torsion pair~$(\CT,\CF$) of $\Coh(X)$. We consider another torsion pair~$(\CT_{\leq 1},\CF')$ of $\Coh(X)$, where $\CT_{\leq 1} = \CT\cap\CA_{\leq 1}$ and $\CF' = \CT_{\leq 1}^\perp$. Define the tilt
\[ \Coh^\dagger(X) = \big\langle \CF'[1],\CT_{\leq 1}\big\rangle\,.\]
Recall Lieblich's~\cite{Li06} moduli stack $\CM$ of objects $E\in D^b(X)$ with
\[ \Ext^{<0}(E,E) = 0\,.\]
The stack~$\CM$ is an Artin stack locally of finite type.

\begin{lemma}\label{lem: open}
The stacks of objects $\Obj(\Coh^\dagger(X))$ and $\Obj(\CA)$ define open substacks of $\CM$.
\end{lemma}
\begin{proof}
In both cases, the heart is defined as a tilt along a torsion pair. The torsion part is defined by the condition $R^1p_\ast L\iota^\ast =0$, see Lemma~\ref{lem: basic}. This an open condition in families. The torsion-free part of the torsion pair is defined as the orthogonal complement, which is an open condition as well. Then, also the tilt defines an open substack~\cite[Theorem A.8]{ABL13}. 
\end{proof}
We consider stable pairs in the subcategory
\[ {}^p\CB=\Big\langle \CO_X[1]\,,\CA_{\leq 1}\Big\rangle_\ex\subset\Coh^\dagger(X)\,.\]
It follows from the argument in~\cite[Lemma 3.5, Lemma 3.8]{To10a} that ${}^p\CB$ is a Noetherian abelian category. Note that $\Coh^\dagger(X)$, however, is not Noetherian. Let $\Obj^{\geq -1}({}^p\CB)$ be the substack of objects of rank~$\geq -1$, thus the rank is either $-1$ or $0$.

\begin{proposition}\label{prop: finitetype}
Let $I\subset\BR$ be an interval, $\delta\in \BR$, $\gamma\in N_1$, and $\alpha\in N_{\leq 1}$, then
\begin{enumerate}
    \item $\Obj^{\geq -1}({}^p\CB)\subset \CM$ is an open substack,
    \item $\CM_\gamma^\nu(I)\subset \Obj(\CA_{\leq 1})$ is an open substack. If $I$ is bounded, $\CM_\gamma^\nu(I)$ is an Artin stack of finite type,
    \item $\CM_\alpha^\nu([\delta, +\infty])$ and $\CM_\alpha^\nu((-\infty, \delta])$ are Artin stacks of finite type.
\end{enumerate}
\end{proposition} 

 \begin{proof}(i) By Lemma~\ref{lem: open} it suffices to show that 
 \[\Obj^{\geq -1}({}^p\CB)\subset \Obj(\Coh^\dagger(X))\]
 is open. This can be proved in the same way as~\cite[Lemma 5.1]{To16}. An object $P\in\Coh^\dagger$ of rank~$0$ (resp.\ $-1$) is contained in ${}^p\CB$ if and only if $\det(P)=0$ (resp.\ $\det(P)\cong\CO_X$) and $\CH^{-1}(P)$ is torsion-free on~$X\ssetminus W$. The openness is proved using a spectral sequence argument as in~\cite[Lemma 3.16]{To10a}.\medskip
 
 (ii) We explain that $\CM_\gamma^\nu(I)\subset \Obj(\CA_{\leq 1})$ is open and that the family of objects in $\CM_\gamma^\nu(I)$ is bounded, if $I$ is bounded. It follows that $\CM_\gamma^\nu(I)$ is an Artin stack of finite type~\cite[Lemma 3.4]{To08}. By Corollary~\ref{cor: finite}, there are only finitely many effective decompositions of $\gamma$ in $N_1^\eff$. Boundedness of the family of objects in $\CM_\gamma^\nu(I)$ then follows from Lemma~\ref{lem: supp}, Proposition~\ref{prop: supp}, Proposition~\ref{prop: finite}, Lemma~\ref{lem: bounded0} and Proposition~\ref{prop: boundedfamilies}. 

Openness can be obtained from arguments of Toda
~\cite{To08,To16} as follows. In~\cite{To16} he considers Calabi--Yau 3-folds~$X$ containing a divisor isomorphic to~$\BP^2$, and the category of sheaves with at most $1$-dimensional support outside of the divisor. He studies objects in the tilt of this category along a torsion pair and proves boundedness of the family of semistable objects~\cite[Proposition 5.2]{To16}. Openness is deduced from boundedness as in~\cite[Theorem 3.20]{To08} and the same proof can be used for $\nu$-stability. 

 (iii) Suppose that $E\in \CM_\alpha^\nu\big([\delta, +\infty]\big)$ (the other case is analogous) and without loss of generality $\delta<0$. Consider the decomposition \[E_0\to E\to E_1\]
 of $E$ with respect to the torsion pair $(\CA_0, \CA_1)$. Let $\gamma\in N_1$ be the residue of~$\alpha$. Then, $E_1\in\CM_\gamma^\nu\big([\delta, +\infty)\big)$, so for any subobject $E'\to E_1$ in $\CA$ we have either $\nu(E')\leq 0$, or
 \begin{align*}
\nu(E')&\leq d(E')=d(E_1)-d(E_1/E')\leq d(E)-\ell(E_1)\delta\\
&\leq d(\alpha)-\ell(\alpha)\delta,
 \end{align*}
 thus $E_1\in \CM_\gamma^\nu\big([\delta,\max\{0, d(\alpha)-\ell(\alpha)\delta\}]\big)$ is bounded. In particular, there are only finitely many possibilities for $d(E_1)$ and hence finitely many possibilities for $d(E_0)$, so the family of possible $E_0$ is bounded by Lemma \ref{lem: bounded0}. Using Lemma 4.11 we conclude (iii).
\end{proof}


The next lemma will be useful in the combinatorical analysis of the wall-crossing formula. Let $A$ be the nef class of condition~\eqref{eq: nef}. The restriction~$\iota^\ast A$ is numerically equivalent to a multiple of~$\sfb$, thus multiplication by $A$ defines a map 
\[ A\cdot(-)\colon N_{\leq 1}\to N_0\,.\]
\begin{lemma}[{\cite[Proposition 7.1.(3)]{BCR18}}]\label{lem: cclassesfinite}
For any $\gamma\in N_1$ the image of the set
\[\{c\in N_0\mid \M_{(\gamma, c)}^{\nu}\neq \emptyset \}\]
in the quotient 
\[N_0/\BZ(A\cdot\gamma)\] 
is finite.
\end{lemma}
\begin{proof}
The proof is the same as in \cite{BCR18}, using Proposition~\ref{prop: finitetype}.
\end{proof}


\subsection{Refined stability}\label{subsec: refined stability}
Finally we introduce the last stability function that we'll need. This stability function $\zeta$ will be used for the $\bs/\PPT$ wall-crossing and is the analog of \cite[Definition 8.1]{BCR18}.

For $E\in\CA_{\leq 1}\ssetminus\{0\}$ define the function
\[\zeta(E)=\Big(-\frac{r}{\ell(E)},\nu(E)\Big)\in (-\infty, +\infty]\times (-\infty, +\infty]\,,\]
where as before $r\in\BZ$ such that $\ch_1(E)=r\sfw$. If $E\in \CA_0$ we set
\[\zeta(E)=(+\infty, +\infty)\,.\]
We give $(-\infty, +\infty]\times (-\infty, +\infty]$ the lexicographic order. For $x,y\in (-\infty, +\infty]\times (-\infty, +\infty]$ we write $[x,y]$ and $]x,y]$ for the set of all $z$ with $x\leq z\leq y$ resp.\ $x<z\leq y$. Note that the first component
\[ \zeta_1(E)=-\frac{r}{\ell(E)}\]
only depends on the class of $[E]$ in $N_1=N_{\leq 1}/{N_0}$. For $\gamma\in N_1$ we will also write $\zeta_1(\gamma)$.

\begin{proposition}
The slope $\zeta$ defines a stability condition on $\CA_{\leq 1}$. 
\end{proposition}
\begin{proof}
The see-saw property is straightforward. To prove that $\CA_{\leq 1}$ is $\zeta$-Artinian, the same strategy as in \cite[Proposition 8.2]{BCR18} can be employed: by Corollary \ref{cor: finite} it's enough to show that $\CA_{\leq 1}$ is $\nu$-Artinian, which we did in Proposition \ref{prop: nustability}.
\end{proof}

Given a subset $I\subset (-\infty, +\infty]\times  (-\infty, +\infty]$ we follow the notation of Section~\ref{subsec: boundedness}, so e.g.\ $\CM^\zeta(I)$ is the stack of $E\in \CA_{\leq 1}$ such that all their $\zeta$-HN-factors are contained in $I$. To apply the wall-crossing formula to $\zeta$-wall-crossing we will need to prove that the stacks $\CM^\zeta(I)$ are open and (locally) of finite type. 

For this, we recall the linear function $L_\mu\colon N_0\to \BR$ defined by
\[L_\mu(c)=L_{\mu}(j\sfb, n)=2n+j+\frac{j}{\mu a_0}.\]
A set $S\subset N_0$ is said to be $L_\mu$-bounded if for each $M\in \BR$,
\[\# \{c\in S\colon L_\mu(c)<M\}<\infty.\]
We say that a set of objects in $\CA_{\leq 1}$ is $L_\mu$-bounded if its image in $N_0$ is $L_\mu$-bounded.

\begin{lemma}[{\cite[Lemma 8.14]{BCR18}}]
\label{lem: Lmubounded}
Given $\mu>0, \eta_1, \eta_2\in \BR$ and $\gamma\in N_1$, the sets
\[\CM^\nu_\gamma\big([\eta_1, +\infty]\big)\cap \CM^\zeta_\gamma\big(](-\infty, -\infty), (\mu,\eta_2)]\big)\]
and 
\[\CM^\nu_\gamma\big((-\infty,\eta_1]\big)\cap \CM^\zeta_\gamma\big([(\mu,\eta_2), (+\infty, +\infty)]\big)\]
are $L_\mu$-bounded.
\end{lemma}
\begin{proof}
See \cite[Lemma 8.14]{BCR18}.\qedhere
\end{proof}

\begin{proposition}
\label{prop: finitetypezeta}
Let $I\subset (-\infty, +\infty]\times  (-\infty, +\infty]$ be an interval, $\gamma\in N_1$ and $(\mu, \eta)\in \BR_{>0}\times\BR$.
\begin{enumerate}
    \item The stack $\CM^\zeta(I)\subset \Obj(\CA_{\leq 1})$ is an open substack locally of finite type.
    \item The family of objects in $\CM^\zeta_\gamma(\mu, \eta)$ is $L_\mu$-bounded.
\end{enumerate}
\end{proposition}
\begin{proof}
Given $E\in \CM^\zeta_{\gamma}(\mu, \eta)$ we consider its decomposition with respect to the $\nu$-HN-filtration
\[E_{\geq \eta}\to E\to E_{<\eta}.\]
Then, both $\gamma'=[E_{\geq \eta}]\in N_1^\eff$ and $\gamma-\gamma'\in N_1^\eff$, and we have
\[E_{\geq \eta}\in \CM^\nu_{\gamma'}\big([\eta, +\infty]\big)\cap \CM^\zeta_{\gamma'}\big(](-\infty, -\infty), (\mu,\eta)]\big)\,.\]
By Corollary~\ref{cor: finite} there are finitely many such~$\gamma'$, so by Lemma \ref{lem: Lmubounded} the set of possibilities for $c_{E_{\geq \eta}}$ is $L_\mu$-bounded. Similarly, the possibilities for $c_{E_{< \eta}}$ are also $L_\mu$-bounded and (ii) immediately follows.

For (i), by \cite[Theorem 3.20]{To08} it is again enough to show that the family of semistable sheaves in $\CM^\zeta_{(\gamma, c)}(\mu, \eta)$ is bounded. But using the decomposition above we have $c_{E_{\geq \eta}}+c_{E_{< \eta}}=c$, so there is a finite number of possibilities for both $c_{E_{\geq \eta}}$ and $c_{E_{<\eta}}$. It then follows from Proposition \ref{prop: finitetype} (iii) that the families of possible $E_{\geq \eta}, E_{<\eta}$ are both bounded. By Lemma \ref{lem: boundedExt} we conclude that $\CM^\zeta_{(\gamma, c)}(\mu, \eta)$ is bounded.\qedhere
\end{proof}

\section{Bryan--Steinberg}
\label{sec: BS}

In this section we introduce numerical invariants $\bs_{\beta,n}$ that naturally realize the quotient
\[\bs_\beta(q,Q)=\frac{\PT_\beta(q,Q)}{\PT_0(q,Q)}\,.\]
The equation will be a wall-crossing formula between $\bs$ and $\PT$ invariants. When $X$ admits a contraction map $X\to Y$ as in Section~\ref{subsec: crepant} these invariants are precisely Bryan--Steinberg invariants \cite{BS16} of the crepant resolution. Roughly speaking they count a modification of pairs $\CO_X\to F$ where instead of requiring the cokernel to have dimension zero we allow it to have support in some of the fibers~$B$. 

We define $\bs$-pairs using a torsion pair of $\Coh_{\leq 1}(X)$. Let
\[\tbs=\big\{T\in \Coh_{\leq 1}(X)\colon \, T_{|X\ssetminus W}\in\Coh_{0}(X\ssetminus W) \text{ and } Rp_\ast \iota^\ast T\in \Coh_0(X)\big\}\,.\]
One easily checks that $\tbs$ is closed under quotients and extensions (see \cite[Lemma 13]{BS16} for the case where a contraction exists), so 
\[\fbs=\{F\in \Coh_{\leq 1}(X): \Hom(\tbs, F)=0\}\]
defines the torsion-free part of a torsion pair $(\tbs, \fbs)$ of $\Coh_{\leq 1}(X)$.

The same proof as given in \cite[Lemma 51]{BS16} can be used to write the torsion pair $(\tbs, \fbs)$ in terms of the stability condition $\mu^A$ introduced in Section \ref{subsec: BS stability}:
\[\tbs=\CM^{\mu^A}\Big(\big[\frac{\infty}{2}, +\infty\big[\Big)\,, \quad \fbs=\CM^{\mu^A}\Big(\big]-\infty, \frac{\infty}{2}\big[\Big)\,,\]
where we used $\frac{\infty}{2}$ to denote 
\[\frac{\infty}{2}=(+\infty, 0)\in (-\infty, +\infty]\times (-\infty, +\infty]\,.\]

The $\bs$ numerical invariants are defined as usual via the integration map~$I$. We denote by $\Pairs^\bs$ the stack of $(\tbs, \fbs)$-pairs in the sense of Definition~\ref{def: pair}. Then, we define $\bs_{\beta,n}\in \BQ$ by the equation
\[I\big((\BL-1)\Pairs^\bs\big)=\sum_{n,\beta}\bs_{n,\beta} z^\beta\, q^n\, t^{-1}\,.\]
We also denote 
\[\bs_\beta(q,Q)=\sum_{n,j\in \BZ}\bs_{ \beta+j\sfb, n}(-q)^n\, Q^j\in \BQ[[q^{\pm 1},Q^{\pm 1}]]\,.\]

\subsection{Wall-crossing between $\bs$ and $\PT$}

The wall-crossing between $\bs$ and $\PT$ invariants can be directly deduced from the discussion in Section \ref{subsec: wallcrossing}. Recall that the usual stable pairs are defined as pairs with respect to the torsion pair \[(\CT_\PT, \CF_\PT)=\big(\Coh_0(X), \Coh_1(X)\big).\] 
The technical conditions required in Section \ref{subsec: wallcrossing} are satisfied.

\begin{proposition}\label{prop: finitetypebs}
The moduli $\Pairs^\bs_{(\beta, n)}\subset \CM$ is an open substack of finite type. Moreover, the pairs $(\CT_\PT, \CF_\PT)$ and $(\CT_\bs, \CF_\bs)$ are wall-crossing material.
\end{proposition}
\begin{proof}
The torsion pair $(\CT_\PT, \CF_\PT)$ is clearly open. The torsion pair $(\CT_\bs, \CF_\bs)$ is also open thanks to the description in terms of $\mu^A$ stability and Proposition \ref{prop: BSstabilityfinitetype}. By \cite[Proposition 4.6]{BCR18} it follows that $\Pairs^\bs$, $\Pairs(\CT_\PT, \CF_\bs)$ are open, locally of finite type substacks of $\CM$. 

To show that the pairs are wall-crossing material  remains to show that $\CW=\CF_\PT\cap \CT_\bs$ satisfies conditions (i)-(iii) in Section \ref{subsec: wallcrossing}. Conditions (i) and (ii) are straightforward. For (iii), write $\alpha_i=(\beta_i, n_i)$. If $\CW_{\alpha_i}\neq \emptyset$ we must have $\beta_i=j_i\sfb$ for some $j_i\geq 1$ and $n_i\geq 0$, so it's clear that there are only finitely many such decompositions.\qedhere
\end{proof}

Joyce's wall-crossing in Section \ref{subsec: wallcrossing} (or \cite[Theorem 6.10]{BCR18}) now applies to show that, for every $\beta\in H_2(X, \BZ)$,
\begin{equation}\PT_\beta(q,Q)=f(q,Q)\,\bs_\beta(q,Q)\,,\label{eq: wallcrossingptbs}\end{equation}
where $f(q,Q)$ is defined by 
\[f(q,Q)=I\big((\BL-1)\log([\CW])\big)\in \BQ[[q,Q]]\]
and 
$\CW=\CF_\PT\cap \CT_\bs=\Coh_1(X)\cap \CT_\bs.$
Note that $f\in \BQ[[q,Q]]$ because the support of sheaves in $\CW\subset \tbs$ is a finite union of finitely many points and fibers $B$. Note also that $f$ doesn't depend on $\beta$, so we get the relation
\[\frac{\PT_\beta(q,Q)}{\bs_\beta(q,Q)}=\frac{\PT_0(q,Q)}{\bs_0(q,Q)}\,.\]

\begin{lemma}
The only $\bs$-pair with Chern class of the form $(-1, 0, j\sfb, n)$ is the trivial pair $(\CO_X\to 0)$. In particular
\[\bs_0(q,Q)=1\,.\]
\end{lemma}
\begin{proof}
The hypothesis of \cite[Lemma 3.11]{BCR18} applies to $\tbs$, showing that $\bs$-pairs have the form $(\CO_X\overset{s}{\to} G)$ where $G\in \fbs$ and $\coker(s)\in \tbs$. Since $\Coh_0(X)\subset \tbs$ we have $\fbs\subset \Coh_1(X)$, so $G$ is a pure 1-dimensional sheaf. Since $\ch_2(G)=j\sfb$, the reduced support of $G$ is a finite union of fibers $B$. 

Letting $Z$ be the subscheme of $X$ determined by $\ker(s)=I_Z$, we get an inclusion $\CO_Z\hookrightarrow G$. The closed subspace underlying $Z$ is a union of fibers $B$, so one easily sees that $\CO_Z\in \tbs$. As $G\in \fbs$ it follows that $G=0$.\qedhere
\end{proof}

As a consequence we get the key result of this section:
\begin{proposition}\label{prop: BSPTwallcrossing}
We have
\[\bs_\beta(q,Q)=\frac{\PT_\beta(q,Q)}{\PT_0(q,Q)}\,.\]
\end{proposition}
We recall that $\PT_0(q,Q)$ can be computed (for example by localization on $K_W$, see appendix \ref{appendix}) and is equal to
\[\PT_0(q,Q)=\prod_{j\geq 1}(1-q^j Q )^{(2g-2)j}\,.\]

\section{Perverse PT invariants} \label{sec: perversept}
Consider the torsion pair $(\CA_0,\CA_1)$ of $\CA_{\leq 1}$ and recall the category
\[ {}^p\CB=\Big\langle \CO_X[1]\,,\CA_{\leq 1}\Big\rangle_\ex\,.\]
 An object $P\in{}^p\CB$ is called \emph{perverse stable pair}, if it is a $(\CA_0,\CA_1)$-pair in the sense of Definition~\ref{def: pair}, i.e.\ $\rk(P)=-1$ and
\[ \Hom(\CA_0,P)=0=\Hom(P,\CA_1)\,.\]

The stack of perverse pairs is denoted by $\PPairs$. Numerical invariants counting perverse stable pairs are defined using the integration map $I$ as explained in Section \ref{sec: Hall}. For $\alpha\in N_{\leq 1}$ we let $\PPT_\alpha\in \BQ$ be the numerical invariants defined by
\[I\big((\BL-1)\,^{p}\PT\big)=\sum_{(\gamma, j,n)}\PPT_{(\gamma,j,n)}z^\gamma Q^{j}q^nt^{-[\CO_X]}.\]
The fact that the integration map $I$ can be applied to $(\BL-1)\,^{p}\PT$ is justified by Lemmas \ref{lem: pairsnuopen} and \ref{lem: wallcrossinglimite}.

In this section, we will provide a proof of the rationality and functional equation of perverse stable pairs, Theorem \ref{thm: perverse}.

\subsection{Rationality via $\nu$-wall-crossing}
\label{subsec: rationalityvianu}

For $\delta\in \BR$ we introduce the torsion pair $(\CT_{\nu,\delta}, \CF_{\nu, \delta})$ on $\CA_{\leq 1}$ by truncating the $\nu$-HN-filtation at $\delta$:
\begin{align*}
    \CT_{\nu,\delta}&=\CM^\nu\big([\delta, +\infty]\big)=\left\{T\in \CA_{\leq 1}: T\twoheadrightarrow Q\neq 0 \Rightarrow \nu(Q)\geq \delta\right\}\,,\\
     \CF_{\nu,\delta}&=\CM^\nu\big((-\infty, \delta)\big)=\left\{F\in \CA_{\leq 1}: 0\neq S\hookrightarrow F\Rightarrow \nu(S)<\delta \right\}.
\end{align*}

This family of torsion pairs depending on $\delta$ will describe the wall-crossing that ultimately will connect $^{p}\Pairs$ ($\delta\to +\infty$ limit) and $\rho(^{p}\Pairs)$ ($\delta\to -\infty$ limit). We denote by $\Pairs^{\nu, \delta}$ the category (or the stack, depending on the context) of $(\CT_{\nu,\delta}, \CF_{\nu, \delta})$-pairs as defined in Section~\ref{subsec: pairs}. This stack admits a decomposition into connected components according to the class of its elements and we write $\Pairs^{\nu, \delta}_{(\gamma, c)}$ for the stack of pairs in class $(-1, \gamma, c)$. 
\begin{lemma}\label{lem: pairsnuopen}
Let $\delta\in\BR$ and $(\gamma, c)\in N_{\leq 1}$. The stack $\Pairs_{(\gamma,c)}^{\nu,\delta}$ is a finite type open substack of $\Obj^{\geq -1}({}^p\CB)$.
\end{lemma}
\begin{proof}
An object $P\in\Obj^{\geq -1}({}^p\CB)$ is a $(\CT_{\nu,\delta}\,,\CF_{\nu,\delta})$-pair if and only if three conditions hold:
 \begin{enumerate}
     \item $\CH^{0}(P)\in\CT_{\nu,\delta}$,
     \item $\CH^0\big(\rho(P)\big)\in\big\langle \CA_0\,,\rho(\CF_{\nu, \delta})\big\rangle_\ex$,
     \item $\CH^1\big(\rho(P)\big)=0$.
 \end{enumerate}
 This characterization is parallel to the description of stable pairs (with respect to torsion theories) in $\big\langle\CO_X[1]\,,\Coh_{\leq 1}(X)\big\rangle_\ex$ using the dualizing functor~\cite[Lemma 4.5]{BCR18}. Instead of the dualizing functor, we use the duality~$\rho$ and apply the same proof as~\cite[Proposition 4.6]{BCR18}. The necessary properties of~$\rho$ are proven in Section~\ref{sec: perverse t-structure}. The first and third properties are open by~\cite[Lemma 4.1]{BCR18}, the second one by~Theorem~\ref{prop: rho} and Property~\eqref{eq: star}. 
\end{proof}Applying the integration morphism in the Hall algebra produces numerical invariants $\PDT^{\nu,\delta}_{(\gamma,c)}\in \BQ$ defined by
\begin{equation}I\big((\BL-1)\Pairs^{\nu, \delta}\big)=\sum_{(\gamma, j,n)}\PDT^{\nu,\delta}_{(\gamma,j,n)}z^\gamma Q^{j}q^nt^{-[\CO_X]}. \label{eq: dtinvariants}\end{equation}

\begin{lemma}[{\cite[Proposition 7.6.(1)]{BCR18}}]\label{lem: deltapairsfinite}
For any $\delta\in \BR$ and $\gamma\in N_1$ the set
\[\{c\in N_0: \Pairs_{(\gamma, c)}^{\nu, \delta}\neq 0\}\]
is finite.
\end{lemma}
\begin{proof}
The proof is an easy adaptation of the proof of \cite[Proposition 7.6.(1)]{BCR18}.
\end{proof}

In the limit $\delta\rightarrow +\infty$ these invariants agree with the perverse PT invariants previously defined.

\begin{lemma}\label{lem: wallcrossinglimite}
Let $P\in {}^p\CB$ be an object of class $(-1, \gamma, c)$. For $\delta\gg 0$ (depending on $\gamma,c$) we have
\[P\in \Pairs^{\nu, \delta}\textup{ if and only if }P\in {}^p\mathrm{Pairs}.\]
\end{lemma}
\begin{proof}
The proof is analogous to \cite[Lemma 7.10]{BCR18}.
\end{proof}

We will now apply Joyce's wall-crossing formula discussed in Section \ref{subsec: wallcrossing}. The next lemma states the technical conditions under which we can use the wall-crossing formula.

\begin{lemma}
\label{lem: wallcrossingmaterial}
Let $\varepsilon>0$ be sufficiently small. Then the torsion pairs $(\CT_{\nu, \delta\pm \varepsilon},\CF_{\nu, \delta\pm \varepsilon})$ are wall-crossing material.
\end{lemma}
\begin{proof}
We begin by clarifying the statement and what we mean by sufficiently small $\varepsilon$. Fixing $l>0$, the moduli of semistable sheaves $\CM^\nu_{\leq l}(\delta')$ with $\ell(E)\leq l$ is empty unless $\delta'\in W_l=\frac{1}{l!}\BZ$. Hence, by picking sufficiently small $\varepsilon$ (depending on $l$) the intersection 
\[\CW=\CT_{\nu, \delta- \varepsilon}\cap \CF_{\nu, \delta+\varepsilon}\] restricted to objects with $\ell(E)\leq l$ will be precisely $\CM^\nu_{\leq l}(\delta)$. This will suffice for the way we'll write the wall-crossing formula.

Now for the actual proof. The stacks of pairs $\Pairs^{\nu, \delta\pm \varepsilon}$ define elements in the (graded pre-)Hall algebra by Lemma \ref{lem: pairsnuopen}. It's then enough to show that $\CW=\CM^\nu(\delta)$ satisfies conditions (i)-(iii) of Section \ref{subsec: wallcrossing}. Condition (ii) is obvious and condition (i) is proven in Proposition \ref{prop: finitetype}. For (iii), let $\alpha_i=(\gamma_i, c_i)$. By Corollary \ref{cor: finite} there are finitely many possibilities for each $\gamma_i$. It also follows from Proposition \ref{prop: finitetype} that for fixed $\delta, \gamma_i$ there are only finitely many $c_i$ so that $\CM^\nu_{(\gamma_i, c_i)}(\delta)$ is non-empty.\qedhere
\end{proof}

By the previous lemma, we can define the invariants $J^{\nu}_{\alpha}$ for $\alpha\in N_{\leq 1}$ by counting semistable perverse sheaves with respect to the slope $\nu$:
\begin{equation}
I\big((\BL-1)\log\left(\M^{\nu}(\delta)\right)\big)=\sum_{\nu(\alpha)=\delta}J^\nu_\alpha z^\alpha.
    \label{eq: Jinvariants}
\end{equation}
The $J$-invariants are analogous to Toda's $N$-invariants in the proof of the rationality of stable pairs generating functions.

The wall-crossing formula between $\PPT$ and $\PDT^{\nu, \delta_0}$ is
\begin{equation}\PPT_{\leq l}t^{-1}=\left(\prod_{\delta\in W_l\cap [\delta_0, +\infty)}\exp\big(\{J_{\leq l}(\delta), -\}\big)\right)\PDT^{\nu, \delta_0}_{\leq l} t^{-1}.\label{eq: pervwallcrossing}
\end{equation}

Here the subscript $\leq l$ means we're restricting the generating functions to the classes $\alpha\in N_{\leq 1}$ such that $\ell(\alpha)\leq l$. Moreover, 
\[W_l=\frac{1}{l!}\BZ\]
is the set of possible walls since  $\ell(\alpha)\leq l$ implies $\nu(\alpha)\in W_l$.

\begin{remark}
In the wall-crossing formula \eqref{eq: pervwallcrossing} the wall-crossing terms interact, i.e. $\{J(\delta), J(\delta')\}$ might be non-trivial. In the usual proof of rationality of $\PT$ generating series or in the $\bs/\PT$ wall-crossing this phenomenom doesn't happen because the wall-crossing terms are at most 1-dimensional, and $\chi$ vanishes when restricted to $\Coh_{\leq 1}\times \Coh_{\leq 1}$. However, that's no longer the case in $\CA_{\leq 1}\times \CA_{\leq 1}$ due to the presence of surface-like objects. In particular we don't get a product formula for wall-crossing similar to Proposition \ref{prop: BSPTwallcrossing}. The same phenomenon already happens in \cite{BCR18}.
\end{remark}

\subsection{Combinatorics of the wall-crossing formula}\label{subsec: combinatorics}
Expanding the right-hand side of the wall-crossing formula \eqref{eq: pervwallcrossing} and extracting the coefficient of $z^\gamma t^{-1}$ we get the following expression for the perverse PT invariants in class $\gamma\in N_1$. The generating series
\begin{equation}
    \PPT_\gamma= \sum_{j,n}\PPT_{(\gamma, j,n)}Q^j q^n=\sum \ldots \label{eq: wallcrossingexpanded}
\end{equation}
is a sum over a set of choices described by an integer $m\in \BZ_{\geq 0}$ and classes $\alpha_1, \ldots, \alpha_i=(\gamma_i, c_i), \ldots, \alpha_m\in N_{\leq 1}$ and $\alpha'=(\gamma', c')\in N_{\leq 1}$, satisfying the following conditions:
\begin{enumerate}
    \item $\gamma=\gamma'+\sum_{i=1}^m \gamma_i$,
    \item$ \delta_0\leq \nu(\alpha_1)\leq \ldots \leq \nu(\alpha_m)$,
    \item $J_{\alpha_i}^\nu\neq \emptyset$,
    \item $\PDT_{\alpha'}^{\delta, \nu}\neq \emptyset$.
\end{enumerate}
We now use the boundedness results to analyze this sum. First, conditions $(3)$ and $(4)$ imply that $\gamma_i, \gamma'\in N_1^\eff$. Together with condition $(1)$ and Corollary~\ref{cor: finite} it follows that there is only a finite amount of possibilities for $\gamma_i, \gamma'$. Lemma \ref{lem: deltapairsfinite} also tells us that there is only a finite number of possibilities for $\alpha'$. Finally, Lemma~\ref{lem: cclassesfinite} says that, after we fix $\gamma_1, \ldots, \gamma_m$ there are finitely many possibilities for the classes $\kappa_i= [c_i]\in N_0/\BZ(A\cdot \gamma_i)$.

Since twisting by $\CO_X(A)$ induces an isomorphism \[\M^{\nu}_{(\gamma_i, c_i)}\cong \M^{\nu}_{(\gamma_i, c_i+A\cdot \gamma_i)}\,,\]
it follows that $J^\nu_{(\gamma_i, c_i)}$ depends only on $\gamma_i$ and the class $\kappa_i=[c_i]$, so we write $J^\nu_{(\gamma_i, \kappa_i)}=J^\nu_{(\gamma_i, c_i)}$.

Due to the combinatorical factor in \eqref{eq: wallcrossingexpanded} we also introduce the set $\CJ$ tracking which of the inequalities in $(2)$ are strict:
\[\CJ=\big\{i\in \{1, \ldots, m-1\}: \nu(\alpha_i)=\nu(\alpha_{i+1})\big\}\,.\]
We group the terms in the right hand side of \eqref{eq: wallcrossingexpanded} in finitely many groups according to the data $\xi=(\{\gamma_i\}_i,\{\kappa_i\}_i, \gamma', c', \CJ)$. Since $\nu(\gamma_i, c_i+A\cdot \gamma_i)=\nu(\gamma_i, c_i)+1$, given a group $\xi$ we can chose a minimal set of representatives $c_i^0\in \kappa_i$ such that 
\[\delta_0\leq \nu(\gamma_1,c_1^0)<\delta_0+1 \textup{ and } \nu(\gamma_i, c_i^0)\leq \nu(\gamma_{i+1},c_{i+1}^0)<\nu(\gamma_i,c_i^0)+1\,.\]
Then we organize equation \eqref{eq: wallcrossingexpanded} as
\begin{equation}
\PPT_\gamma=\sum_{\xi}A(\xi)\sum_{(k_1, \ldots, k_m)\in S_\CJ} B_\xi(k_1, \ldots, k_m)z^{c'+\sum_{i=1}^m(c_i^0+k_i(A\cdot \gamma_i))}    \label{eq: wallcrossinggrouping}
\end{equation}
where the first sum runs over the finitely many possible groups and the second sum runs over the set
\[S_\CJ=\big\{(k_1\leq \ldots\leq k_m): k_{i}=k_{i+1}\Leftrightarrow i\in \CJ\big\}\,.\]
Since $B_\xi$ is a quasi-polynomial of period 2, the rationality of $\PPT_\gamma$ follows from \cite[Lemma 2.21]{BCR18}.

\subsection{Functional equation}

After we have established the rationality part of Theorem \ref{thm: perverse}, we turn to the functional equation. For this, the duality $\rho$ introduced in Section \ref{subsec: rho} plays a crucial role. 
\begin{lemma}
\label{lem: symmetrydelta}
Let $\delta\in \BR\ssetminus \BQ$. Then \[\rho\big(\Pairs^{\nu, \delta}\big)=\Pairs^{\nu, -\delta}\,.\]
In particular,
\[\PDT_{\alpha}^{\nu, \delta}=\PDT_{\rho(\alpha)}^{\nu, -\delta}\,.\]
\end{lemma}
\begin{proof}
The lemma is proven exactly as in \cite[Lemma 7.4]{BCR18}, replacing $\Coh(\CY)$ by $\CA$ and $\BD^{\CY}$ by $\rho$. The properties of $\rho$ needed for the proof are Theorem~\ref{thm: rho} and Proposition~\ref{prop: nusymmetry}.
\end{proof}

\begin{lemma}\label{lem: limitdegree}
Let $\gamma\in N_1$. We have
\[\lim_{\delta\rightarrow -\infty}\deg\left(\PPT_\gamma-\PDT^{\nu,\delta}_\gamma\right)=-\infty.\]
\end{lemma}
\begin{proof}
We consider the wall-crossing equation \eqref{eq: wallcrossingexpanded}, \eqref{eq: wallcrossinggrouping} with $\delta=\delta_0$. Note that the terms in \eqref{eq: wallcrossinggrouping} with $m=0$ (that is, in groups $\xi=(\emptyset, \emptyset, \gamma, c, \emptyset)$) give precisely $\PDT^{\nu,\delta}_\gamma$, so we may express the difference $\PPT_\gamma-\PDT^{\nu,\delta}_\gamma$ as the sum on the right-hand side of \eqref{eq: wallcrossinggrouping} restricted to $m\geq 1$. Thus we have 
\[\deg\left(\PPT_\gamma-\PDT^{\nu,\delta}_\gamma\right)\leq \max_{\xi}\left(d(c')+\sum_{i=1}^m d(c_i^0)\right)\]
where the max is taken over the groups $\xi$ with $m\geq 1$. Summing $d(\gamma)$ to both sides 
\[\deg\left(\PPT_\gamma-\PDT^{\nu,\delta}_\gamma\right)+d(\gamma)\leq \max_{\xi}\left(d(\gamma',c')+\sum_{i=1}^m d(\gamma_i,c_i^0)\right)\,.\]
By the minimality of $c_i^0$ we know that $d(\gamma_i, c_i^0)<\delta_0+i$, and therefore we get the bound 
\[\deg\left(\PPT_\gamma-\PDT^{\nu,\delta}_\gamma\right)+d(\gamma)\leq \max_{\xi}\left(d(\gamma',c')+m\delta_0+\frac{m(m+1)}{2}\right)\,.\]
Now taking $\delta\rightarrow -\infty$ gives the desired limit.
\end{proof}

By Lemmas~\ref{lem: wallcrossinglimite} and \ref{lem: symmetrydelta}, for any $\alpha\in N_{\leq 1}$ and sufficiently small $\delta$ we have $\PDT^{\nu,\delta}_{\alpha}=\PPT_{\rho(\alpha)}$. Thus, we have
\[\PDT^{\nu, -\infty}_\alpha = \lim_{\delta\rightarrow -\infty}\PDT^{\nu, \delta}_{\alpha}=\lim_{\delta\rightarrow +\infty}\PDT^{\nu, \delta}_{\rho(\alpha)}=\PPT_{\rho(\alpha)}\,.\]
Here $\rho(\alpha)$ denotes the action on cohomology induced by $\rho$ determined by Proposition~\ref{prop: actionCoh}. One can write this action as $\rho(\gamma, c)=\big(\gamma, \rho_\gamma(c)\big)$, where for each $\gamma=(r\sfw,\beta)$ the involution $\rho_\gamma\colon N_0\to N_0$ is
\[\rho_\gamma(j\sfb, n)=\big((-j+\sfw\cdot \beta-2r)\sfb, -n\big)\,.\]

We write the previous relation between $\PDT^{\nu, -\infty}$ and $\PPT$ as an equality of generating functions for $\gamma\in N_1$:

\[\PDT_\gamma^{\nu, -\infty}= \sum_{c\in N_0}\PDT_{(\gamma, c)}^{\nu, -\infty}z^c=\sum_{c\in N_0}\PPT_{(\gamma, \rho_\gamma(c))}z^c=\rho_\gamma(\PPT_\gamma)\,.\]

It follows that $\PDT_\gamma^{\nu, -\infty}$ is the expansion of a rational function in $\BQ[q,Q]_{-d}$. On the other hand, by Lemma~\ref{lem: wallcrossinglimite}
\[\lim_{\delta\rightarrow -\infty}\deg\left(\PDT^{\nu, -\infty}_\gamma-\PDT^{\nu,\delta}_\gamma\right)=-\infty\]
and, together with Lemma~\ref{lem: limitdegree}, we have an equality of rational functions
\[\PPT_\gamma=\PDT_{\gamma}^{\nu, -\infty}=\rho_\gamma\big(\PPT_{\gamma}\big)\,.\]
This finishes the proof of Theorem~\ref{thm: perverse}.

\section{Bryan--Steinberg vs.\ perverse PT invariants}
\label{sec: zetawallcrossing}

In this section we will prove the wall-crossing between the Bryan--Steinberg invariants and perverse PT invariants. Together with the $\bs/\PT$ wall-crossing of Section \ref{sec: BS}, the output of this section is a proof of Theorem~\ref{thm: wall-crossing}.

We will use the stability condition $\zeta$ defined in Section~\ref{subsec: refined stability}. The wall-crossing is entirely analogous to~\cite[Section 8]{BCR18}, where Bryan--Steinberg pairs are compared to orbifold PT pairs to prove the crepant resolution conjecture. For us, matters simplify and it is worth to point out how exactly.

The stability~$\zeta$ leads to torsion pairs $(\CT_{\zeta,(\mu,\eta)},\CF_{\zeta,(\mu,\eta)})$ on $\CA_{\leq 1}$ labelled by $(\mu,\eta)\in\BR_{>0}\times\BR$. These are defined analogously to $(\CT_{\nu,\delta},\CF_{\nu,\delta})$ in Section \ref{subsec: rationalityvianu}, by truncating the $\zeta$-HN-filtration. We consider the stack $\Pairs^{\zeta,(\mu,\eta)}$ of $(\CT_{\zeta,(\mu,\eta)},\CF_{\zeta,(\mu,\eta)})$-pairs in
\[ {}^p\CB=\Big\langle \CO_X[1],\CA_{\leq 1}\Big\rangle_\ex\]
in the sense of Definition~\ref{def: pair}. 

\begin{lemma}\label{lem: pairszetaopen}
Let $(\mu,\eta)\in\BR_{>0}\times \BR$ and $(\gamma, c)\in N_{\leq 1}$.
\begin{enumerate}
    \item The stack $\Pairs^{\zeta,(\mu,\eta)}_{(\gamma,c)}\subset\Obj^{\geq -1}({}^p\CB)$ is an open substack of finite type.
    \item The family of objects in $\Pairs^{\zeta,(\mu,\eta)}_\gamma$ is $L_\mu$-bounded.
    \end{enumerate}
\end{lemma}
\begin{proof}
The same strategy of \cite[Proposition 8.16]{BCR18} can be employed to prove the result from Lemmas \ref{lem: Lmubounded} and \ref{lem: pairsnuopen} and Proposition \ref{prop: finitetypezeta}.
\end{proof}

We define numerical invariants \[\PDT^{\zeta,(\mu,\eta)}_{\gamma,c}\in \BQ\]
as we did for pairs defined using $\nu$ in Section \ref{sec: perversept}, see equation \eqref{eq: dtinvariants}.

The notion of $(\mu,\eta)$-pairs is locally constant. More precisely, for fixed $\gamma\in N_1$ there is a finite set of possible walls $V_\gamma$ such that stability is constant on
\[ \big(\BR_{>0}\ssetminus V_\gamma\big)\times \BR\,.\]
The limit $0<\mu\ll 1$ coincides with $\bs$-pairs, the limit $\mu\to+\infty$ coincides with perverse stable pairs. Crossing a wall $\mu\in V_\gamma$ leads to a wall-crossing formula. This wall-crossing is controlled in a concrete way. There is precisely one effective class $0<\gamma'\leq \gamma$ characterized by $L_\mu(A\cdot\gamma')=0$, where as before
\[L_\mu(j,n) = 2n+j+\frac{j}{\mu\,a_0}\,.\]
The asymmetry of $n$ and $j$ in this formula hints at how varying~$\mu$ separates BS from perverse PT (see Example \ref{example: Lmu} below). Recall that $L_\mu$ is the same linear function introduced in Section~\ref{sec: Hall} that controls the expansion of the rational function.

Then, to cross the $\mu$-wall, it is possible to enter the wall from either sides because for $0<\varepsilon\ll 1$ we have
\[ \Pairs^{\zeta,(\mu\pm \varepsilon,\eta)} = \Pairs^{\zeta,(\mu,\pm \infty)}\,.\]

The wall-crossing inside $\{\mu\}\times \BR$ is similar to the $\nu$-wall-crossing in Section~\ref{sec: perversept}. The combinatorics is controlled in the same way.

\begin{example}\label{example: Lmu}
We include an illustration of the wall-crossing for the limit $\mu\to 0^+$. Let $B\subset W$ be a $\BP^1$-fiber of the projection. Since $\chi\big(\CO_B(-1)\big)=0$, the class~$\sfb$ of the ruling is identified with the K-theory class $[\CO_B(-1)]$. The linear function $L_\mu$ specifies which classes in $N_0$ are considered effective. Recall the structure sheaves $k(x)$ of points in~$X$ and the perverse sheaves $\CO_B(-1)$ and $\CO_B(-2)[1]$ in $\CA_0$. Their $K$-theory classes are
\[ \big[k(x)\big] = \pt\,,\quad \big[\CO_B(-1)\big] = \sfb\,,\quad   \big[\CO_B(-2)[1]\big] = \pt - \sfb\,.\]
Both $\pt$ and $\sfb$ satisfy $L_\mu > 0$ for all $\mu>0$, i.e.\ both classes are considered effective at all times. In contrast to that, the class of $\CO_B(-2)[1]$ (considered effective for perverse stable pairs) satisfies 
\begin{align*} L_\mu(\pt-\sfb)&>0\,, \qquad \mu>1\,,\\
L_\mu(\pt-\sfb)&<0\,, \qquad 0<\mu\ll 1\,.
\end{align*}
The limit $\mu\to 0^+$ serves the purpose to exclude all such perverse sheaves (two-term complexes in~$\CA_0$) from being considered effective. 
\begin{center}
\includegraphics[width=\textwidth]{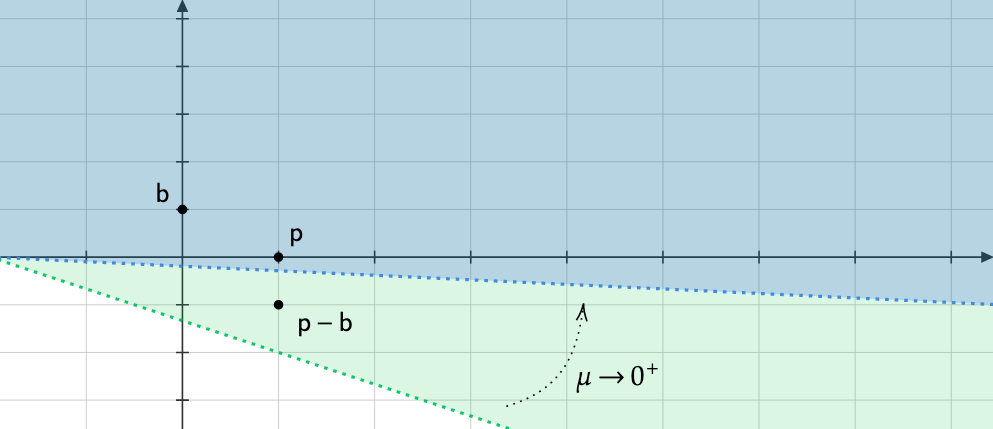}
\end{center}
The picture displays two lines $L_\mu = -2$ for $\mu<0\ll 1$ (blue dotted line) and $\mu \gg 1$ (green dotted line), and their respective areas $L_\mu > -2$.
\end{example}

\subsection{Walls}
Let $\gamma\in N_1$. Define the set of possible walls
\[ V_\gamma = \big\{ \zeta_1(\gamma'): 0<\gamma'\leq\gamma\big\}\cap \BR_{>0}\,.\]

\begin{lemma}
Stability is constant on $\big(\BR_{>0}\ssetminus V_\gamma \big)\times \BR$.
\end{lemma}

In the following, when $\mu\in\big(\BR_{>0}\ssetminus V_\gamma\big)$ we let $\eta\in\BR$ arbitrary. Crossing a wall $\mu\in V_\gamma$ is controlled by the linear function $L_\mu$. The basic reason is the following relation between $L_\mu$ and $\zeta_1$:
\[ L_\mu (A\cdot\gamma) = d(A\cdot\gamma)\,\Big(1-\frac{\zeta_1(\gamma)}{\mu}\Big)\,.\]

\begin{lemma}\label{lem: crosszetawall}
There is, up to scaling, a unique class $\gamma_\mu$ such that $0<\gamma_\mu\leq \gamma$ and $L_\mu(A\cdot\gamma_\mu)=0$. The class $A\cdot \gamma_\mu\in N_0$ is uniquely characterized by this property.
\end{lemma}

\begin{proof}
The proof is a simplified version of~\cite[Lemma 8.21]{BCR18}.\footnote{In \cite{BCR18} the authors choose a very general ample class to define the stability $\zeta$ and function $L_\mu$. This choice is not necessary for our application because $\ch_1(E)\in\BZ\cdot\sfw$ for all $[E]\in N_{\leq 1}$ and $\ch_2(E)\in\BZ\cdot\sfb$ for all $[E]\in N_0$.}
\end{proof}
\begin{example}
We illustrate the previous result for $W\cong \BP^1\times\BP^1$ with projection $p\colon W\to \BP^1$. Let $B$ and $C$ be a fiber resp.\ section of~$p$ and 
\[ \sfb = [B]\,,\quad \sfc = [C]\]
their classes in $N_1$. Consider the class $\gamma=\sfc-\sfb\in N_1$. It is an effective class:
\[\gamma = \big[\CO_W(-2C-B)[1]\big]+\big[\CO_W(-C-2B)\big]\,.\]
The two objects are contained in $\CF[1]$ and $\CT$ respectively and the sum gives rise to the effective decomposition
\[ \gamma = (-\sfw,\sfc)+(\sfw,-\sfb)\,.\]
Recall the line bundle $A$ and $\ell(r,\beta) = 2A\cdot\beta + r\,a_0$. We have
\[ \zeta_1(\CO_W(-2C-B)[1]) = -\frac{-1}{2a_0 - a_0} =\frac{1}{a_0}\] 
and there is only one wall 
\[V_\gamma = \bigg\{ \frac{1}{a_0}\bigg\}\,.\]
The unique class $\gamma_\mu$ is $\big[\CO_W(-2C-B)[1]\big]=(-\sfw,\sfc)$ and 
\[A\cdot\gamma_\mu=(-a_0\,, a_0)\in N_0\,.\]
The linear function $L_\mu$ uniquely specifies $A\cdot\gamma_\mu$ as
\[ L_{\mu'}(A\cdot\gamma_\mu) \begin{cases} >0\,, \quad \mu' >\frac{1}{a_0}\,,\\[5pt]
=0\,, \quad\mu' =\frac{1}{a_0}\,,\\[5pt]
<0\,,\quad\mu' < \frac{1}{a_0}\,. 
\end{cases}
\] 
Correspondingly, the class $A\cdot\gamma_\mu = a_0 \big[\CO_B(-2)[1]\big]\in N_0$ is considered effective in the expansion of the rational function with respect to $L_{\mu'}$ for $\mu'>\frac{1}{a_0}$ ($\PPT$ pairs), whereas it is non-effective for $\mu'<\frac{1}{a_0}$~($\bs$-pairs).
\end{example}
\subsection{Limit stability I}
\label{subsec: limitI}
We identify the limit of $(\mu,\eta)$-stability for $0<\mu\ll 1$ with BS stability. First, we can give an explicit description of the limit of the torsion pair for $0<\mu\ll 1$.
\begin{definition}
We define the torsion pair $(\CT_{\zeta,0}, \CF_{\zeta,0})$ in $\CA_{\leq 1}$ by 
\[\CT_{\zeta,0}=\left\{A\in \CA_{\leq 1}: A\twoheadrightarrow Q\Rightarrow Q\in \CA_0 \textup{ or }\ch_1(Q)\in \BZ_{<0} W\right\}\]
and the orthogonal complement $\CF_{\zeta,0}=\CT_{\zeta,0}^\perp$.
\end{definition}

It's straightforward to see that the pair $(\CT_{\zeta,0}, \CF_{\zeta,0})$ is the limit of $(\CT_{\zeta,(\mu,\eta)},\CF_{\zeta,(\mu,\eta)})$ when $\mu$ becomes very small, in the following precise sense:

\begin{lemma}
Let $P\in {}^p\CB$ of class $(-1,\gamma,c)$ and $0<\mu<\min V_\gamma$. Then, $P$ is a $(\CT_{\zeta,0}, \CF_{\zeta,0})$ pair if and only if $P$ is a $(\CT_{\zeta,(\mu,\eta)},\CF_{\zeta,(\mu,\eta)})$ pair.
\end{lemma}

\begin{lemma}\label{lem: identifytorsionpairs}
We have
\begin{align*}
\CT_{\zeta,0}&=\left\langle \CF[1], \CT_0\right \rangle\\
\CF_{\zeta,0}&=\CT_1.
\end{align*}
\end{lemma}
\begin{proof}
We begin by proving that $\langle \CF[1], \CT_0\rangle \subset T_{\zeta,0}$. We first note that we can write
\[\CA_{\leq 1}=\langle \CF[1], \CT_{\leq 1}\rangle=\big\langle \CF[1], \langle \CT_0, \CT_1\rangle\big\rangle=\big\langle \langle\CF[1],  \CT_0\rangle, \CT_1\big\rangle,\]
so $\langle\CF[1],  \CT_0\rangle$ is closed under quotients. Hence it's enough to show that if $E\in\CF[1]$ or $E\in \CT_0$ then $E\in \CA_0$ or $\ch_1(E)\in \BZ_{<0}\sfw$. For $T\in \CT_0$ this is clear. If $F[1]\in \CF[1]$ then $\ch_1(F[1])=r\sfw$ with $r\leq 0$ and equality if and only if $F\in \Coh_{\leq 1}(X)$. So it remains to show that if $F\in \CF$ and $\ch_1(F)=0$ then $F\in \CF_0$, i.e.\ $\CF\cap \Coh_{\leq 1}(X)=\CF_0$. 

We let $F\in \CF\cap \Coh_{\leq 1}(X)$. Then $\supp(F)\subset W$ is at most 1-dimensional. If there is a fiber $B=p^{-1}(c)$ such that $\supp(F)\cap B$ is 0-dimensional and non-empty then $(R^0p_\ast \iota^\ast F)_c\neq 0$, which would contradict $F\in \CF$. Thus $\supp(F)$ is a finite union of fibers $B$, so $F\in \CF_0$ as we wanted and proving the first inclusion.

For the inclusion $T_{\zeta,0}\subset \langle \CF[1], \CT_0\rangle$, let $E\in T_{\zeta,0}$ and consider the decomposition of $E$ in the torsion pair $\CA_{\leq 1}=\langle\CF[1], \CT_{\leq 1}\rangle$
\[0\to F[1]\to E\to T\to 0.\]
Since $\ch_1(T)\in\BZ_{\geq 0} \sfw$, by the definition of $\CT_{\zeta,0}$ we have $T\in \CT\cap \CA_0=\CT_0$.

This finishes the proof of the first equality $T_{\zeta,0}=\langle \CF[1], \CT_0\rangle$. The second equality follows from the first and
\[\langle \CT_{\zeta,0}, \CF_{\zeta,0}\rangle=\CA_{\leq 1}=\big\langle \langle\CF[1],  \CT_0\rangle, \CT_1\big\rangle.\qedhere\]
\end{proof}

Recall that $\CT_0=\CT\cap \CA_0=\tbs$, so in particular $\tbs\subset \CT_{\zeta,0}$. The key result of this section is
\begin{proposition}
Let $P\in D^b(X)$ be such that $\ch_1(P)=0$. Then $P$ is a $(\CT_{\zeta,0}, \CF_{\zeta,0})$-pair if and only if $P$ is a $(\tbs,\fbs)$-pair. In particular, for any $\beta\in N_1(X)$ and $0<\mu<\min V_{\beta}$ we have
\[\PDT^{\zeta, (\mu,\eta)}_{\beta}=\bs_{\beta}\,.\]

\end{proposition}
\begin{proof}
We begin with the proof that if $P$ is a $(\tbs, \fbs)$-pair then it's a $(\CT_{\zeta,0}, \CF_{\zeta,0})$-pair. If $P$ is a $\bs$-pair, by \cite[Lemma 3.13]{BCR18} we can write $P=(\CO_X\to F)$ with $F\in \fbs$ and $Q=\coker(\CO_X\to F)\in \tbs=\CT_0\subset \CA_{\leq 1}$. We first prove that $F\in \CA_{\leq 1}$, so $P\in {}^p\CB$. If $Z$ is the scheme-theoretical support of $F$ (which is a curve), we have the short exact sequence of sheaves
\[0\to \CO_Z\to F\to Q\to 0.\]
Since both $\CO_Z$ and $Q$ are contained in $\CA_{\leq 1}$, which is closed under extensions, it follows that $F\in \CA_{\leq 1}$. Moreover for $T\in \CT_{\zeta,0}$
\[\Hom(T, P)=\Hom(T, F)=\Hom(H^0(T), F)=0\]
The last vanishing holds because $H^0(T)\in \CT_0=\tbs$ by Lemma \ref{lem: identifytorsionpairs}  and $F\in \fbs$. Similarly, for $G\in \CF_{\zeta,0}$,
\[\Hom(P, G)=\Hom(Q,G)=0\]
vanishes since $Q\in \tbs\subset \CT_{\zeta,0}$. So we conclude that $P$ is a $(\CT_{\zeta,0}, \CF_{\zeta,0})$-pair.

We now assume that $P$ is a $(\CT_{\zeta,0}, \CF_{\zeta,0})$-pair with $\ch_1(P)=0$. Since
\[P\in {}^p\CB=\big\langle \CO_X[1], \CF[1], \CT_{\leq 1}\big\rangle_\ex\]
we can easily see that $\CH^i(P)=0$ for $i\neq -1, 0$ and $\CH^{-1}(P), \CH^0(P)$ have ranks 1 and 0, respectively. Moreover the torsion part $T\hookrightarrow \CH^{-1}(P)$ is in $\CF$, so $T[1]\in \CF[1]\subset \CT_{\zeta,0}$. By definition of $(\CT_{\zeta,0}, \CF_{\zeta,0})$-pair the composition
\[T[1]\hookrightarrow \CH^{-1}(P)[1]\rightarrow P\]
vanishes, forcing $T$ to vanish. Thus $\CH^{-1}(P)$ is torsion-free. By Lemma \ref{lem: identifytorsionpairs} we have
\[\CH^0(P)\in \CT_{\zeta,0}\cap \Coh(X)=\CT_0=\tbs.\]
In particular it follows that
\[\ch_1\big(\CH^{-1}(P)\big)=\ch_1\big(\CH^0(P)\big)-\ch_1(P)=0.\]
Hence $\CH^{-1}(P)$ is a torsion-free, rank 1 sheaf with trivial determinant, hence it's an ideal sheaf $\CH^{-1}(P)\cong I_C$. So $P$ fits in an exact triangle
\[I_C[1]\rightarrow P \rightarrow \CH^0(P).\]
Using the argument of \cite[Lemma 3.11 (ii)]{To10a} with the fact that 
\[H^1\big(X, \CH^0(P)\big)=0\,,\]
we get that $P$ has the form $P=(\CO_X\to F)$. We already know that $\CH^0(P)\in \tbs$ so it remains to show that $F\in \fbs$ (see \cite[Remark 3.10]{BCR18}). For $T\in \tbs$ we have
\[\Hom(T, F)=\Hom(T, P)=0\]
since $T\in \tbs\subset \CT_{\zeta,0}$, and we're done.
\end{proof}
\subsection{Limit stability II}
\label{subsec: limitII}
We identify the limit of $(\mu,\eta)$-stability for $\mu\to\infty$ with $\PPT$ stability.
\begin{lemma}
Let $P\in {}^p\CB$ be of class $(-1,\gamma,c)$ and $\mu>\max V_\gamma$. Then, $P$ is a perverse stable pair if and only if $P$ is a $(\CT_{\zeta,(\mu,\eta)}, \CF_{\zeta,(\mu,\eta)})$ pair. In particular, for any $\gamma\in N_{1}$ and $\mu>\max V_{\gamma}$ we have
\[\PDT^{\zeta, (\mu,\eta)}_{\gamma}=\PPT_{\gamma}\,.\]

\end{lemma}
\begin{proof}
The proof is analogous to \cite[Lemma 8.20]{BCR18}: for such $\mu$ and $E\in \CA_{\leq 1}$ with $[E]\leq \gamma$ in $N_1$, such that $E\in \CT_{\zeta,(\mu, \eta)}$, we must have $E\in \CA_0$.\qedhere
\end{proof}

\subsection{Crossing a wall}
Let $\mu\in V_\gamma$. First, we show that we can enter the wall $\{\mu\}\times\BR$ from either side in the following sense.

\begin{lemma}\label{lem: enterwall}
Let $\alpha\in N_{\leq 1}$ and $0<\varepsilon\ll 1$.
\begin{enumerate}
    \item For sufficiently large $\eta\gg 0$
    \[\Pairs^{\zeta,(\mu,\eta)}_\alpha=\Pairs^{\zeta,(\mu+\varepsilon,\eta)}_\alpha\,,\]
    \item for sufficiently small $\eta\ll 0$
    \[\Pairs^{\zeta,(\mu,\eta)}_\alpha=\Pairs^{\zeta,(\mu-\varepsilon,\eta)}_\alpha\,.\]
\end{enumerate}
\end{lemma}
\begin{proof}
The proof is a simplified version of~\cite[Lemma 8.25]{BCR18}.
\end{proof}

We explain now the wall-crossing inside $\{\mu\}\times \BR$. Let $c_\mu\in N_0$ be the unique class of Lemma~\ref{lem: crosszetawall}. For any $c\in N_0$ define
\[ \PDT^{\zeta,(\mu,\eta)}_{\gamma,c+\BZ c_\mu} = \sum_{k\in\BZ} \PDT^{\zeta, (\mu,\eta)}_{\gamma, c+kc_\mu}z^{c+kc_\mu}\in \BQ[[Q^{\pm 1}, q^{\pm 1}]]\,.\]
We have used the Novikov parameter~$z$ to track both $q$ and $Q$. 
By the previous lemma, the notion of $(\mu,\eta)$-pair is constant for $\mu\gg 0$ (respectively $\mu\ll 0$) and fixed $\alpha\in N_{\leq 1}$. Thus, we can define the limit for $\eta\to\pm\infty$, which agrees with the generating series for $(\mu\pm\varepsilon,\eta)$:
\[\PDT^{\zeta,(\mu,\pm\infty)}_{\gamma,c+\BZ c_\mu} = \PDT^{\zeta, (\mu\pm\varepsilon,\eta)}_{\gamma,c+\BZ c_\mu}\,.\]
\begin{lemma}
The two generating series $\PDT^{\zeta, (\mu,\pm \infty)}_{\gamma,c+\BZ c_\mu}$ are the expansion of the same rational function. 
\end{lemma}
\begin{proof}
The combinatorics is the same as in Section~\ref{subsec: combinatorics}, see also \cite[Corollary 8.28]{BCR18}. 

The technical conditions to apply the wall-crossing formula are verified using Proposition \ref{prop: finitetypezeta} and Lemma \ref{lem: pairszetaopen} in essentially the same way as we did in the proof of Lemma \ref{lem: wallcrossingmaterial}. For condition (iii) of Section \ref{subsec: wallcrossing} we note that if $\sum_{i=1}^n c_i=c$ is fixed and each $c_i$ belongs to a $L_\mu$-bounded set, then there are only finitely many possibilities for each $c_i$. 
\end{proof}

The main result of this section is then a formal consequence.
\begin{proposition}
There exists a rational function $f_\gamma(q,Q)$ such that for all $\mu\in V_\gamma$ the series $\PDT^{\zeta, (\mu\pm\varepsilon,\eta)}_\gamma$ are the expansion of $f_\gamma$ with respect to $L_{\mu\pm\varepsilon}$. 
\end{proposition}
\begin{proof}
Let $\mu=\max V_\gamma$ be the biggest wall and $c_\mu\in N_0$ the class of Lemma~\ref{lem: crosszetawall}. 
By Lemma~\ref{lem: enterwall} and Section~\ref{subsec: limitII} the series $\PDT^{\zeta, (\mu+\varepsilon,\eta)}_\gamma$ agrees with perverse stable pairs $\PPT_\gamma$ and it is the expansion of a rational function $f^\mu_\gamma$ as proven in Section~\ref{sec: perversept}. Note that in the limit $\mu'\to\infty$ the linear function
\[ L_{\mu'}(c) = d(c) + \frac{j}{\mu'\,(a_0)}\]
agrees with $d(-)$ in the sense that expansion of the rational function $f^\mu_\gamma$ is the same for $L_{\mu'}$ and $d$.

The previous lemma says that the two series $\PDT^{\zeta, (\mu,\pm \infty)}_{\gamma,c+\BZ c_\mu}$ agree as rational function for each $c\in N_0$. Their difference is a quasi-polynomial function in~$k$. Recall that, by definition of $c_\mu$, we have
\[L_{\mu+\varepsilon}(c_\mu)>0\,,\quad L_{\mu-\varepsilon}(c_\mu)<0\,.\]
It is then a formal consequence~\cite[Lemma 2.22]{BCR18} that $\PDT^{\zeta, (\mu-\varepsilon,\eta)}_\gamma$ is the expansion of the same rational function~$f^\mu_\gamma$, with respect to $L_{\mu-\varepsilon}$. 

Since stability is constant on $\big(\BR_{>0}\ssetminus V_\gamma\big)\times \BR$ we can argue by induction on the finite set of walls $\mu'\in V_\gamma$. In particular, we obtain the same rational function $f_\gamma$ for each wall.
\end{proof}

The limit of $\zeta$-stability for $0<\mu\ll 1$ was found to agree with BS stability in Section~\ref{subsec: limitI} which, together with Section~\ref{sec: BS}, concludes the proof of Theorem~\ref{thm: wall-crossing}.
\section{Gromov--Witten theory}
\label{sec: applicationGW}

In this section we assume the GW/PT correspondence for~$X$. Let
\[R=\BC\left[Q^{\pm 1}, \left(\frac{1}{1-Q^j}\right)_{j\geq 1}\right][u^{-1}, u]]\]
and
\[R_a=\big\{f\in R: f(Q, u)=Q^a f(Q^{-1}, -u)\big\}.\]
More explicitly, elements of $R$ are written as
\[f(u,Q)=\sum_{h\geq H}f_h(Q)u^h\]
where $f_h(Q)$ are rational functions of the form
\[f_h(Q)=\frac{p(Q)}{\prod_{j}(1-Q^{a_j})}\]
with $p(Q)$ a Laurent polynomial. Then $f\in R_a$ if and only if \[Q^af_h(Q^{-1})=(-1)^h f_h(Q)\,.\]

\begin{proposition}\label{prop: feperversegw}
Let $\beta\in H_2(X, \BZ)$. After the change of variables $q=e^{iu}$, one has
\[\PPT_\beta(e^{iu},Q)\in R_{\sfw\cdot \beta}\,.\]
\end{proposition}
\begin{proof}
We prove first that $\PPT_\beta\in R$. By Theorem~\ref{thm: main} it holds that 
\[\PPT_\beta\in \BQ\left[q^{\pm 1}, Q^{\pm 1}, \left(\frac{1}{1-q^a Q^b}\right)_{a,b\geq 0}\right].\]
Since clearly $q^{\pm 1}, Q^{\pm 1}\in R$ it suffices to show that $\frac{1}{1-q^a Q^b}\in R$, which follows from the following simple computation:
\begin{align*}
    \frac{1}{1-e^{iau}Q^b}&=\sum_{k\geq 0} e^{ikau}Q^{kb}=\sum_{k,s\geq 0} u^s\frac{(ia)^s}{s!} k^s Q^{kb}\\
    &=\sum_{s\geq 0} u^s\frac{(ia)^s}{s!}\mathrm{Li}_{-s}(Q^b).
\end{align*}
Since the polylogarithm $\mathrm{Li}_{-s}(Q)$ is a rational function with denominator $(1-Q)^{s+1}$  for $s\geq 0$, the claim follows.

The rest of the Proposition follows from the functional equation part of Theorem~\ref{thm: main}. We have $\rho_\beta(j\sfb, n)=((-j+\sfw\cdot \beta) \sfb, -n)$ by Proposition~\ref{prop: actionCoh}, so
\[Q^{\sfw\cdot \beta}\left(\PPT_\beta(q^{-1},Q^{-1})\right)=\PPT_\beta(q,Q)\,.\]
After the change of variables $q=e^{iu}$, it follows that $\PPT_\beta\in R_{\sfw\cdot \beta}.$
\end{proof}

\begin{conjecture}
The Proposition above still holds if we replace $R$ by the smaller ring
\[R=\BC\left[Q^{\pm 1}, \frac{1}{1-Q}\right][u^{-1}, u]].\]
\end{conjecture}

We now deal with the exceptional part. This requires that we exclude genus 0 and 1 terms. More precisely, define
\[\widetilde \PT_0(q,Q)=\PT_0(q,Q)\cdot \exp\left(\frac{2}{u^2}\Li_3(Q)+\frac{1}{6}\Li_1(Q)\right).\]

\begin{proposition}\label{prop: feexceptionalgw}
After the change of variables $q=e^{iu}$, one has
\[\widetilde \PT_0(e^{iu},Q)\in R_{0}.\]
\end{proposition}
\begin{proof}
We have
\[\PT_0(q,Q)=\exp\left(\sum_{k\geq 1}\frac{2(qQ)^k}{k(1-q^k)^2}\right).\]
Writing $c_h$ for the coefficients in the $u$-expansion of 
\[\frac{2e^{iu}}{(1-e^{iu})^2}=\sum_{h\geq -2}c_h u^h\]
one has the formula
\[\PT_0(q,Q)=\exp\left(\sum_{h\geq -2} c_h u^h\Li_{1-h}(Q)
\right).\]
As easy inspection shows that $c_{-2}=-2, c_{-1}=0=c_1$ and $c_0=-1/6$. Thus, the definition of $\widetilde \PT_0$ removes the first terms in the previous formula and we find that 
\[\widetilde \PT_0(q,Q)=\exp\left(\sum_{h\geq 2} c_h u^h\Li_{1-h}(Q)
\right).\]
This concludes the proof since, for $h\geq 2$, $\Li_{1-h}(Q)$ is a rational function with denominator $(1-Q)^{h}$ and satisfies the symmetry property
\[\Li_{1-h}(Q^{-1})=(-1)^h\Li_{1-h}(Q).\qedhere\]
\end{proof}

We provide now the proof to Corollary~\ref{cor: rationality}. We denote 
\[\widetilde \PT_{\beta}(q,Q)=\PT_{\beta}(q,Q)\cdot \exp\left(\frac{2}{u^2}\Li_3(Q)+\frac{1}{6}\Li_1(Q)\right).\]

By Theorem~\ref{thm: wall-crossing},
\[{}^p\PT_{\beta}(q,Q)\widetilde \PT_0(q,Q) = \widetilde\PT_\beta(q,Q)\]
so Propositions \ref{prop: feperversegw} and \ref{prop: feexceptionalgw} together imply that $\widetilde\PT_\beta(q,Q)\in R_{\sfw\cdot \beta}$. Hence the generating function
\[\sum_{\beta\in N_1} \widetilde\PT_\beta(q,Q)z^\beta\]
belongs to the ring
\[\CR=\left\{\sum_{\beta\in N_1}f_\beta(q,Q)z^\beta: f_\beta\in R_{\sfw\cdot \beta}\right\}.\]

Moreover, with the usual change of variable $q=e^{iu}$, we have
\[\exp\left( \sum_{(g,\beta)\neq (0,0), (1,0)}u^{2g-2}z^\beta \sum_{j\in \BZ} \GW_{g, \beta+j\sfb}Q^j\right)=\sum_{\beta\in N_1} \widetilde\PT_\beta(q,Q)z^\beta\in \CR.\]
Since $\CR$ is a ring, taking the logarithm preserves $\CR$, finishing the proof of Corollary~\ref{cor: rationality}.

\appendix

\section{Local Hirzebruch surface}
\label{appendix}

In this appendix we take a closer look at the case of local Hirzebruch surfaces~$K_W$ and we use the topological vertex to compute their enumerative invariants. In particular, we prove the following strengthening of Corollary~\ref{cor: rationality} in the local case:

\begin{theorem}
\label{thm: localrationality} Let $X=K_W$ be a local Hirzebruch surface. For all $g\in \BZ_{\geq 0}$ and $\beta\in H_2(W)$ such that $(g,\beta)\neq (0,m\sfb)\,,(1,m\sfb)$, the series
\[ \sum_{j\geq 0} \GW_{g,\beta+j\sfb}\, Q^j \]
is the expansion of a rational function $f_\beta(Q)$ of the form
\[f_\beta(Q)=\frac{p_\beta(Q)}{(1-Q)^{4(\sfb\cdot \beta)+2g-2}}\]
where $p_\beta$ is a Laurent polynomial. Moreover $f_\beta$ satisfies the functional equation
\[ f_\beta(Q^{-1}) = Q^{-K_W\cdot \beta} f_\beta(Q)\,.\]
\end{theorem}

In the theorem, the intersection products $\sfb \cdot \beta$ and $K_W\cdot \beta$ are taken in $H^\ast(W)$. The canonical class is 
\[K_W=-2\sfc-(r+2)\sfb\,\]
where $\sfc$ is the class of the torus-invariant section with non-positive self-intersection $\sfc^2=-r$.

\begin{remark}
The form of the rational function implies that if we fix $k, g, r$ then $\GW_{g,m\sfc+j\sfb}^{K_{\mathbb{F}_r}}$ is a polynomial in $j$ of degree $4m+2g-3$ for large enough $j$. In \cite[Equation 5.2]{KKV97} the authors predict the assymptotic behavior for $g=0$
\[\GW_{g=0,m\sfc+j\sfb}^{K_{\mathbb{F}_r}}\sim \gamma_{m}j^{4m-3}\]
for some constant $\gamma_m$ that doesn't depend on $r$. The independence of $r$ is not difficult to see from our proof.
\end{remark}

\subsection{Combinatorics of the 2-leg topological vertex}

The local Hirzebruch surface~$K_W$ is a toric non-compact Calabi--Yau 3-fold, so its Pandharipande--Thomas invariants can be computed via the formalism of the topological vertex. The 2-leg case of the topological vertex admits simple combinatorical expressions, also known as Iqbal's formula \cite{Iqbal02,LLZ07,YZ10,Zhou03}. We now describe such formula.

Given a partition $\mu$, we associate to it the Schur function $s_{\mu}(x_1, \ldots, x_n)$ (see for example \cite[I.3]{Ma95}). An explicit way to define $s_\mu$ is the following:
\[s_{\mu}=\det\left(h_{\mu_i-i+j}\right)_{1\leq i,j\leq N}\]
where $N\geq \ell(\mu)$ and $h_k=h_k(x_1, \ldots, x_n)$ are the complete homogeneous polynomials. We will often consider the specialization of $s_\mu$ to the infinite set of variables $x=(1, q,q^2,\ldots)$. In this case the definition above is still valid with
\[h_k(1,q,q^2, \ldots)=\prod_{j=1}^k\frac{1}{1-q^j} \textup{ for } k\geq 0\]
(if $k<0$ then $h_k=0$). An alternative way to write $s_\mu(1, q,Q^2, \ldots)$ is the hook-content product formula
\[s_\mu(1,q,q^2,\ldots)=q^{n(\mu)}\prod_{\square\in \mu}\frac{1}{1-q^{h(\square)}}.\]
In the formula above $n(\mu)$ is $\sum_{i=1}^{\ell(\mu)}(i-1)\mu_i$, the product runs over boxes in the Young diagram of $\mu$ and $h(\square)$ is the hook length of a square $\mu$.

Iqbal introduced $\CW$ functions for 1 and 2 partitions that play a role in the 1-leg and 2-leg vertex formulas, respectively. For one partition $\mu$, it's defined as
\[\CW_\mu(q)=(-1)^{|\mu|}q^{k(\mu)/2+|\mu|/2}s_\mu(1, q,Q^2, \ldots ).\]
Here
\[k(\mu)=\sum_{i=1}^{\ell(\mu)}\mu_i(\mu_i-2i+1)\in \BZ.\]
For two partitions $\mu, \nu$ we define
\[\CW_{\mu\nu}(q)=q^{|\nu|/2}\CW_\mu(q)s_\nu(q^{\mu_1-1}, q^{\mu_2-2}, \ldots)\,.\]
Althought it's not aparent from this definition, we have symmetry in the two partitions, i.e. $\CW_{\mu\nu}=\CW_{\nu\mu}$ \cite[Theorem 5.1]{Zhou03a}.

We can now formulate Iqbal's formula for the Gromov-Witten invariants of local toric surfaces. 

Let $W$ be a toric surface. Let $D_1, D_2, \ldots, D_N, D_{N+1}=D_1$ be the toric divisors in the order they appear in the moment polygon of $W$. Denote $s_j=D_j^2\in \BZ$ the self-intersection numbers.

\begin{theorem}[Theorem 1 in \cite{YZ10}]
\label{thm: iqbal}
The partition function for the disconnected Gromov-Witten invariants of $K_W$ is
\[Z^{K_W}=\sum_{\mu_1, \ldots, \mu_N} \prod_{j=1}^N\left((-1)^{s_j|\mu_j|} q^{k(\mu_j)s_j/2}  \CW_{\mu_j, \mu_{j+1}}(q) z^{|\mu_j|D_j}\right)\]
after the change of variables $q=e^{iu}$.
\end{theorem}

Recall that under the change of variables $q=e^{iu}$ we have
\[Z^{K_W}=\PT^{K_W}(q, z)=\sum_{n,\beta}P_{\beta, n}(-q)^n z^\beta.\]

\subsection{Iqbal's formula for Hirzebruch surfaces}

We specialize Theorem \ref{thm: iqbal} to the case of the Hirzebruch surface $W=\BF_r$. The homology $H_2(W, \BZ)$ is generated by two classes $\sfb, \sfc$ where $\sfb$ is the fiber class and $\sfc$ is the class of the torus-invariant section $\BP^1\hookrightarrow W$ with non-positive self-intersection $\sfc^2=-r$. The toric divisors of $W$ are
\[D_1=\sfb=D_3\,,\quad D_2=\sfc+r\sfb\,,\quad D_4=\sfc\,.\]
We denote by $Q=z^\sfb$ and $Q_{\sfc}=z^\sfc$ the Novikov variables relative to $\sfb$ and $\sfc$, respectively. 

\begin{align}
\label{eq: iqbalhirzebruch}
Z^{K_{W}}&=\sum_{\mu_1, \ldots, \mu_4}\big(q^{r(k(\mu_2)-k(\mu_4))}\CW_{\mu_1\mu_2}\CW_{\mu_2\mu_3}\CW_{\mu_3\mu_4}\CW_{\mu_4\mu_1}\\
&\quad \quad \quad\quad \quad \times ((-1)^r Q_{\sfc})^{|\mu_2|+|\mu_4|}Q^{|\mu_1|+|\mu_3|+r|\mu_2|}\big)\nonumber\\
&=\sum_{m=0}^\infty Q_{\sfc}^j(-1)^{rm}\sum_{|\mu_2|+|\mu_4|=m}\Bigg(q^{r(k(\mu_2)-k(\mu_4))}Q^{r|\mu_2|}\nonumber\\
&\quad \quad \quad\quad \quad \times\bigg(\sum_{\lambda} \CW_{\mu_2\lambda}\CW_{\mu_4\lambda}Q^{|\lambda|}\bigg)^2
\Bigg)\nonumber
\end{align}

The sum appearing in the last line
\[S_{\mu \nu}(q,Q)=\sum_{\lambda} \CW_{\mu\lambda}(q)\CW_{\nu\lambda}(q)Q^{|\lambda|}\in \BQ((q,Q))\]
admits a nice closed formula~\cite[Proposition 1]{EK03}. We give a proof which is a bit more direct than the one found in~\cite{EK03}. Let
\[p_\mu(q)=\sum_{i=1}^\infty q^{\mu_i-i}=\frac{q^{-\ell(\mu)}}{q-1}+\sum_{i=1}^{\ell(\mu)} q^{\mu_i-i}.\]

\begin{lemma}[Proposition 1 in \cite{EK03}]
For any two partitions $\mu, \nu$ we have the following identity in $\BQ((q,Q))$:
\begin{equation}S_{\mu\nu}=\mathcal W_\mu\mathcal W_\nu\exp\left(\sum_{k=1}^\infty p_\mu(q^k)p_\nu(q^k)    \frac{(qQ)^k}{k}\right).\label{eq: Smunu}\end{equation}
\end{lemma}
\begin{proof}
Let $x_i=(qQ)^{1/2}q^{\mu_i-i}$, $y_j=(qQ)^{1/2}q^{\nu_j-j}$. Then 
\[p_\mu(q^k)=(qQ)^{-k/2}\sum_{i\geq 1} x_i^k=(qQ_F)^{-k/2}P_k(x)\,,\]
where $P_k(x)$ is the $k$-th power function. For a partition $\lambda$ let \[P_\lambda(x) = \prod P_{\lambda_i}(x)\,,\quad m_k=\#\{i:\lambda_i=k\}\,,\quad z_\lambda=\prod k^{m_k}m_k!\,.\]
By expanding the exponential and cancelling $\mathcal W_\mu \mathcal W_\nu$ on both sides, using
\[\mathcal W_{\mu\lambda}= q^{|\lambda|/2}\mathcal W_\mu s_{\lambda}(q^{\mu_1-1},q^{\mu_2-2},\ldots)\,, \]
we're left to show
\begin{align*} &\sum_{\lambda} (qQ)^{|\lambda|}s_{\lambda}(q^{\mu_1-1},q^{\mu_2-2},\ldots)s_{\lambda}(q^{\nu_1-1},q^{\nu_2-2},\ldots)\\
&= \sum_{\lambda} \prod_{k=1}^{\ell(\lambda)} \frac{1}{m_k!}\left(\frac{P_{k}(x)P_{k}(y)}{k}\right)^{m_k}.\end{align*}

By the Cauchy identity \cite[Eq.\ 4.3]{Ma95} the LHS is
\[\prod_{i,j\geq 1}\frac{1}{1-x_iy_j}\,,\]
and the RHS is
\[\sum_{\lambda} z_\lambda^{-1} P_\lambda(x)P_\lambda(y)\,.\]
The two sides agree \cite[Eq.\ 4.1, 4.3]{Ma95}.\qedhere
\end{proof}

\subsection{Rationality of $\PT_\beta/\PT_0$}
We give now a quick proof of our main rationality statement \ref{thm: main} in the local case based on our computations. Equation \eqref{eq: Smunu} can also be written as an infinite product formula in the following way. If we write
\[p_\mu(q)p_\nu(q)=\frac{\sum_{i=-s}^s a_i q^i}{(1-q)^2}\]
then
\[S_{\mu\nu}=\CW_\mu\CW_\nu\prod_{i=-s}^s\left(\prod_{j\geq 1}(1-q^{j+i}Q)^{-j}\right)^{a_{i}}.\]
Note in particular that taking the constant $Q_{\sfc}^0$ coefficient in equation \eqref{eq: iqbalhirzebruch} we find
\[\PT_0(q,Q)=[Q_{\sfc}^0]Z^{K_W}=S_{\emptyset \emptyset}^2=\prod_{j\geq 1}(1-q^j Q)^{-2j}.\]
Since $\CW_\mu, \CW_\nu\in \BQ(q)$ and
\[\sum_{i=-s}^s a_i=1\, \textup{ and }\, \sum_{i=-s}^s ia_i=0\]
one can see that 
\[\frac{S_{\mu\nu}}{S_{\emptyset\emptyset}}\in \BQ(q,Q).\]
Together with \eqref{eq: iqbalhirzebruch} it follows that
\[\frac{\PT_{m\sfc}(q,Q)}{\PT_0(q,Q)}=[Q_{\sfc}^m]\frac{Z^{K_W}}{S_{\emptyset \emptyset}^2}\in \BQ(q,Q).\]

\subsection{Proof of Theorem \ref{thm: localrationality}}

We give the proof of Theorem  \ref{thm: localrationality} based on the application of Iqbal's formula \eqref{eq: iqbalhirzebruch}. We first remark that it's enough to prove the result when $\beta=j\sfc$ for some $j\geq 0$. Indeed, if $\beta'=\beta+k\sfb$ then the corresponding generating functions are related by multiplication by $Q^{-k}$ and 
\[\sfb\cdot \beta'=\sfb\cdot \beta\,, \quad -K_W\cdot \beta'=-K_W\cdot \beta+2k.\]

We define a refinement $R_{a,b}\subset R_a$ of the sets introduced in Section~\ref{sec: applicationGW}. Elements of $R_{a,b}$ are Laurent series of the form
\[f(Q,u)=\sum_{h\geq H}f_h(Q)u^h\]
such that $f_h(Q)$ take the form
\[f_h(Q)=\frac{p(Q)}{(1-Q)^{b+h}}\]
and satisfy
\[Q^a f_h(Q^{-1})=(-1)^h f_h(Q).\]
For a Laurent series $f(q,Q)$ in variables $q,Q$ we say that $f\in R_{a,b}$ if it's in $R_{a,b}$ after the change of variables $q=e^{iu}$. We're required to show that
\[\widetilde \PT_{j\sfc}(q,Q)\in R_{2j, (r+2)j}\]
(see Section \ref{sec: applicationGW} for the definition of $\widetilde \PT$).

We consider the $u$-expansion of the series
\[p_\mu(e^{iu})p_{\nu}(e^{iu})e^{iu}=\sum_{n=-2}^\infty c_n^{\mu\nu}u^n.\]
The first few terms of the expansion are easily computed:
\[p_\mu(e^{iu})p_{\nu}(e^{iu})e^{iu}=-u^{-2}+\left(|\mu|+|\nu|-\frac{1}{12}\right)+\frac{iu}{2}\big(k(\mu)+k(\nu)\big)+O(u^2).\]
Plugging the expansion into equation \eqref{eq: Smunu} we get
\[S_{\mu\nu}=\CW_\mu\CW_\nu\exp\left(\sum_{n=-2}^\infty c_n^{\mu\nu}u^n\Li_{1-n}(Q)\right).\]
Defining now the modification
\[\widetilde S_{\mu\nu}=S_{\mu\nu}\exp\left(\frac{1}{u^2}\Li_3(Q)+\frac{1}{12}\Li_1(Q)\right)\]
we have the formula
\begin{align*}\widetilde S_{\mu\nu}=(1-Q)&^{-2(|\mu|+|\nu|)}\widetilde \CW_\mu\widetilde \CW_\nu\exp\left(\frac{iu}{4}\big(k(\mu)+k(\nu)\big)\frac{1+Q}{1-Q}\right)\\
&\times\exp\left(\sum_{n=2}^\infty c_n^{\mu\nu}u^n\Li_{1-n}(Q)\right).\end{align*}
where
\[\widetilde \CW_\mu(q)=q^{-\frac{k(\mu)}{4}}\CW_\mu(q)=\exp\left(\frac{iu}{4}k(\mu)\right)\CW_\mu(q).\]
We used the identities
\[\Li_1(Q)=-\log(1-Q) \textup{ and }\Li_0(Q)=\frac{Q}{1-Q}.\]
For $n\geq 2$, $\Li_{1-n}(Q)$ is a rational function with denominator $(1-Q)^{n}$ and satisfies the symmetry property
\[\Li_{1-n}(Q^{-1})=(-1)^n\Li_{1-n}(Q)\,.\]
Moreover, $\widetilde\CW$ satisfies $\widetilde \CW(q)=\widetilde \CW(1/q)$ (see \cite[Proposition 5.1]{Zhou03a}) so we have, for $m=|\mu|+|\nu|$,
\[\widetilde S_{\nu}\in R_{2m,2m}\,.\]
We can now easily finish the proof of Theorem~\ref{thm: localrationality}. From \eqref{eq: iqbalhirzebruch} we have
\[\widetilde \PT_{m\sfc}(q,Q)=(-1)^{rm}\sum_{|\mu_2|+|\mu_4|=m}\Bigg(q^{r(k(\mu_2)-k(\mu_4))}Q^{r|\mu_2|} \widetilde S_{\mu_2\mu_4}^2
\Bigg).\]
We pair the $(\mu_2, \mu_4)$ and $(\mu_4, \mu_2)$ terms and note that
\[q^{r(k(\mu_2)-k(\mu_4))}Q^{r|\mu_2|}+q^{r(k(\mu_4)-k(\mu_2))}Q^{r|\mu_4|}\in R_{0, rm}\]
so 
\[\PT_{m\sfc}(q,Q)\in R_{2m, (r+2)m}\,.\]

\bibliographystyle{mybstfile.bst}
\bibliography{refs.bib}
\end{document}